\documentclass[11pt, reqno]{amsart}
\usepackage{graphicx,amssymb,amsmath,amsthm}
\usepackage{enumerate}
\usepackage{dsfont}
\usepackage[colorlinks, citecolor=red]{hyperref}
\usepackage{comment,cite,color}
\usepackage{cite,color}

\usepackage{mathrsfs}
\usepackage{epsfig}
\usepackage{lscape}
\usepackage{subfigure}
\usepackage[outdir=./]{epstopdf}

\usepackage{caption}
\usepackage{bm}
\usepackage{algorithm}
\usepackage{algpseudocode}
\usepackage[english]{babel}

\usepackage{multirow}

\usepackage{fullpage}
\usepackage[margin=1in]{geometry}
\newcommand{\xkh}[1]{\left(#1\right)}
\newcommand{\dkh}[1]{\left\{#1\right\}}

\newcommand{\nj}[1]{\langle {#1} \rangle}

\newcommand{\norm}[1]{\|{#1}\|}

\newcommand{\norms}[1]{\|{#1}\|}
\newcommand{\abs}[1]{\lvert#1\rvert}

\newcommand{\R}{{\mathbb R}}
\newcommand{\Rn}{{\mathbb R}^n}

\newcommand{\T}{\top}

\newcommand{\vx}{{\bm x}}

\newcommand{\vz}{{\bm z}}

\newcommand{\va}{{\bm a}}

\newcommand{\RNum}[1]{\uppercase\expandafter{\romannumeral #1\relax}}

\numberwithin{equation}{section}

\newtheorem{definition}{Definition}[section]

\newtheorem{prop}[definition]{Proposition}
\newtheorem{theorem}[definition]{Theorem}

\newtheorem{remark}[definition]{Remark}

\newtheorem{example}[definition]{Example}

\newtheorem{thm}{Theorem}[section]
\newtheorem{lem}{Lemma}[section]
\theoremstyle{remark}
\newtheorem{rem}{Remark}[section]
\newtheorem*{rem*}{Remark}

\date{}

\begin{document}

\author{Jian-Feng Cai}
\thanks{J. F. Cai was supported in part by Hong Kong Research Grant Council grants 16309518, 16309219, 16310620, 16306821.}
\address{Department of Mathematics, The Hong Kong University of Science and Technology, Clear Water Bay, Kowloon, Hong Kong, China}
\email{jfcai@ust.hk}

\author{Meng Huang}
\address{Department of Mathematics, The Hong Kong University of Science and Technology, Clear Water Bay, Kowloon, Hong Kong, China}
\email{menghuang@ust.hk}

\author{Dong Li}
\address{SUSTech International Center for Mathematics and Department of Mathematics, Southern University of Science and Technology, Shenzhen, China}
\email{lid@sustech.edu.cn}

\author{Yang Wang}
\thanks{Y. Wang was supported in part by the Hong Kong Research Grant Council grants 16306415 and 16308518.}
\address{Department of Mathematics, The Hong Kong University of Science and Technology, Clear Water Bay, Kowloon, Hong Kong, China}
\email{yangwang@ust.hk}

\baselineskip 18pt
\bibliographystyle{plain}
\title[Smoothed amplitude model]{The global landscape of phase retrieval \RNum{1}:  perturbed  amplitude models}
\maketitle

\begin{abstract}
A fundamental task in phase retrieval is to  recover an unknown signal $\vx\in \Rn$ from a set of 
magnitude-only measurements $y_i=\abs{\nj{\va_i,\vx}}, \; i=1,\ldots,m$. 
In this paper, we propose two novel perturbed amplitude models (PAMs) which have non-convex
and quadratic-type loss function. When the measurements $ \va_i \in \Rn$ are Gaussian random vectors and the number of measurements $m\ge Cn$, we rigorously prove that the PAMs admit no spurious local minimizers with high probability, i.e., the target solution $ \vx$ is the unique global minimizer (up to a global phase) and the loss function has a negative directional curvature around each saddle point.  Thanks to the well-tamed benign geometric landscape, one can employ the vanilla gradient descent method to locate the global minimizer $\vx$ (up to a global phase) without spectral initialization. We carry out extensive numerical experiments to show that the gradient descent algorithm with random initialization outperforms  state-of-the-art algorithms with spectral initialization  in empirical success rate and convergence speed.
\end{abstract}

\section{Introduction}
\subsection{Background}
The basic amplitude model for phase retrieval can be written as 
\[
y_j=\abs{\nj{\va_j, \vx}}, \qquad j=1,\ldots,m,
\]
where $ \va_j\in \Rn, j=1,\ldots,m$ are given vectors and $m$ is the number of measurements. The goal is to recover the unknown signal $\vx\in \Rn$
based on the measurements $\dkh{(\va_j,y_j)}_{j=1}^m$. 
This problem arises in many fields of science and engineering such as X-ray crystallography \cite{harrison1993phase,millane1990phase}, microscopy
\cite{miao2008extending}, astronomy \cite{fienup1987phase}, coherent diffractive
imaging \cite{shechtman2015phase,gerchberg1972practical} and optics
\cite{walther1963question} etc.  In practical applications
due to the physical limitations  optical detectors  can only record the magnitude of signals while losing the phase information. Despite its simple mathematical formulation, it has been
shown that reconstructing a finite-dimensional discrete signal from the magnitude of its Fourier transform is generally an {\em NP-complete} problem \cite{Sahinoglou}.

Many algorithms have been designed to solve the phase retrieval problem, which can be categorized into  convex algorithms and non-convex ones.  The convex algorithms usually rely on a ``matrix-lifting'' technique, which lifts the phase retrieval problem into a low rank matrix recovery problem. By using convex relaxation one can recast the matrix recovery problem as a convex optimization problem. The corresponding algorithms include PhaseLift \cite{phaselift,Phaseliftn}, PhaseCut \cite{Waldspurger2015} etc. 
It has been shown \cite{Phaseliftn} that PhaseLift can achieve the exact recovery under the optimal sampling complexity with Gaussian random measurements.

Although convex methods have good theoretical guarantees of convergence, they tend to be computationally inefficient for large scale problems. In contrast, many non-convex algorithms bypass the lifting step and operate directly on the lower-dimensional ambient space, making them much more computationally efficient. Early non-convex algorithms were mostly based on the technique of alternating projections, e.g. Gerchberg-Saxton \cite{Gerchberg1972} and Fineup \cite{ER3}. The main drawback, however, is the lack of theoretical guarantee. Later Netrapalli et al \cite{AltMin} proposed the AltMinPhase  algorithm based on a technique known as {\em spectral initialization}. They proved that the algorithm linearly converges to the true solution with $O(n \log^3 n)$ resampling Gaussian random measurements. This  work led 
further to several other non-convex algorithms based on spectral initialization. 
A common thread  is 
 first choosing a good initial guess through spectral initialization, and then solving an optimization model through gradient descent. Two widely used optimization estimators are the intensity-based loss
\begin{equation}\label{eq:mod1}
\min_{\vz\in \R^n}\,\,F(\vz)=\sum_{j=1}^m \xkh{\abs{\nj{\va_j,\vz}}^2-y_j^2}^2;
\end{equation}
and the amplitude-based loss
\begin{equation}\label{eq:mod2}
\min_{\vz\in \R^n}\,\,F(\vz)=\sum_{j=1}^m \xkh{\abs{\nj{\va_j,\vz}}-y_j}^2.
\end{equation}
Specifically,  Cand\`es et al developed the Wirtinger Flow (WF)   method  \cite{WF}  based on  (\ref{eq:mod1}) and proved that the WF algorithm can achieve linear convergence with $O(n \log n)$ Gaussian random measurements.   Chen and Cand\`es  
in \cite{TWF} improved the results to $O(n)$ Gaussian random measurements by incorporating a truncation which leads to a novel Truncated Wirtinger Flow (TWF) algorithm. Other methods based on (\ref{eq:mod1}) include the Gauss-Newton method \cite{Gaoxu}, the trust-region 
method \cite{Sun18} and the like \cite{huang2019dc}.  For the amplitude flow estimator
\eqref{eq:mod2}, several algorithms have also been developed recently, such as the Truncated Amplitude Flow (TAF) algorithm \cite{TAF}, the Reshaped Wirtinger Flow (RWF) \cite{RWF} algorithm, randomized Kaczmarz methods \cite{tan2019phase,huang2021linear,jeong2017,Wei2015} and the Perturbed Amplitude Flow (PAF) \cite{PAF} algorithm. Those algorithms  have been shown to  converge linearly to the true solution up to a global phase with $O(n)$ Gaussian random measurements. Furthermore,  there is ample evidence from
 numerical simulations showing  that algorithms based on the amplitude flow loss (\ref{eq:mod2}) tend to outperform algorithms based on loss (\ref{eq:mod1}) when measured in empirical
success rate and convergence speed.

\subsection{Prior arts and connections} \label{sec:prob}
As was already mentioned earlier, producing a good initial guess using spectral initialization seems to
be a prerequisite for prototypical  non-convex algorithms to succeed with good theoretical guarantees. A natural and fundamental  question is:

{\em Is it possible for non-convex algorithms to achieve successful recovery with a random initialization (i.e. without spectral initialization or any additional truncation)}?  


For the intensity-based estimator (\ref{eq:mod1}), the answer is affirmative.  In the recent
work \cite{Sun18}, Ju Sun et al.  carried out a deep study of the global geometric structure of the loss function of (\ref{eq:mod1}). They proved that the loss function $F(\vz)$  does not have any spurious local minima under $O(n \log^3 n)$ Gaussian random measurements. More specifically, it was shown in \cite{Sun18} that all 
minimizers coincide with the target signal $\vx$ up to a global phase, and the loss function has a negative directional curvature around each saddle point. Thanks to this benign geometric landscape any algorithm which can avoid saddle points  converges to the true solution with high probability. 
A trust-region
method was employed in \cite{Sun18} to find the global minimizers with random initialization. To reduce the sampling complexity, it has been shown in \cite{cai2019} that a  combination of the loss function (\ref{eq:mod1}) with a judiciously chosen activation function also possesses the benign geometry structure  under $O(n)$ Gaussian random measurements. Recently, a smoothed amplitude flow estimator has been proposed in \cite{2020a} and the authors show that the loss function has benign geometry structure under the optimal sampling complexity. Numerical tests show that the estimator in \cite{2020a}  yields very stable and fast convergence with random initialization and performs as good as or even better than the existing gradient descent methods with spectral initialization.

The emerging concept of a benign geometric landscape has also recently been   explored in many other applications of signal processing and machine learning, e.g. matrix sensing \cite{bhojanapalli2016global,park2016non}, tensor decomposition \cite{ge2016matrix}, dictionary learning\cite{sun2016complete} and matrix completion \cite{ge2015escaping}. 
For general optimization problems there exist a plethora of loss functions with 
well-behaved geometric landscapes such that all local optima are also global optima and 
each saddle point has a negative direction curvature in its vincinity. 
Correspondingly several techniques have been developed to guarantee that the standard gradient based optimization algorithms can escape such saddle points efficiently, see e.g. \cite{jin2017escape,du2017gradient,jin2017accelerated}.

\subsection{Our contributions}
This paper aims to give a positive answer to the problem proposed in Subsection \ref{sec:prob}, especially for  the amplitude-based model. We first introduce two novel estimators based on a deep modification
of \eqref{eq:mod2} and then we prove rigorously that their loss functions have a benign geometric landscape under the optimal sampling complexity $O(n)$, namely, the loss functions have no spurious local minimizers and  have a negative directional curvature around each saddle point. Such properties allow first order method like gradient descent  to find a global minimum with random initial guess.
We carry out extensive numerical experiments and show that the gradient descent algorithm with random initialization outperforms  the  state-of-the-art algorithms with spectral initialization  in empirical success rate and convergence speed.

We now give a slightly more detailed summary of the main theoretical results proved in our papers. 
Consider the loss function which is akin to the estimator \eqref{eq:mod2}:

\begin{align} \label{Au19:mod2}
\min_{\vz \in \R^n} F(\vz ) &= \frac 1 m \sum_{i=1}^m \biggl( \sqrt{\beta |\vz|^2+(\va_i^T \vz)^2}
-\sqrt{ \beta |\vz|^2 + y_i^2} \biggr)^2.
\end{align}
The following theorem shows that the loss function above has benign geometry structure under optimal sampling complexity.

\begin{theorem} [Informal] \label{th:first}
 Consider the perturbed amplitude model (PAM1) \eqref{Au19:mod2}.  Assume 
$\{\va_i\}_{i=1}^m$ are i.i.d.  Gaussian random vectors and $\vx\ne 0$.
Let $0<\beta<\infty$. 
If $m\ge C n $, then with probability at least $1-O(m^{-2})$, 
the loss function $F=F(\vz)$ has no spurious local minimizers. The only global minimizer is
$\pm \vx$, and all saddle points are strict saddles. 
\end{theorem}

The avid reader should notice that the probability concentration in Theorem \ref{th:first} is only 
$1-O(m^{-2})$. Besides, the function is only Lipschitz continuous near the 
origin.\footnote{This is due to the function $\sqrt{\beta |\vz|^2+ (\va_i^\T \vz)^2}$.}
To remedy this and improve the probability of success,  we introduce
the following genuinely globally smooth estimator:
\begin{align} \label{Au19:mod3}
\min_{\vz \in \R^n} F(\vz ) &= \frac 1 m \sum_{i=1}^m \biggl( \sqrt{\beta |\vz|^2+(\va_i^T \vz)^2 +y_i^2}  
-\sqrt{ \beta |\vz|^2 +2 y_i^2 } \biggr)^2.
\end{align}
The geometric landscape is stated below.

\begin{theorem} [Informal] 
Consider the perturbed amplitude model (PAM2) \eqref{Au19:mod3}.  Assume 
$\{\va_i\}_{i=1}^m$ are i.i.d.  Gaussian random vectors and $\vx\ne 0$.
Let $0<\beta<\infty$. 
If $m\ge C n $, then with probability at least $1-e^{-cm}$, 
the loss function $F=F(\vz)$ has no spurious local minimizers. The only global minimizer is
$\pm \vx$, and all other critical points are strict saddles. 
\end{theorem}

\begin{remark}
In a con-current work \cite{2020a}, 
we considered another new smoothed amplitude-based estimator which is based
on a piece-wise smooth modification of the amplitude estimator 
\eqref{eq:mod2}. The estimator takes the form
\begin{align} \notag
\min_{\vz\in \mathbb R^n} 
F(\vz)=\frac{1}{2m}\sum_{i=1}^m y_i^2 \xkh{\gamma\xkh{\frac{\va_i^\T \vz}{y_i}}-1}^2, 
\end{align}
where the function $\gamma(t)$ is taken to be 
\begin{align*}
\gamma(t):=
\begin{cases}
     \abs{t}, &  \abs{t} > \beta;  \\
    \frac{1}{2\beta} t^2+\frac{\beta}{2}, &  \abs{t} \leq \beta .
\end{cases} 
\end{align*}
For $0<\beta\le  1/2$, we prove that the loss function has a benign landscape under the
optimal sampling threshold $m=O(n)$. There are subtle technical difficulties in connection with
the piecewise-smoothness of the loss function which make the overall proof therein quite special.
On the other hand, there are exciting evidences that the machinery developed in this work can
be generalized significantly in various directions (including complex-valued cases etc). We plan to address some of these important
issues in forthcoming works.
\end{remark}

\subsection{Organization}
The paper is organized as follows. 
In Section 2,  we analyze the global geometric structure for the first estimator, and the global analysis for the second estimator is given in Section 3. For both estimators, we  show that their loss functions have no spurious local minimizers  under optimal sampling complexity $O(n)$.  In Section 4, we give some numerical experiments to demonstrate the efficiency of our proposed estimators.
In Appendix, we collect the technique lemmas which are used in the proof.

\subsection{Notations}
Throughout this proof we fix $\beta>0$ as a constant and do not study
the precise dependence of other parameters on $\beta$.  We write $u \in \mathbb S^{n-1}$ if $u\in \mathbb R^n$ and 
$\|u \|_2=\sqrt{ \sum_{j} (u_j)^2} =1$. 
 We use $\chi$ to denote the usual characteristic function. For example  $\chi_A (x)=1$ if $x \in A$ and $\chi_A(x)=0$ if $x\notin A$. 
 We denote by $\delta_1$, $\epsilon$, $\eta$, $\eta_1$ various 
constants whose value will be taken sufficiently small. The needed smallness will be
clear from the context. 
 For any quantity $X$, we shall write $X=O(Y)$ if $|X| \le C Y$ for some constant
$C>0$.  We write $X\lesssim Y$ if $X\le CY$ for some constant $C>0$.  
We shall write $X \ll Y$ if $ X \le c Y$ where the constant $c>0$ will be sufficiently
small. 
 In our proof it is important for us to specify the precise dependence of the sampling size $m$ in terms of the dimension $n$. For this purpose
we shall write $m\gtrsim n $ if $m\ge C n$ where the constant $C$ is allowed to depend on 
$\beta$ and 
the small constants $\epsilon$, $\epsilon_i$ etc used in the argument.  
One can extract more explicit dependence 
of $C$ on the small constants and $\beta$ but for simplicity we suppress this dependence here. 
 We shall say an event $A$ happens with \textbf{high probability} if  $\mathbb P (A) \ge   1-C e^{-cm}$, 
where $c>0$, $C>0$ are constants.  The constants $c$ and $C$ are allowed to depend
on $\beta$ and the small constants $\epsilon$, $\delta$ mentioned before.

\section{perturbed amplitude model \RNum{1}}  \label{S:model2}
Recall the loss function of perturbed amplitude model (PAM1) \eqref{Au19:mod2}:
\begin{align} \label{model2e1}
f(u) &= \frac 1 m \sum_{j=1}^m \Bigl( \;\sqrt{\beta |u|^2+(a_j\cdot u)^2}
-\sqrt{ \beta |u|^2 + (a_j\cdot x)^2}\;\;  \Bigr)^2,
\end{align}
where $\beta>0$ is a parameter. Here, we denote $|u|:=\|u\|_2=\sqrt{\sum_j u_j^2}$ for the convenience and
write $\va_i$ as $a_i$, $\vx$ as $x$ to alleviate the notation. The global geometric structure of above empirical loss is stated below.

\begin{thm} \label{thmC}
Let $0<\beta <\infty$. 
Assume 
$\{a_i\}_{i=1}^m$ are i.i.d.  standard Gaussian random vectors and $x\ne 0$. 
There exist positive constants $C$, $C_1$ depending only on $\beta$, such that if $m\ge C n  $, then
with probability at least $1- \frac {C_1} {m^2}$ the loss function $f(u)$ 
defined by \eqref{model2e1}
has no spurious local minimizers. The only global minimizer is
$\pm  x$, and the loss function is strongly convex in a neighborhood of $\pm x$.
At the point $u=0$ the loss function has non-vanishing directional gradient along any direction
$\xi \in \mathbb S^{n-1}$. 
All other critical points are strict saddles, i.e., each saddle point has a neighborhood
where the function has negative directional curvature.
\end{thm}

\begin{remark}
We shall show that most of the statements can be proved with high probability
$1-e^{-cm}$. The only part where the weaker probability $1-O(m^{-2})$ is used
comes in the analysis of the strong convexity near the global minimizer $u=\pm x$ (see e.g. Lemma \ref{Sp1_1}). This can be refined but we shall not dwell on it here.
\end{remark}

In view of this homogeneity and the rotation invariance of the Gaussian distribution, we may assume 
without loss of generality that $x= e_1$ when studying the landscape of $f(u)$.
Thus throughout the rest of the proof we shall assume $x=e_1$.

\subsection{The regimes $\|u\|_2\le  \frac {\sqrt{\beta}} {4(1+\beta)} $ and $\|u\|_2\ge 3\sqrt{1+\beta}$ are fine}$\;$

Write  $u=\rho \hat u$ where $\hat u \in S^{n-1}$.   Then 
\begin{eqnarray*}
f_j(u,e_1)& = & \Bigl(\; \sqrt{\beta |u|^2 + (a_j\cdot u)^2}- \sqrt{\beta |u|^2 +(a_j\cdot e_1)^2}  \;\; \Bigr)^2 \notag \\
& = &  \rho^2 ( (a_j\cdot \hat u)^2 + 2 \beta) +(a_j\cdot e_1)^2 -2 \rho \sqrt{\beta + (a_j\cdot \hat u)^2} \sqrt{ \beta \rho^2 +(a_j\cdot e_1)^2}.
\end{eqnarray*}
The derivative with respect to $\rho$ is
\begin{align} \label{partialrho}
\partial_{\rho} f &=\frac 1m \sum_{j=1}^m \Bigl(
2 \rho ( (a_j \cdot \hat u)^2 +2\beta) -2 \sqrt{\beta+ (a_j\cdot \hat u)^2} 
\sqrt{\beta \rho^2 + (a_j\cdot e_1)^2} 
-\frac { 2\beta \rho^2 \sqrt{\beta+ (a_j \cdot \hat u)^2}}
{\sqrt{\beta \rho^2 + (a_j\cdot e_1)^2}} \Bigr).
\end{align}

\begin{lem}[The regime $\rho\ge 3\sqrt{1+\beta}$ is OK] \label{Au28_e1}
For $m\gtrsim n$,  with high probability it holds that
\begin{align*}
\partial_{\rho} f >0,  \quad\forall\, \rho \ge 3\sqrt{1+\beta}, \;\text{}\forall\, \hat u \in \mathbb S^{n-1}.
\end{align*}
\end{lem}
\begin{proof}
To prove this lemma, we need to lower bound the first term and upper bound the last two terms of $\partial_{\rho} f $.
For the first term, by using Bernstein's inequality, we have with high probability,
\begin{align} \label{eq:unibound}
\Bigl| \frac 1 m \sum_{j=1}^m  (a_j\cdot \hat u)^2
 - 1 \Bigr| \le \delta_1 \ll 1, \qquad\forall\, \hat u \in \mathbb S^{n-1}.
 \end{align}
It immediately  gives 
\begin{align*}
&\frac 1 m \sum_{j=1}^m \Bigl( 
2 \rho ( (a_j \cdot \hat u)^2 +2\beta) \Bigr)
\ge 2\rho (  (1-\delta_1) +2\beta).
\end{align*}
For the second term,  simple calculation leads to 
\begin{align*}
& \frac 2m \sum_{j=1}^m \Bigl( \sqrt{\beta+ (a_j\cdot \hat u)^2} 
\sqrt{\beta \rho^2 + (a_j\cdot e_1)^2} 
\Bigr)  \notag \\
\le &\; \frac 2m 
\Bigl( \sum_{j=1}^m (\beta+ (a_j\cdot \hat u)^2 ) \Bigr)^{\frac 12}
\cdot \Bigl( \sum_{j=1}^m (\beta \rho^2 +(a_j\cdot e_1)^2 ) \Bigr)^{\frac 12}
\notag \\
\le&\; 2 (\beta+1+\delta_1)^{\frac 12}
( \beta \rho^2+ 1+\delta_1)^{\frac 12}.
\end{align*}
Finally, it is easy to derive from \eqref{eq:unibound} that
\begin{align*}
&\frac 2m  \sum_{j=1}^m \Bigl(  
\sqrt{\beta+ (a_j \cdot \hat u)^2} \frac { \beta \rho^2}
{\sqrt{\beta \rho^2 + (a_j\cdot e_1)^2}} \Bigr) \notag \\
\le & \; \frac 2m \sum_{j=1}^m \Bigl( \sqrt{\beta+ (a_j\cdot \hat u)^2} 
\sqrt{\beta \rho^2 + (a_j\cdot e_1)^2} 
\Bigr) \notag \\
\le &\; 2 (\beta+1+\delta_1)^{\frac 12}
( \beta \rho^2+ 1+\delta_1)^{\frac 12}.
\end{align*}
Putting all above estimators into \eqref{partialrho} gives
\begin{eqnarray*}
\partial_{\rho} f 
&\ge & 2\rho ( (1-\delta_1) +2\beta)- 
4 (\beta+1+\delta_1)^{\frac 12}
( \beta \rho^2+ 1+\delta_1)^{\frac 12}  \\ 
&=& 2\cdot \frac { \rho^2 (1-(2+8\beta) \delta_1) -4(1+\beta) + (-8-4\beta) \delta_1+ (\rho^2-4) \delta_1^2 } { \rho ( (1-\delta_1) +2\beta)+
2(\beta+1+\delta_1)^{\frac 12}
( \beta \rho^2+ 1+\delta_1)^{\frac 12} }.
\end{eqnarray*}
Clearly if $\delta_1>0$ is sufficiently small and $\rho\ge 3\sqrt{1+\beta}$,
then $\partial_{\rho} f>0$. 
\end{proof}

\begin{lem}[The regime $0<\rho\le  \frac {\sqrt{\beta}}{4(1+\beta)} $ is OK] \label{Au28_e2}
For $m\gtrsim n$,  with high probability it holds that
\begin{align*}
\partial_{\rho} f <0,  \quad\forall\,  0<\rho \le \frac{\sqrt{\beta}}{4(1+\beta)}, \;\text{}\forall\, \hat u \in \mathbb S^{n-1}.
\end{align*}
\end{lem}
\begin{proof}
By Bernstein's inequality, we have with high probability, 
\begin{align*}
&\Bigl |\frac 1m \sum_{j=1}^m 
|a_j\cdot e_1| - \sqrt{\frac 2{\pi}} 
\Bigr| \le \delta_1\ll 1,   \\
& \Bigl| \frac 1 m \sum_{j=1}^m  (a_j\cdot \hat u)^2
- 1 \Bigr| \le \delta_1 \ll 1, \qquad\forall\, \hat u \in \mathbb S^{n-1}.
\end{align*}
Putting this into \eqref{partialrho} gives
\begin{align*}
\partial_{\rho} f
& \le \frac 1m \sum_{j=1}^m \Bigl( 2\rho ((a_j\cdot \hat u)^2 +2\beta )\Bigr) 
 - \frac 1m \sum_{j=1}^m  2\sqrt{\beta}  |a_j \cdot e_1| \notag \\
 & \le 2\rho (1+\delta_1+2\beta) - 2\sqrt{\beta}(\sqrt{\frac 2 {\pi} } -\delta_1).
 \end{align*}
 Since $\sqrt{\frac 2{\pi}} \approx 0.797885 $, the desired result clearly
 follows by taking $\delta_1$ sufficiently small. 
\end{proof}

The point $u=0$ needs to be treated with care since our loss function $f(u)$ is only
Lipschitz at this point. To this end, we define the one-sided directional
derivative of $f$ along a direction $\xi \in \mathbb S^{n-1}$ as
\begin{align} \label{Dxi0}
D_{\xi} f (0) = \lim_{t\to 0^{+} } \frac {f(t\xi)} t.
\end{align}
It is easy to check that 
\begin{align*}
D_{\xi} f(0) = - \frac 2m \sum_{j=1}^m \sqrt{\beta+(a_j\cdot \xi)^2} |a_j \cdot e_1|.
\end{align*}

\begin{lem}[The point $u=0 $ is OK] \label{Au28_e3}
For $m\gtrsim n$,  with high probability it holds that
\begin{align*}
D_{\xi} f(0) <-\sqrt{\beta},  \quad\forall\, \xi \in \mathbb S^{n-1}.
\end{align*}
\end{lem}
\begin{proof}
Clearly with high probability and uniformly in $\xi \in \mathbb S^{n-1}$,
\begin{align*}
D_{\xi} f(0) \le -\frac 2m \sum_{j=1}^m \sqrt{\beta} |a_j\cdot e_1|
<- 2\sqrt{\beta} (\sqrt{\frac 2 {\pi} } -0.01) <-\sqrt{\beta}.
\end{align*}
\end{proof}

In summary, we have the following theorem.

\begin{thm}[Non-vanishing gradient when $\|u\|_2\le \frac {\sqrt{\beta}}{4(1+\beta)}$
or $\|u\|_2\ge 3\sqrt{1+\beta}$] \label{thmAu28_1}
For $m\gtrsim n$, with high probability the following hold:
\begin{enumerate}
\item We have 
\begin{align*}
&\partial_{\rho} f <0, \qquad\forall\, 0<\rho\le \frac {\sqrt{\beta}}{4(1+\beta)}, 
\quad \forall\, \hat u \in \mathbb S^{n-1};\\
&\partial_{\rho} f>0, \qquad\forall\, \rho\ge 3\sqrt{1+\beta},\quad\forall\, \hat u
\in \mathbb S^{n-1}.
\end{align*}
\item For $u=0$,  we have
\begin{align*}
D_{\xi} f(0) <-\sqrt{\beta}, \qquad \forall\, \xi \in \mathbb S^{n-1},
\end{align*}
where $D_{\xi} f(0)$ was defined in \eqref{Dxi0}.
\end{enumerate}
\end{thm}
\begin{proof}
This follows from Lemmas \ref{Au28_e1}, \ref{Au28_e2} and \ref{Au28_e3}.
\end{proof}

\subsection{Analysis of the regime $\rho \sim 1$, $||\hat u \cdot e_1|-1|\ge
\epsilon_0>0$} 

$\;$

\vspace{0.2cm}
In this section we consider the regime $0< c_1 < \rho < c_2<\infty$, $ 
|\hat u \cdot e_1| < 1-\epsilon_0$, where $0<\epsilon_0\ll 1$. The choice of the 
constants $c_1$ and $c_2$ can be quite flexible. For example, we can
take $c_1= \frac{\sqrt{\beta}} {4(1+\beta)}$, $c_2= 3(1+\beta)$. For this reason
we write $\rho\sim 1$. 
To simplify the discussion, we need to employ a new coordinate system.  Write
\begin{align*}
\hat u &= (\hat u \cdot e_1) e_1 +\tilde u, \notag \\
 &= t e_1+ \sqrt{1-t^2} e^{\perp}, 
 \end{align*}
 where $e^{\perp} \in \mathbb S^{n-1}$ satisfies $e^{\perp} \cdot e_1=0$. 
Clearly in the regime $\rho\sim 1$, $|t|<1$, we have a smooth representation
\begin{align*}
u = \rho \hat u = \rho (t e_1+ \sqrt{1-t^2} e^{\perp} ) =: \psi( \rho, t, e^{\perp}).
\end{align*}
The following pedestrian proposition shows that the landscape of a smooth function
 undergoes mild changes
under smooth change of variables. 
\begin{prop}[Criteria for no local minimum]
\label{Au29e1}
In the regime $\rho \sim 1$, $|t|<1$, consider
\begin{align*}
f(u) = f ( \psi(\rho,t,e^{\perp} ) ) = : g(\rho, t, e^{\perp}).
\end{align*}
Then the following hold:
\begin{enumerate}
\item If at some point $|\partial_t g | >0$, then $\| \nabla f \|_2 >0$ at the corresponding point.
\item If at some point $\partial_{tt} g<0$, then either $\nabla f \ne 0$ at the corresponding point,
or $\nabla f=0$ and $f$ has a negative curvature at the corresponding 
point (i.e. a strict saddle). 
\end{enumerate}
\end{prop}
\begin{proof}
These easily follow from the formulae:
\begin{align*}
&\partial_t g = \nabla f \cdot \partial_t \psi, \\
&\partial_{tt} g= \nabla f \cdot \partial_{tt} \psi+ (\partial_t \psi)^T \nabla^2 f \partial_t \psi,
\end{align*}
where $\nabla^2 f = (\partial_{ij} f)$ denotes the Hessian matrix of $f$. 
\end{proof}
Proposition \ref{Au29e1} allows us to simplify the computation greatly by looking only
at the derivatives $\partial_t $ and $\partial_{tt}$.  We shall use these in the regime
$|t| < 1-\epsilon_0$ where $0<\epsilon_0\ll 1$. Now observe that
 \begin{align*}
 \frac 12\mathbb E f = \frac 12 (1+2\beta) \rho^2 +\frac 12- \rho
 \mathbb E  \sqrt{ \beta + (a\cdot \hat u)^2} \sqrt{\beta \rho^2 + (a\cdot e_1)^2},
 \end{align*}
 where $a\sim \mathcal N (0, \operatorname{I_n})$. 
Denote $X_1= a\cdot e_1$ and $Y_1= a\cdot e^{\perp}$ so that
$a\cdot \hat u = t X_1 +\sqrt{1-t^2} Y_1=: X_t$. 
We focus on the term 
\begin{align*}
 &\mathbb E  \sqrt{ \beta + (a\cdot \hat u)^2} \sqrt{\beta \rho^2 + (a\cdot e_1)^2} \\
=&\; \mathbb E \sqrt{\beta+X_t^2} \sqrt{\beta \rho^2+X_1^2} =:\, h_{\infty}(\rho, t). 
\end{align*}
\begin{lem}[The limiting profile] \label{lemAu29_2}
For any $0<\eta_0\ll 1$, the following hold:

\begin{enumerate}
\item $ \sup_{|t| \le 1-\eta_0, \rho \sim 1} ( |\partial_t h_{\infty}(\rho,t)| + |\partial_{tt} h_{\infty} (\rho, t) | + |\partial_{ttt} h_{\infty}(\rho, t) |) \lesssim 1$. \vspace{1em}
\item $|\partial_t h_{\infty}(\rho, t)| \gtrsim |t|$ for $0<|t|<1$, $\rho\sim 1$. \vspace{1em}
\item $\partial_{tt} h_{\infty}(\rho, t)\gtrsim 1$ for $|t| \ll 1$, $\rho\sim 1$. 
\end{enumerate}
\end{lem}
\begin{proof}
See appendix. 
\end{proof}

\begin{thm}[The regime $\|u\|_2\sim 1$, $||\hat u \cdot e_1|-1|>\eta_0$ is fine]
\label{thmA31_1}
For any given $0<\eta_0 \ll 1$ and $0<c_1<c_2<\infty$, if $m\gtrsim n$,
then the following hold with high probability: In the regime $c_1< \rho =\|u\|_2< c_2$, $| |\hat u \cdot e_1|-1 |
> \eta_0$, there are only two possibilities:
\begin{enumerate}
\item  $\| \nabla f (u)\|_2 >0$;
\item $\nabla f(u)=0$, and $f$ has a negative directional curvature at this point.
\end{enumerate}
\end{thm}
\begin{proof}
Denote
\begin{align*}
g(\rho, t, e^{\perp})
&= 2\beta \rho^2 +\rho^2
\frac 1m \sum_{j=1}^m (a_j\cdot \hat u)^2 - 2\rho \frac 1m \sum_{j=1}^m
\sqrt{\beta +(a_j\cdot \hat u)^2}
\cdot \sqrt{ \beta \rho^2 +X_j^2} \notag \\
&=:2\beta \rho^2 +\rho^2 h_0 (\rho, t,e^{\perp})- 2\rho h(\rho, t, e^{\perp}),
\end{align*}
where $X_j=a_j\cdot e_1$, and
\begin{align*}
a_j \cdot \hat u = t X_j +\sqrt{1-t^2} Y_j, \qquad Y_j =a_j\cdot e^{\perp}.
\end{align*}
Clearly
\begin{align*}
&\partial_t g = \rho^2 \partial_t h_0 -2\rho \partial_t h; \\
&\partial_{tt}g
= \rho^2 \partial_{tt} h_0 -2 \rho \partial_{tt} h.
\end{align*}
Observe that
\begin{align*}
h_0 & = \frac 1m \sum_{j=1}^m (tX_j +\sqrt{1-t^2} Y_j)^2 \notag \\
& = t^2 \frac 1m \sum_{j=1}^m X_j^2
+ 2t\sqrt{1-t^2} \frac 1m
\sum_{j=1}^m X_j Y_j + (1-t^2) \frac 1m \sum_{j=1}^m Y_j^2.
\end{align*}
Clearly then for any small $\epsilon>0$ and $m\gtrsim n$, it holds with high
probability that
\begin{align*}
|\partial_t h_0 - \mathbb E \partial_t h_0|
+|\partial_{tt} h_0 - \mathbb E \partial_{tt} h_0|
\le \epsilon, \qquad\forall\, |t| \le 1-\epsilon_0, e^{\perp}\cdot e_1=0,
e^{\perp} \in \mathbb S^{n-1}.
\end{align*}
Note that we actually have $\mathbb E \partial_t h_0=0$ and $\mathbb E \partial_{tt} h_0=0$.
By Lemma \ref{A30_2} and \ref{A30_3},  
for any small $\epsilon>0$ and $m\gtrsim n$, it also holds with high
probability that
\begin{align*}
& |\partial_t h - \mathbb E \partial_t h| 
\le \epsilon, \qquad\forall\, |t| \le 1-\epsilon_0, e^{\perp}\cdot e_1=0,
e^{\perp} \in \mathbb S^{n-1}, \forall\, c_1\le \rho \le c_2;\\
&\partial_{tt} h \ge \mathbb E \partial_{tt} h
-\epsilon, \qquad\forall\, |t| \le 1-\epsilon_0, e^{\perp}\cdot e_1=0,
e^{\perp} \in \mathbb S^{n-1}, \forall\, c_1\le \rho \le c_2.
\end{align*}
We then obtain for small $\epsilon>0$, if $m\gtrsim n$, it holds 
with high probability  that
\begin{align*}
&|\partial_t g -\mathbb E \partial_t g | \le 2 \epsilon; \\
& \partial_{tt} g \le 
\mathbb E \partial_{tt} g +2 \epsilon.
\end{align*}
Clearly
\begin{align*}
&\mathbb E \partial_t g = -2 \rho  \partial_t h_{\infty}(\rho, t) ;\\
&\mathbb E \partial_{tt} g = -2 \rho  \partial_{tt} h_{\infty}(\rho, t).
\end{align*}
By Lemma \ref{lemAu29_2}, we can take $t_0\ll 1$ such that
\begin{align*}
&\partial_{tt} h_{\infty} (\rho, t) \ge \epsilon_1>0, \qquad\forall\, |t|\le t_0, 
c_1\le \rho \le c_2; \\
&|\partial_t h_{\infty} (\rho, t)| \ge \epsilon_2>0, \qquad\forall\, t_0\le |t|\le 1-\epsilon_0,
c_1\le \rho\le c_2.
\end{align*}
By taking $\epsilon>0$ sufficiently small, we can then guarantee that
\begin{align*}
& |\partial_t g| >\epsilon_3>0, 
\qquad\forall\, |t|\le t_0, 
c_1\le \rho \le c_2; \\
&\partial_{tt} g  \le -\epsilon_4<0, \qquad\forall\, t_0\le |t|\le 1-\epsilon_0,
c_1\le \rho\le c_2.
\end{align*}
The desired result then follows from Proposition \ref{Au29e1}. 
\end{proof}

\subsection{Localization of $\rho$, the regime $||\hat u \cdot e_1|-1|\ll 1$}
$\,$

\vspace{0.2cm}

In this section we shall localize $\rho$ under the assumption
that $||\hat u \cdot e_1|-1| \ll 1$, i.e., we shall show that if 
$| |\hat u \cdot e_1|-1|\le \epsilon_0 \ll 1$, then with high probability that
$|\rho -1| \le  \eta(\epsilon_0) \ll 1$. 
In the lemma below we assume $\rho\ge c_1$ since by Theorem \ref{thmAu28_1} the 
regime $\rho \ll 1$ is already treated. 
\begin{lem}\label{lem_6.5}
Let $0<\beta\le \frac 14$ and consider the regime $0<c_1\le \rho <1$. If $0< \eta_0\ll 1$ is sufficiently small, then
for $m\gtrsim n$, it holds
with high probability that
\begin{align*}
\partial_{\rho} f <0, \qquad \forall\,  \rho \le 1-c(\eta_0), 
\forall\, \hat u \in \mathbb S^{n-1}
\text{ with $| |\hat u \cdot e_1|-1| \le \eta_0$}, 
\end{align*}
where $c(\eta_0)\to 0$ as $\eta_0\to 0$.
\end{lem}
\begin{remark}
In Theorem \ref{thA31_0a} we shall remove the constraint  $0<\beta\le \frac 14$  and prove the result for all $0<\beta<\infty$.
\end{remark}

\begin{proof}[Proof of Lemma \ref{lem_6.5}]
Recall
\begin{align*}
\frac 12 \partial_{\rho} f &=\frac 1m \sum_{j=1}^m \Bigl(
 \rho ( (a_j \cdot \hat u)^2  +2\beta)- \sqrt{\beta+ (a_j\cdot \hat u)^2} 
\sqrt{\beta \rho^2 + (a_j\cdot e_1)^2} - \frac { \beta \rho^2 \sqrt{\beta+ (a_j \cdot \hat u)^2}}
{\sqrt{\beta \rho^2 + (a_j\cdot e_1)^2}} \Bigr).
\end{align*}
Without loss of generality we assume
\begin{align*}
\| \hat u -e_1\|_2\le r \ll 1.
\end{align*}
The other case $\| \hat u+e_1\|_2 \ll 1$ is similar and therefore omitted.  Note that
\[
|\sqrt{\beta+ (a_j\cdot \hat u)^2}-\sqrt{\beta+ (a_j\cdot e_1)^2}| = \Bigl| \frac   {(a_j \cdot (\hat u-e_1))  (a_j\cdot (\hat u +e_1) )} { \sqrt{\beta+ (a_j\cdot \hat u)^2}+\sqrt{\beta+ (a_j\cdot e_1)^2}} \Bigr| \le  |a_j \cdot (\hat u -e_1)|.
\]
It immediately gives
\begin{align*}
\frac 12 \partial_{\rho}f 
& \le \frac 1m \sum_{j=1}^m \Bigl( \rho ( (a_j\cdot \hat u)^2 +2\beta) 
-\sqrt{\beta+(a_j\cdot e_1)^2} \sqrt{ \beta \rho^2 +(a_j\cdot e_1)^2}
-\frac {\beta \rho^2\sqrt{\beta+(a_j\cdot e_1)^2}}{\sqrt{\beta \rho^2+(a_j\cdot e_1)^2}} \Bigr)
+H,
\end{align*}
where
\begin{align*}
H &\lesssim \frac 1m \sum_{j=1}^m |a_j\cdot (\hat u-e_1)|
(1+|a_j\cdot e_1|)  \lesssim \frac 1m \sum_{j=1}^m \Bigl(a_j\cdot (\hat u-e_1))^2\cdot \frac 1 {2r}
+r (1 + (a_j\cdot e_1)^2) \Bigr). 
\end{align*}
Clearly it holds with high probability that
\begin{align*}
H \le B_1 r,
\end{align*}
where $B_1>0$ is a constant.  For $\rho \le 1$, we have
\begin{align*}
&-\sqrt{\beta+(a_j\cdot e_1)^2} \sqrt{ \beta \rho^2 +(a_j\cdot e_1)^2}
\le  - (\beta \rho^2 + (a_j\cdot e_1)^2); \\
& -\sqrt{\beta+(a_j\cdot e_1)^2}
\cdot \frac {\beta \rho^2}{\sqrt{\beta \rho^2+(a_j\cdot e_1)^2}}
\le -\beta \rho^2.
\end{align*}
Then assuming $\rho<1$,  it holds with high probability that
\begin{align*}
\frac 12 \partial_{\rho} f &\le \rho ( 1+2\beta) - (\beta \rho^2+1)
- \beta \rho^2 +B_1 r +\delta_1,
\end{align*}
where $\delta_1>0$ is a small constant which accounts for the deviation from the 
mean value used in the Bernstein's inequality.  For $0<\beta\le \frac 14$ (actually
$0<\beta<\frac 12$ suffices) the desired conclusion
then clearly follows by taking $\delta_1= O(\eta_0)$ and $r=O(\eta_0)$. 
\end{proof}

\begin{lem}[$\partial_{\rho\rho}f$ is good] \label{A31_e10}
We have almost surely it holds that
\begin{align*}
\partial_{\rho\rho} f >0, \qquad\forall\, 0<\rho<\infty, \, \forall\, \hat u \in \mathbb S^{n-1}.
\end{align*}
Furthermore, for any fixed two constants $0<c_1<c_2<\infty$, if
$m\gtrsim n$, then it holds with high probability that 
\begin{align*}
\partial_{\rho\rho} f \ge \alpha >0, \qquad\forall\,
c_1\le \rho \le c_2, \,\forall\, \hat u \in \mathbb S^{n-1},
\end{align*}
where $\alpha>0$ is a constant depending only on $(c_1,c_2,\beta)$. 
\end{lem}

\begin{proof}
Recall
\begin{align*}
\frac 12 \partial_{\rho} f &=\frac 1m \sum_{j=1}^m \Bigl(
\rho ( (a_j \cdot \hat u)^2  +2\beta) - \sqrt{\beta+ (a_j\cdot \hat u)^2} \sqrt{\beta \rho^2 + (a_j\cdot e_1)^2} 
-  \frac { \beta \rho^2 \sqrt{\beta+ (a_j \cdot \hat u)^2}}{\sqrt{\beta \rho^2 + (a_j\cdot e_1)^2}} \Bigr).
\end{align*}
A simple calculation leads to 
\begin{eqnarray*} 
&& \frac 12 \partial_{\rho\rho} f \\
&=& \frac 1m \sum_{j=1}^m \Bigl(
  ( (a_j \cdot \hat u)^2  +2\beta) - 
\frac{3\beta \rho \sqrt{\beta+ (a_j\cdot \hat u)^2} }{\sqrt{\beta \rho^2 + (a_j\cdot e_1)^2} } +
 \frac { \beta^2 \rho^3\sqrt{\beta+ (a_j \cdot \hat u)^2}}
{({\beta \rho^2 + (a_j\cdot e_1)^2})^{\frac 32} } \Bigr)\notag \\
&=&  \frac 1m \sum_{j=1}^m \biggl(
  ( (a_j \cdot \hat u)^2  +2\beta) - \sqrt{\beta+ (a_j\cdot \hat u)^2} 
\cdot \sqrt{\beta}\cdot
\Bigl( 3 \left(\frac { a_j\cdot e_1}{\sqrt{\beta} \rho} \right)^2+2 \Bigr)
\cdot \Bigl( \left(\frac { a_j\cdot e_1}{\sqrt{\beta} \rho} \right)^2+1 \Bigr)^{-\frac 32}
\biggr).
\end{eqnarray*}
For $0\le x <\infty$, denote 
\begin{align*}
h_0(x)= \frac {3x+2} {  (1+x)^{\frac 32} }.
\end{align*}
It is not difficult to check that
\begin{align} \label{eq:lem2.61}
h_0(x) \le 2, \qquad \forall\, 0\le x <\infty
\end{align}
and the equality holds if and only if $x=0$. 
Now define $h_1(x) = h_0(x^2)$. Then we can rewrite $\partial_{\rho\rho} f$ as
\begin{align*}
\frac 12 \partial_{\rho\rho} f
= \frac 1m \sum_{j=1}^m ( \sqrt{\beta +(a_j\cdot \hat u)^2} -\sqrt{\beta} )^2 
+\frac 1m \sum_{j=1}^m \sqrt{\beta+(a_j\cdot \hat u)^2}
\cdot \sqrt{\beta}
\cdot \Bigl( 2 -h_1( \frac {a_j \cdot e_1} {\sqrt{\beta }\rho }  )\Bigr).
\end{align*}
It then follows from \eqref{eq:lem2.61} that 
\begin{align*}
\frac 12 \partial_{\rho\rho}f
\ge  \frac \beta m \sum_{j=1}^m 
\Bigl( 2 -h_1( \frac {a_j \cdot e_1} {\sqrt{\beta} \rho }  )\Bigr) >0
\end{align*}
holds almost surely  since the event $\bigcap_{j=1}^m \{ a_j\cdot e_1=0 \}$ has 
zero probability.
By using the Bernstein's inequality, we have with high probability that
\begin{align*}
\frac 1m \sum_{j=1}^m 
\Bigl( 2 -h_1( \frac {a_j \cdot e_1} {\sqrt{\beta }\rho }  )\Bigr)\gtrsim 1, \qquad \forall\, c_1\le \rho \le c_2.
\end{align*}
Thus 
\begin{align*}
\partial_{\rho\rho} f \gtrsim 1, \qquad \forall\,
c_1\le \rho \le c_2, \quad \forall\, \hat u \in \mathbb S^{n-1}.
\end{align*}
\end{proof}

\begin{thm}[Localization of $\rho$ when $||\hat u \cdot e_1|-1|\ll 1$] \label{thA31_0a}
Consider the regime $0<c_1\le \rho \le c_2$. If $0<\eta_0\ll 1$ is sufficiently small, then
for $m\gtrsim n$, it holds
with high probability that
\begin{align*}
&\partial_{\rho} f <0, \qquad \forall\,  \rho \le 1-c(\eta_0), \forall\, \hat u\in \mathbb S^{n-1}
\text{ with $||\hat u \cdot e_1|-1| \le \eta_0$};\\
&\partial_{\rho} f >0, \qquad \forall\,  \rho \ge 1+c(\eta_0), \forall\, \hat u\in \mathbb S^{n-1}
\text{ with $||\hat u \cdot e_1|-1| \le \eta_0$};
\end{align*}
where $c(\eta_0)\to 0$ as $\eta_0 \to 0$.
\end{thm}

\begin{proof}
We shall sketch the proof. 
We first consider  the regime $\rho \ge 1$. 
Without loss of generality we assume $\| \hat u -e_1\|_2 \le \eta_0$. The other
case $\| \hat u +e_1\|_2 \le \eta_0$ can be similarly treated.

First observe that
\begin{align*}
(\partial_{\rho} f)( \rho=1, \hat u=e_1)=0.
\end{align*}
Then by a calculation similar to the estimate of $H$ term in
Lemma \ref{lem_6.5}, we have with high probability that
\begin{align*}
|\partial_{\rho} h( \rho=1, \hat u)| =c_1(\eta_0)\ll 1, \qquad \forall\, |\hat u -e_1|\le \eta_0,
\end{align*}
where $c_1(\eta_0)\to 0$ as $\eta_0\to 0$.
Now by Lemma \ref{A31_e10}, it holds with high probability that
\begin{align*}
\partial_{\rho\rho} f \ge \alpha>0, \quad\forall\, c_1\le \rho \le c_2,
\forall\, \hat u \in \mathbb S^{n-1}.
\end{align*}
It then implies that for $\rho \ge  1+ \frac {2c_1(\eta_0)} {\alpha_0}$, we have
\begin{align*}
(\partial_{\rho} f) ( \rho, \hat u) \ge  \alpha_0
\cdot \frac {2c_1(\eta_0)} {\alpha_0}   - c_1 (\eta_0)=c_1(\eta_0)>0. 
\end{align*}
Redefining $c(\eta_0)$ suitably then yields the result.  The argument
for $\rho \le 1-c(\eta_0)$ is similar. We omit the details.
\end{proof}

\subsection{Strong convexity near the global minimizers $u=\pm e_1$: analysis
of the limiting profile}
$\;$

\vspace{0.2cm}

In this section we shall show that in the small neighborhood of $u=\pm e_1$
where
\begin{align*}
| |\hat u \cdot e_1| -1| \ll 1, \quad  |\rho-1| \ll 1,
\end{align*}
  the Hessian of the expectation of the loss function 
must be strictly positive definite. In yet other words $\mathbb E f$ must be strictly convex
in this neighborhood so that $u=\pm e_1$ are the unique minimizers. 
To this end consider
\begin{align} \label{Au31e71}
h(u) =\frac 12 ( \mathbb E f -1)= \frac 12 (1+2\beta ) |u|^2   -  \mathbb E
\sqrt{\beta |u|^2 + (a\cdot u)^2} \sqrt{ \beta|u|^2 +(a\cdot e_1)^2},
\end{align}
where $a\sim \mathcal N (0, \operatorname{I}_n)$. 
\begin{thm}[Strong convexity of $\mathbb E f$ when $\| u\pm e_1\| \ll 1$]
\label{thA31e50}
Consider $h$ defined by \eqref{Au31e71}.
There exist $0<\epsilon_0\ll 1$ and a positive constant $\gamma_1$ such that the following hold:
\begin{enumerate}
\item If $\| u-e_1\|_2 \le \epsilon_0$, then for any $\xi \in \mathbb S^{n-1}$ it holds
\begin{align*}
\sum_{i,j=1}^n \xi_i \xi_j \mathbb E 
(\partial_i \partial_j h)(u) \ge \gamma_1 >0.
\end{align*}
\item If $\| u+ e_1\|_2 \le \epsilon_0$, then for any $\xi \in \mathbb S^{n-1}$ it holds
\begin{align*}
\sum_{i,j=1}^n \xi_i \xi_j \mathbb E 
(\partial_i \partial_j h)(u) \ge \gamma_1 >0.
\end{align*}

\end{enumerate}

\end{thm}
\begin{proof}
We shall only consider the case $\| u -e_1\|_2\ll 1$. The other case $\| u +e_1\|_2\ll 1$
is similar and therefore omitted. Note that
\begin{align*}
\| u -e_1\|_2^2 = \| \rho \hat u -e_1\|_2^2
= (\rho-1)^2 +2\rho (1- t) \le \epsilon_0^2,
\end{align*}
where $t=\hat u \cdot e_1$.  Thus for $0<\epsilon_0\ll 1$, we have
\begin{align*}
|\rho -1| \le  \epsilon_0, \qquad    1- {\epsilon_0^2} \le t \le 1.
\end{align*}

We now need to make a  change of variable. The representation $u=t e_1+\sqrt{1-t^2} e^{\perp}$ is not so suitable since the derivatives
blow up as $t \to 1-$.  This is an artificial singularity due to the non-smoothness of
the representation $\sqrt{1-t^2}$ as $t\to 1-$.  To resolve this, we use a
different representation (recall $1-\epsilon_0^2 \le \hat u \cdot e_1 \to 1$),
\begin{align*}
\hat u = \sqrt{1-s^2} e_1 + s e^{\perp}, \qquad e^{\perp} \cdot e_1 =0, \, e^{\perp}
\in \mathbb S^{n-1}, 
\end{align*}
where we assume $0\le s \ll 1$. Note that $s= \frac {|u^{\prime}|}{\rho} $, and
$u^{\prime}=u-(u\cdot e_1) e_1= (0, u_2,\cdots,u_n)$. 


To calculate $\partial^2 h$ we need to compute the Hessian expressed in the
$(\rho, s)$ coordinate. It is not difficult to check that by \eqref{Au31e71},
the value of $h(u)$  depends only on $(\rho, s)$. 
Thus by a slight abuse of notation
we write  $h= h(|u|,\, \frac{|u^{\prime}|} {|u|} ) =h(\rho,s)$ (we denote $|u|=\|u\|_2$,
$|u^{\prime}|=\|u^{\prime}\|_2$) and compute
(below we assume $s>0$ so that $|u^{\prime}|>0$)
\[
\partial_i h = \partial_{\rho} h \frac{u_i}{\rho}
+ \partial_s h\cdot  (-\frac{u_i}{\rho^3} |u^{\prime}| +
1_{i\ne 1} \frac 1 {\rho} \cdot \frac {u_i}{|u^{\prime}|})
\]
and 
\begin{eqnarray*}
\partial_{ij} h & =  & \partial_{\rho\rho} h \frac{u_i u_j}{\rho^2} + \frac{u_i}{\rho} \partial_{\rho s} h \cdot (-\frac{u_j}{\rho^3} |u^{\prime}|  + \partial_{\rho} h \cdot ( \frac {\delta_{ij}}{\rho} - \frac{u_i u_j}{\rho^3} ) + \frac{u_j} {\rho} \partial_{\rho s}  h\cdot  (-\frac{u_i}{\rho^3} |u^{\prime}| +
1_{i\ne 1} \frac 1 {\rho} \cdot \frac {u_i}{|u^{\prime}|}) \notag \\
&&   + 1_{j\ne 1} \frac 1 {\rho} \cdot \frac {u_j}{|u^{\prime}|})   + \partial_{ss} h\cdot  (-\frac{u_i}{\rho^3} |u^{\prime}| +
1_{i\ne 1} \frac 1 {\rho} \cdot \frac {u_i}{|u^{\prime}|})
\cdot  (-\frac{u_j}{\rho^3} |u^{\prime}| +
1_{j\ne 1} \frac 1 {\rho} \cdot \frac {u_j}{|u^{\prime}|}) \notag \\
&&  + \partial_s h \cdot (
-\frac{\delta_{ij}} {\rho^3} |u^{\prime}| + 
\frac{3u_i u_j}{\rho^5} |u^{\prime}| - \frac{u_i}{\rho^3} \frac{u_j}{|u^{\prime}|}
1_{j\ne 1} - \frac{u_j}{\rho^3} \frac{u_i}{|u^{\prime}|}
1_{i\ne 1} + 1_{i\ne 1} \frac 1 {\rho} \frac{\delta_{ij}}{|u^{\prime}|}
-1_{i \ne 1} 1_{j\ne 1} \frac 1 {\rho} \cdot \frac{u_i u_j} {|u^{\prime}|^3}).
\end{eqnarray*}
Then denoting $a=\xi \cdot \hat u$,  $b= \sum_{j\ne 1} \xi_j \cdot \frac{u_j}{|u^{\prime}|}$,
we have
\begin{eqnarray*}
\sum_{i,j=1}^n \xi_i \xi_j \partial_{ij } h
&=&  \partial_{\rho\rho} h \cdot a^2 
+ 2 a \partial_{\rho s} h \cdot( - \frac{a s}{\rho} +\frac b {\rho})  + \partial_{\rho} h \cdot ( \frac{|\xi|^2- |\xi \cdot \hat u|^2} {\rho} ) + \partial_{ss} h \cdot ( - \frac{ as}{\rho} +\frac{b}{\rho})^2 \notag \\
&&  + \partial_s h \cdot( -\frac{|\xi|^2}{\rho^2} s+ 3 \frac{a^2 s}{\rho^2} - 2\frac{ab}{\rho^2}+\frac{|\xi^{\prime}|^2} {\rho |u^{\prime}|} - \frac{b^2}{\rho |u^{\prime}|} ) \notag \\
&=&   \partial_{\rho\rho} h \cdot a^2 
+ 2 a \partial_{\rho s} h \cdot( - \frac{a s}{\rho} +\frac b {\rho})  \quad
+ \partial_{\rho} h \cdot ( \frac{|\xi|^2- |\xi \cdot \hat u|^2} {\rho} ) \notag \\
&& + \partial_{ss} h \cdot ( \frac {a^2 s^2-2abs}{\rho^2} )
+ (\partial_{ss} h-\frac 1s \partial_s h) \frac{b^2} {\rho^2}  + \partial_s h \cdot( 
-\frac{|\xi|^2}{\rho^2} s
+ 3 \frac{a^2 s}{\rho^2} - 2\frac{ab}{\rho^2}+\frac{|\xi^{\prime}|^2}
{\rho^2 s}  ).
\end{eqnarray*}
We should point it out that, in the above computation, one does not need to worry about the formal singularity
caused by $\frac 1 s$. Since 
$\partial_s h(\rho, s=0)=0$, we write
\begin{align*}
(\partial_s h)(\rho, s) \cdot
\frac {1}s 
= \frac { (\partial_s h)(\rho, s) -(\partial_s h)(\rho, 0) } s
= \int_0^1 (\partial_{ss} h) (\rho, \theta s) d\theta, \quad s>0.
\end{align*}
In particular we have
\begin{align*}
&\lim_{s\to 0^{+}}
(\partial_s h)(\rho, s)  \cdot \frac 1s
 =( \partial_{ss} h) (\rho, 0) ;\\
 &\Bigl| (\partial_{ss} h)(\rho,s)- \frac 1 s \partial_s h_1 (\rho,s) \Bigr|
 = O(s) \to 0, \quad \text{as $s\to 0$}.
\end{align*}
By using this observation and Lemma \ref{leAu31_61}, we obtain
\begin{align*}
\sum_{i,j=1}^n \xi_i \xi_j (\partial_{ij} h)(e_1)
&=  (\partial_{\rho\rho} h)(1,0) \cdot a^2\Bigr|_{a=\xi_1}
+ (\partial_{ss} h)(1,0) \cdot |\xi^{\prime}|^2 \notag \\
& \ge \gamma_0 \cdot |\xi|^2, \qquad \forall\, \xi \in \mathbb S^{n-1},
\end{align*}
where $\gamma_0>0$ is a constant. 
Now for $\| u -e_1\|_2 \ll 1$, by using Lemma \ref{leAu31_60} and
Lemma \ref{leAu31_61}, we have
\begin{align*}
 & \Bigl| \sum_{i,j=1}^n \xi_i \xi_j \Bigl( (\partial_{ij}h)(e_1)-
 (\partial_{ij} h)(u) \Bigr) \Bigr| \notag \\
 \lesssim & \; 
 |(\partial_{\rho\rho} h)(\rho, s) - (\partial_{\rho\rho} h)(1,0) |
 + | (\partial_{\rho\rho}h)(\rho,s)| \cdot |
 (\xi\cdot \hat u)^2-(\xi \cdot e_1)^2| 
 \notag \\
 &\quad+ | (\partial_{\rho s} h) (\rho, s)| \cdot (1+s) 
 + |(\partial_{\rho} h)(\rho, s)| +|\partial_{ss} h(\rho,s)| \cdot (s+s^2) \notag \\
 & \qquad + | \partial_s h(\rho, s)|
 + \Bigl| \frac  1 {\rho^2} \int_0^1 (\partial_{ss} h) (\rho, \theta s) d\theta
 - (\partial_{ss} h)(1,0) \Bigr| \notag \\
 & \qquad + \Bigl| (\partial_{ss} h)(\rho,s) - (\partial_s h) (\rho,s) \cdot \frac 1 s\Bigr|
 \cdot \frac {b^2} {\rho^2} \notag \\
 \lesssim & \; O(|\rho-1| + |s| +\| \hat u-e_1\|_2).
 \end{align*}
 It follows that if $\|u-e_1\|_2$ is sufficiently small, we then have
 \begin{align*}
 \sum_{i,j=1}^n \xi_i \xi_j (\partial_{ij} h)(u)
 \ge \frac {\gamma_0} 2 |\xi|^2, 
 \qquad\forall\, \xi \in \mathbb S^{n-1}.
 \end{align*}
\end{proof}

\subsection{Near the global minimizer:  strong convexity} 
$\,$

\vspace{0.2cm}

In this section we show strong convexity of the loss function $f(u)$ near the global
minimizer $u=\pm e_1$.

\begin{thm}[Strong convexity near the global minimizer] \label{Sep2e1}
There exist $0<\epsilon_0\ll 1$ and  positive constants $\beta_1, \gamma $ such that if
$m\gtrsim n$, then the following hold with probability at least $1- \frac {\beta_1}{m^2}$:

\begin{enumerate}
\item If $\|u -e_1\|_2\le \epsilon_0$, then 
\begin{align*}
\sum_{i,j=1}^n \xi_i \xi_j (\partial_{ij} f)(u) \ge \gamma>0, \qquad \forall\, \xi \in \mathbb
S^{n-1}.
\end{align*}
\item If $\|u +e_1\|_2\le \epsilon_0$, then 
\begin{align*}
\sum_{i,j=1}^n \xi_i \xi_j (\partial_{ij} f)(u) \ge \gamma>0, \qquad \forall\, \xi \in \mathbb
S^{n-1}.
\end{align*}
\end{enumerate}
In other words, $f(u)$ is strongly convex in a sufficiently small neighborhood of $\pm e_1$.
\end{thm}
\begin{proof}
Recall 
\begin{align*}
f(u) = 2 f_0(u) + \frac 1m \sum_{k=1}^m( (a_k\cdot u)^2 +2\beta |u|^2) ,
\end{align*}
where 
\begin{align*}
f_0(u) =- \frac1m \sum_{k=1}^m \sqrt {\beta |u|^2+ (a_k\cdot u)^2}
\cdot \sqrt{\beta |u|^2+(a_k\cdot e_1)^2}.
\end{align*}
Clearly 
\begin{align*}
\sum_{i,j=1}^n \xi_i \xi_j \partial_{ij} f(u)
=2\Bigl( \frac 1m \sum_{k=1}^m |a_k \cdot \xi |^2 \Bigr) 
+ 4\beta|\xi |^2 +2 \sum_{i,j=1}^n \xi_i \xi_j \partial_{ij} f_0(u). 
\end{align*}
Obviously we have for $m\gtrsim n$, it holds with high probability that
\begin{align*}
\Bigl| \frac 1m \sum_{k=1}^m |a_k\cdot \xi|^2 
- 1 \Bigr| \le \frac{\epsilon}{100}, \qquad\forall\, \xi \in \mathbb S^{n-1}.
\end{align*}
By Lemma \ref{Sp1_1}, we have
\begin{align*}
\Bigl| \sum_{i,j=1}^n \xi_i \xi_j (\partial_{ij} f_0(u) -
\mathbb E \partial_{ij} f_0(u) )
\Bigr| \le \frac{\epsilon}{100}, \qquad\forall\, \xi \in \mathbb S^{n-1},
\;\forall\,  \frac 13 \le \|u \|_2 \le 3.
\end{align*}
Thus we have
\begin{align*}
\Bigl| \sum_{i,j=1}^n \xi_i \xi_j (\partial_{ij} f(u) -
\mathbb E \partial_{ij} f(u) )
\Bigr| \le \frac{\epsilon}{100}, \qquad\forall\, \xi \in \mathbb S^{n-1},
\;\forall\,  \frac 13 \le \|u \|_2 \le 3.
\end{align*}
The desired result then follows from Theorem \ref{thA31e50} by taking 
$\epsilon>0$ sufficiently small. 
\end{proof}
We now complete the proof of the main theorem.

\begin{proof}[Proof of Theorem \ref{thmC}]
We proceed in several steps. 
\begin{enumerate}
\item By Theorem \ref{thmAu28_1}, we see that
with high probability the function
$f(u)$ has non-vanishing gradient in the regimes
\begin{align*}
0<\|u\|_2 \le \frac {\sqrt{\beta}} {4(1+\beta)}=c_1
\end{align*}
and 
\begin{align*}
\|u\|_2 \ge 3 (1+\beta)=c_2.
\end{align*}
Moreover at the point $u=0$, we have the directional gradient
is strictly less than $-\sqrt{\beta}$ along any direction
$\xi \in \mathbb S^{n-1}$.

\item By Theorem \ref{Sep2e1}, there exists $\epsilon_0>0$ sufficiently
small, such that with probability at least
$1-O(m^{-2})$, $f(u)$ is strongly convex in the neighborhood
$\|u\pm e_1\|_2\le \epsilon_0$.  

\item By Theorem \ref{thA31_0a}, we have that  with high
probability
\begin{align*}
\| \nabla f\|_2 >0, 
\end{align*}
if $|\rho -1|\ge c(\eta_0)$  and $| |\hat u \cdot e_1| -1| \le \eta_0$. 
Here we recall $\rho =\|u\|_2$ and $u=\rho \hat u$.  Observe that
\begin{align*}
\| u\pm e_1\|_2^2 = (\rho-1)^2 + 2\rho (1\pm \hat u\cdot e_1).
\end{align*}
By taking $\eta_0=\epsilon_0^2/100$, we see that 
$\|u\pm e_1\|_2 >\epsilon_0$, $||\hat u\cdot e_1| -1|\le \eta_0$ must
imply
\begin{align*}
|\rho-1| >\frac {\epsilon_0}{10}.
\end{align*}
Thus it remains for us to treat the regime $\| |\hat u\cdot e_1|-1
|>\eta_0$, $c_1\le \|u\|_2\le c_2$. 

\item In the regime $||\hat u\cdot e_1|-1|>\eta_0$, $\|u\|_2\sim 1$,
we have by Theorem \ref{thmA31_1}, with high probability it holds that
either the function has a non-vanishing gradient at the point $u$, or the
gradient vanishes at $u$, but $f$ has a negative directional curvature
at this point.
\end{enumerate}
\end{proof}

\section{perturbed amplitude model \RNum{2}}  \label{S:model3}
In this section, we introduce the second perturbed amplitude model for solving phase retrieval problem and consider the global 
landscape of it. Specifically, we consider the following empirical loss for some parameter $\beta>0$, 
\begin{align} \label{model3e1}
f(u) &= \frac 1 m \sum_{j=1}^m \Bigl( \;\sqrt{\beta |u|^2+(a_j\cdot u)^2 +(a_j\cdot x)^2}
-\sqrt{ \beta |u|^2 +2 (a_j\cdot x)^2}\;\;  \Bigr)^2.
\end{align}

\begin{thm} \label{thmD}
Let $0<\beta <\infty$. 
Assume 
$\{a_i\}_{i=1}^m$ are i.i.d.  standard Gaussian random vectors and $x\ne 0$. 
There exist positive constants $c$, $C$ depending only on $\beta$, such that if $m\ge C n  $, then
with probability at least $1- e^{-cm} $ the loss function $f=f(u)$ 
defined by \eqref{model3e1}
has no spurious local minimizers. The only global minimizer is
$\pm  x$, and the loss function is strongly convex in a neighborhood of $\pm x$.
The point $u=0$ is a local maximum point with strictly negative-definite Hessian.
All other critical points are strict saddles, i.e., each saddle point has a neighborhood
where the function has negative directional curvature.
\end{thm}
\begin{remark}
One should note that the set  $\bigcup_{j=1}^m \{ a_j\cdot x=0 \}$ 
has measure zero. Therefore for a typical realization, $a_j\cdot x$ is always non-zero for all $j$ and
the function 
\begin{align*}
\tilde f_j (y) =\sqrt{y^2+(a_j\cdot x)^2}
\end{align*}
is smooth. In particular, we can compute (for each realization) the derivatives of
the summands in \eqref{model3e1} without any problem. 
\end{remark}

\begin{remark}
Thanks to the regularization term $(a_j\cdot x)^2$, our new model \eqref{model3e1} enjoys
a better probability concentration bound $1-e^{-cm}$ than the model \eqref{model2e1}
where the weaker probability concentration $1-O(m^{-2})$ is proved. 
\end{remark}

Wthout loss of generality we shall assume $x=e_1$ throughout the rest of the proof.

\subsection{The regimes $\|u\|_2\ll 1 $ and $\|u\|_2\gg 1$ are fine}$\;$

Write  $u=\rho \hat u$ where $\hat u \in S^{n-1}$.   Then
\begin{align*}
&\Bigl(\; \sqrt{\beta |u|^2 + (a_j\cdot u)^2+|a_j\cdot e_1|^2 }- \sqrt{\beta |u|^2 +2(a_j\cdot e_1)^2} 
\;\; \Bigr)^2
\notag \\
=\;&\;   \rho^2 ( (a_j\cdot \hat u)^2 + 2 \beta) +3(a_j\cdot e_1)^2
-2  \sqrt{\beta \rho^2 + \rho^2(a_j\cdot \hat u)^2+(a_j\cdot e_1)^2} \sqrt{ \beta \rho^2 +2(a_j\cdot e_1)^2}.
\end{align*}
Thus, the derivative of $f$ is
\begin{eqnarray}
\partial_{\rho} f &= & 2\rho \frac 1m \sum_{j=1}^m \Bigl(
( (a_j \cdot \hat u)^2 +2\beta) -\frac{\beta+(a_j\cdot \hat u)^2} {\sqrt{\beta \rho^2+\rho^2(a_j\cdot \hat u)^2
+(a_j\cdot e_1)^2} }\sqrt{\beta \rho^2 + 2(a_j\cdot e_1)^2}  \notag \\
&&- \sqrt{\beta \rho^2+ \rho^2(a_j \cdot \hat u)^2+(a_j\cdot e_1)^2} \frac { \beta}
{\sqrt{\beta \rho^2 + 2(a_j\cdot e_1)^2}} \Bigr). \label{mod2partf}
\end{eqnarray}

\begin{lem}[The regime $\rho\gg 1$ is OK] \label{Sep3_e1}
There exist constants $R_1=R_1(\beta)>0$, $d_1=d_1(\beta)>0$ such that the 
following hold:  For $m\gtrsim n$,  with high probability it holds that
\begin{align*}
\partial_{\rho} f \ge d_1 \rho ,  \quad\forall\, \rho \ge R_1 , \;\text{}\forall\, \hat u \in \mathbb S^{n-1}.
\end{align*}
\end{lem}
\begin{proof}
We only sketch the proof.  Denote $X_j =a_j\cdot e_1$ and $Z_j= a_j \cdot \hat u$. 
We next gives several estimation bounds for the terms of $\partial_\rho f$. We first establish an upper bound for the second term.
Before proceeding,  observe that 
\begin{eqnarray*}
    \frac{\beta \rho^2 +2X_j^2} { (\beta +Z_j^2) \rho^2 +X_j^2}
   - \frac {\beta}{\beta +Z_j^2} & =& \frac {\beta X_j^2 + 2 X_j^2 Z_j^2} { \Bigl( (\beta +Z_j^2) \rho^2 +X_j^2 \Bigr)
\cdot (\beta +Z_j^2) } \notag \\
& \le & \frac {2X_j^2} {(\beta +Z_j^2) \rho^2 +X_j^2} \\
& \le &  \frac 1 {\rho^2} \cdot \frac {2X_j^2} {\beta +Z_j^2},
\end{eqnarray*}
which means
\begin{align*}
(\beta +Z_j^2) \cdot \frac {\sqrt{\beta \rho^2+2X_j^2} }{\sqrt{\rho^2(\beta +Z_j^2) +X_j^2}}
  & \le (\beta +Z_j^2) \Bigl( \sqrt{\frac {\beta}{\beta +Z_j^2} }+
  \frac 1 {\rho} \frac {\sqrt 2 |X_j|} {\sqrt{\beta +Z_j^2} } \Bigr) \notag \\
  & \le \sqrt{\beta} \sqrt{\beta+Z_j^2} + \frac 1 {\rho}
  \sqrt{2} |X_j| \sqrt{\beta +Z_j^2},
\end{align*}
where we use the fact $\sqrt{a+b} \le \sqrt{a}+\sqrt{b}$ for any positive number $a,b$ in the first inequality.
For the third term, it is easy to see that
\begin{align*}
\sqrt{\rho^2(\beta+Z_j^2)+X_j^2}
\cdot \frac {\beta}{\sqrt{\beta \rho^2+2X_j^2}} \le \sqrt{\beta}
\cdot \sqrt{\beta+Z_j^2}.
\end{align*}
Putting the above two estimators into \eqref{mod2partf}, we get
\begin{align*}
\frac 1{2\rho} \partial_{\rho} f
&\ge \frac 1m \sum_{j=1}^m \Bigl(  Z_j^2+2\beta
-2\sqrt{\beta} \sqrt{\beta+Z_j^2} \Bigr)
-\frac 1 {\rho} \sqrt{2} \frac 1m \sum_{j=1}^m 
|X_j| \cdot \sqrt{\beta+Z_j^2} \notag \\
& \ge \frac 1m \sum_{j=1}^m (\sqrt{Z_j^2+\beta}-\sqrt{\beta})^2
-\frac 1 {\rho} \cdot \frac 1m \sum_{j=1}^m (X_j^2+Z_j^2+\beta).
\end{align*}
By Bernstein's inequality and simple union bound arguments, we clearly have  with high probability, 
\begin{align*}
&\Bigl|  \frac 1m \sum_{j=1}^m (\sqrt{Z_j^2+\beta}-\sqrt{\beta})^2
-\operatorname{mean} \Bigr| \le \epsilon,
\quad \forall\, \hat u \in \mathbb S^{n-1}; \\
&\frac 1m \sum_{j=1}^m (X_j^2+Z_j^2+\beta)
\le 3+\beta, \qquad \forall\, \hat u \in \mathbb S^{n-1}.
\end{align*}
The desired result then clearly follows.
\end{proof}

\begin{lem}[The regime $\|u\|_2 \ll 1 $ is OK] \label{Sep3_e2}
There exist constants $R_2=R_2(\beta)>0$, $d_2=d_2(\beta)>0$ such that the 
following hold: 
For $m\gtrsim n$,  with high probability it holds that
\begin{align*}
\partial_{\rho} f \le -d_2 \rho<0 ,  \quad\forall\, 0<\rho \le R_2 , \;\text{}\forall\, \hat u \in \mathbb S^{n-1}.
\end{align*}
Moreover, at $u=0$, we have $\nabla f (0)=0$, and
\begin{align*}
\sum_{k,l=1}^n  \xi_k \xi_l (\partial_{kl} f)(0)  \le - d_2 \| \xi\|_2^2,
\qquad\forall\, \xi \in \mathbb S^{n-1}.
\end{align*}
In yet other words, $u=0$ is a strict local maximum point with strictly negative definite Hessian. 

\end{lem}
\begin{proof}
We only sketch the proof.  Again denote $X_j =a_j\cdot e_1$ and $Z_j= a_j \cdot \hat u$. 
Observe that 
\begin{align*}
 \frac{\beta \rho^2 +2X_j^2} {\rho^2 (\beta +Z_j^2) +X_j^2}
=2 - \frac {\rho^2 (\beta+2Z_j^2)} {\rho^2(\beta+Z_j^2) +X_j^2}
\end{align*}
and 
\begin{align*}
&(\beta+Z_j^2) 
\sqrt{ \frac{\beta \rho^2 +2X_j^2} {\rho^2 (\beta +Z_j^2) +X_j^2}}  \ge  \sqrt{2} (\beta +Z_j^2)
- (\beta+Z_j^2) \sqrt{ 
\frac {\rho^2(\beta+2Z_j^2)} {\rho^2(\beta+Z_j^2) +X_j^2} }.
\end{align*}
On the other hand,
\begin{align*}
\frac { \rho^2(\beta+Z_j^2) +X_j^2}
{\beta \rho^2 +2X_j^2}
-\frac 12 = \frac 12 \cdot \frac {\rho^2(\beta+2Z_j^2)}
{\beta \rho^2 +2 X_j^2} \ge 0.
\end{align*}
Thus 
\begin{align*}
\frac 1 {2\rho}
\partial_{\rho} f
&\le \frac 1m \sum_{j=1}^m (Z_j^2 +2\beta -\sqrt{2}\cdot ({\beta+Z_j^2})
-\frac 1 {\sqrt 2} \beta)   + \frac 1m \sum_{j=1}^m 
(\beta+Z_j^2) \sqrt{ 
\frac {\rho^2(\beta+2Z_j^2)} {\rho^2(\beta+Z_j^2) +X_j^2} }.
\end{align*}
Since 
\begin{align*}
Z_j^2 +2\beta -\sqrt{2}\cdot ({\beta+Z_j^2})
-\frac 1 {\sqrt 2} \beta 
= -(\sqrt 2-1) Z_j^2 - (\sqrt 2+ \frac 1 {\sqrt 2} -2) \beta,
\end{align*}
the first summand clearly gives a nontrivial negative lower bound. 
The desired result then follows from Lemma \ref{ApMod3E1}.
We note that the result for $u=0$ follows by taking $u= t \xi$ and re-run the above
argument taking $t\to 0+$.  
\end{proof}

\begin{thm}[The regimes $\|u\|_2\ll 1$
and $\|u\|_2\gg 1$ are OK] \label{thmSep3_1}
For $m\gtrsim n$, with high probability the following hold:
\begin{enumerate}
\item We have 
\begin{align*}
&\partial_{\rho} f \ge d_1 \rho>0 , \qquad\forall\, \rho\ge R_1, 
\quad \forall\, \hat u \in \mathbb S^{n-1};\\
&\partial_{\rho} f\le -d_2 \rho<0, \qquad\forall\, 0<\rho\le R_2,\quad\forall\, \hat u
\in \mathbb S^{n-1},
\end{align*}
where $d_1$, $d_2$, $R_1$, $R_2$ are constants depending only on $\beta$. 
\item The point $u=0$ is a local maximum point with strictly negative-definite
Hessian, 
\begin{align*}
\sum_{k,l=1}^n  \xi_k \xi_l (\partial_{kl} f)(0) \le -d_2 <0, 
\qquad \forall\, \xi \in \mathbb S^{n-1}. 
\end{align*}
\end{enumerate}
\end{thm}
\begin{proof}
This follows from Lemmas \ref{Sep3_e1} and \ref{Sep3_e2}.
\end{proof}

\begin{thm}[The regime $\|u\|_2 \sim 1 $, $||\hat u\cdot e_1|-1|\le \epsilon_0$,
 $|\|u\|_2-1|\ge c(\epsilon_0)$ is OK] \label{Sep3_e3}
Let $R_1$, $R_2$ be the same as in Lemma \ref{Sep3_e1} and \ref{Sep3_e2}.
Let $0<\epsilon_0\ll 1$ be given and consider the regime $ \Bigl | |\hat u \cdot e_1|-1
\Bigr| \le \epsilon_0$ with $R_1 \le \| u\|_2 \le R_2$.  There exists 
a constant $c(\epsilon_0)>0$ ($c(\epsilon_0)$ also depends on $\beta$ but
we suppress this dependence) 
which tends to zero as $\epsilon_0\to 0$ such that the following
hold:
For $m\gtrsim n$,  with high probability it holds that
\begin{align*}
&\partial_{\rho} f <0,  \quad\forall\, R_2\le \rho \le 1-c(\epsilon_0) , \;\;\text{}\forall\, \hat u \in \mathbb S^{n-1}\quad  \text{with}\quad  | |\hat u \cdot e_1|-1| \le \epsilon_0;\\
&\partial_{\rho} f >0,  \quad\forall\, 1+ c(\epsilon_0) \le  \rho \le R_1,\; \;\text{}\forall\, \hat u \in \mathbb S^{n-1} \quad \text{with} \quad | |\hat u \cdot e_1| -1| \le \epsilon_0.
\end{align*}
\end{thm}
\begin{proof}
We shall work with the variable $R=\rho^2$. Write
\begin{align*}
f(u) = \frac 1m \sum_{j=1}^m F(\rho^2, (a_j\cdot \hat u)^2, (a_j\cdot e_1)^2),
\end{align*}
where
\begin{align*}
F(R, s, t) = R( s+ 2\beta) + 3 t -
2 \sqrt{ R(\beta+s) +t} \sqrt{ \beta R + 2 t}.
\end{align*}
Observe that 
\begin{align}
\partial_{R} F
&= s+2\beta - \Bigl( \beta
\sqrt{ \frac {R(\beta+s)+t} {\beta R+2t} }
+ (\beta+s) \sqrt{ \frac {\beta R+2t} {R(\beta+s) +t} } \Bigr) \label{otc2.08tpp01} \\
&=s+2\beta -F_0(\sqrt{z(R,s,t)} ), \notag 
\end{align}
where
\begin{align*}
F_0(y)= \beta y + (\beta+s) y^{-1}, \quad z(R,s,t)=
{  \frac {R(\beta+s)+t} {\beta R+2t} }.
\end{align*}
Note that $F_0^{\prime}(y) <0$ for any $0<y <\sqrt{ \frac {\beta +s} {\beta} }$. 
It is easy to check that for $t>0$, $s\ge 0$  (note that $t=(a_j\cdot e_1)^2>0$ for all $j$ 
almost surely)
\begin{align*}
z(R,s,t) < \frac {\beta +s} {\beta}.
\end{align*}
Furthermore, since
\begin{align*}
z(R,s,t)=\frac{ \beta +s}{\beta} - \frac {t \cdot \frac{\beta+2s} {\beta} } {\beta R+2t},
\end{align*}
we have $\partial_R z(R,s,t)>0$ for $R>0$, $t>0$, $s\ge 0$. Thus
\begin{align} \label{Sep3_e3.0}
\partial_{RR} F >0, \qquad\forall\, R>0, \, \forall\, s\ge 0, t>0.
\end{align}
On the other hand, by directly using \eqref{otc2.08tpp01},
it is not difficult to check that for $R\sim 1$, 
\begin{align} \label{Sep3_e3.1}
|(\partial_R F)(R,s,t) -(\partial_R F)(R, t, t) | \lesssim |s-t|.
\end{align}
Also observe that for $R\sim 1$, we have 
\begin{align}
&z(R, t, t) \sim 1, \quad \partial_R z (R, t, t) \sim \frac {t} {1+t}; \notag \\
&  -F_0^{\prime} (\sqrt{z(R,t,t)} ) 
=-\beta + \frac {\beta +t} {z(R,t,t)} = \frac 1 {z(R,t,t)} \cdot 
\frac {t (\beta+2t)} {\beta R+2t} \sim t; \notag \\
& \partial_{RR} F(R,t,t)=-F_0^{\prime}(\sqrt {z(R,t,t)})
\frac 12 z(R,t,t)^{-\frac 12} \cdot \partial_R 
z(R,t,t) \sim \frac {t^2} {1+t}. \label{Sep3_e3.2}
\end{align}
Note that $z(1, t,t)=1$ and $\partial_R F(1,t,t)=0$. 
By using \eqref{Sep3_e3.0}, \eqref{Sep3_e3.1} and \eqref{Sep3_e3.2}, we obtain
for $R\ge 1+\eta_0$ ($0<\eta_0\ll 1$ will be specified later)
\begin{align*}
(\partial_R F)(R, s,t ) & > (\partial_R F)(1+\eta_0, s,t) \notag \\
& \ge (\partial_R F)(1+\eta_0, t, t) - B_1 |s-t| \notag \\
& \ge B_2 \cdot \frac {t^2}{1+t} \cdot \eta_0 - B_1 |s-t|,
\end{align*}
where $B_1>0$, $B_2>0$ are constants depending only on $\beta$.
Consequently we have for $R\ge 1+\eta_0$,
\begin{align*}
\partial_R f \ge  B_2 \eta_0 \frac 1m \sum_{j=1}^m \frac {(a_j\cdot e_1)^4}
{1+(a_j\cdot e_1)^2} - B_1 \frac 1m \sum_{j=1}^m | (a_j\cdot \hat u)^2
-(a_j\cdot e_1)^2|.
\end{align*}
Clearly with high probability,
\begin{align*}
\frac 1m \sum_{j=1}^m \frac {(a_j\cdot e_1)^4}
{1+(a_j\cdot e_1)^2} & \ge B_3>0, \\
 \frac 1m \sum_{j=1}^m | (a_j\cdot \hat u)^2
-(a_j\cdot e_1)^2|
&\le \frac 1m \sum_{j=1}^m | a_j\cdot (\hat u- e_1)|
\cdot |a_j \cdot (\hat u +e_1) | \\
&\le 
B_4 \min \{ \|\hat u -e_1\|_2, \| \hat u+e_1\|_2 \}, \qquad\forall\, \hat u \in
\mathbb S^{n-1}, 
\end{align*}
where $B_3>0$, $B_4>0$ are absolute constants.  Clearly then
for $R\ge 1+\eta_0$, we have (below $B_5>0$, $B_6>0$ are
 constants depending only on $\beta$)
\begin{align*}
\partial_R f \ge B_5 \eta_0 - B_6 \sqrt{ 1-|\hat u \cdot e_1|}>0
\end{align*}
if $\eta_0$ is chosen suitably small.  The case for $R\le 1-\eta_0$ is similar.
We omit the details.
\end{proof}
\subsection{Analysis of the regime $\rho \sim 1$, $||\hat u \cdot e_1|-1|\ge
\epsilon_0>0$} 

$\;$

\vspace{0.2cm}
In this section we consider the regime $\rho \sim 1$, 
$ |\hat u \cdot e_1| < 1-\epsilon_0$, where $0<\epsilon_0\ll 1$. 
To simplify the discussion, we use the coordinate system
\begin{align*}
\hat u &= (\hat u \cdot e_1) e_1 +\tilde u, \notag \\
 &= t e_1+ \sqrt{1-t^2} e^{\perp}, 
 \end{align*}
 where $e^{\perp} \in \mathbb S^{n-1}$ satisfies $e^{\perp} \cdot e_1=0$. 
Clearly in the regime $\rho\sim 1$, $|t|<1$, we have a smooth representation
\begin{align*}
u = \rho \hat u = \rho\cdot (t e_1+ \sqrt{1-t^2} e^{\perp} ) = \psi( \rho, t, e^{\perp}).
\end{align*}
By Proposition \ref{Au29e1}, we can simplify the computation by examining only
at the derivatives $\partial_t $ and $\partial_{tt}$.  We shall use these in the regime
$|t| < 1-\epsilon_0$ where $0<\epsilon_0\ll 1$. 
Now observe
 \begin{align*}
 \frac 12\mathbb E f = \frac 12 (1+2\beta) \rho^2 +\frac 32- 
 \mathbb E  \sqrt{ \beta\rho^2 + \rho^2(a\cdot \hat u)^2
 +(a\cdot e_1)^2 } \sqrt{\beta \rho^2 + 2(a\cdot e_1)^2},
 \end{align*}
 where $a\sim \mathcal N (0, \operatorname{I_n})$. 
Denote $X_1= a\cdot e_1$ and $Y_1= a\cdot e^{\perp}$ so that
$a\cdot \hat u = t X_1 +\sqrt{1-t^2} Y_1=: X_t$. 
We focus on the term 
\begin{align*}
 &\mathbb E  \sqrt{ \beta \rho^2 + \rho^2(a\cdot \hat u)^2 +(a\cdot e_1)^2} \sqrt{\beta \rho^2 + 2(a\cdot e_1)^2} \\
=&\; \mathbb E \sqrt{\beta \rho^2+\rho^2X_t^2+X_1^2} \sqrt{\beta \rho^2+2X_1^2} =:\, h_{\infty}(\rho, t). 
\end{align*}
\begin{lem}[The limiting profile] \label{lemSep4_2a}
For any $0<\eta_0\ll 1$, the following hold:

\begin{enumerate}
\item $\sup_{|t| \le 1-\eta_0, \rho \sim 1} ( |\partial_t h_{\infty}(\rho,t)|
+ |\partial_{tt} h_{\infty} (\rho, t) | + |\partial_{ttt} h_{\infty}(\rho, t) |)
\lesssim 1$. \vspace{1em}
\item $|\partial_t h_{\infty}(\rho, t)| \gtrsim |t|$ for $0<|t|<1$, $\rho\sim 1$. \vspace{1em}
\item $\partial_{tt} h_{\infty}(\rho, t)\gtrsim 1$ for $|t| \ll 1$, $\rho\sim 1$. 
\end{enumerate}
\end{lem}
\begin{proof}
See appendix. 
\end{proof}

\begin{thm}[The regime $\|u\|_2\sim 1$, $||\hat u \cdot e_1|-1|>\eta_0$ is fine]
\label{thmSep4_1a}
For any given $0<\eta_0 \ll 1$ and $0<c_1<c_2<\infty$, if $m\gtrsim n$,
then the following hold with high probability:
In the regime $c_1< \rho =\|u\|_2< c_2$, $| |\hat u \cdot e_1|-1 |
> \epsilon_0$, there are only two possibilities:
\begin{enumerate}
\item  $\| \nabla f (u)\|_2 >0$;
\item $\nabla f(u)=0$, and $f$ has a negative directional curvature at this point.
\end{enumerate}
\end{thm}
\begin{proof}
Denote  $X_j=a_j\cdot e_1$ and 
\begin{align*}
g(\rho, t, e^{\perp})
&=2\beta \rho^2 + \rho^2 \frac 1m \sum_{j=1}^m  (a_j\cdot \hat u)^2  
-2 \frac 1m \sum_{j=1}^m 
 \sqrt{\beta \rho^2 + \rho^2(a_j\cdot \hat u)^2+X_j^2} \sqrt{ \beta \rho^2 +2X_j^2} 
 \notag \\
 &=: 2\beta \rho^2 + \rho^2 h_0(\rho, t,e^{\perp})
 -2h(\rho, t, e^{\perp} ),
\end{align*}
where 
\begin{align*}
a_j\cdot \hat u= t X_j +\sqrt{1-t^2} Y_j,
\quad X_j=a_j\cdot e_1, \; Y_j=a_j\cdot e^{\perp}.
\end{align*}
By the same argument as in the proof of Theorem \ref{thmA31_1}, we have
 for any small $\epsilon>0$ and $m\gtrsim n$, it holds with high
probability that
\begin{align*}
|\partial_t h_0 - \mathbb E \partial_t h_0|
+|\partial_{tt} h_0 - \mathbb E \partial_{tt} h_0|
\le \epsilon, \qquad\forall\, |t| \le 1-\epsilon_0, e^{\perp}\cdot e_1=0,
e^{\perp} \in \mathbb S^{n-1}.
\end{align*}
Note also $\mathbb E \partial_t h_0=0$ and $\mathbb E \partial_{tt} h_0=0$.
By Lemma \ref{Sep5a_2},  
for any small $\epsilon>0$ and $m\gtrsim n$, it also holds with high
probability that
\begin{align*}
& |\partial_t h - \mathbb E \partial_t h| +|\partial_{tt} h
-\mathbb E \partial_{tt} h|
\le \epsilon, \qquad\forall\, |t| \le 1-\epsilon_0, e^{\perp}\cdot e_1=0,
e^{\perp} \in \mathbb S^{n-1}, \forall\, c_1\le \rho \le c_2.
\end{align*}
We then obtain for small $\epsilon>0$, if $m\gtrsim n$, it holds 
with high probability  that
\begin{align*}
&|\partial_t g -\mathbb E \partial_t g | \le 2 \epsilon; \\
& \partial_{tt} g \le 
\mathbb E \partial_{tt} g +2 \epsilon.
\end{align*}
Clearly
\begin{align*}
&\mathbb E \partial_t g = -  \partial_t h_{\infty}(\rho, t) ;\\
&\mathbb E \partial_{tt} g = - \partial_{tt} h_{\infty}(\rho, t).
\end{align*}
By Lemma \ref{lemSep4_2a}, we can take $t_0\ll 1$ such that
\begin{align*}
&\partial_{tt} h_{\infty} (\rho, t) \ge \epsilon_1>0, \qquad\forall\, |t|\le t_0, 
c_1\le \rho \le c_2; \\
&|\partial_t h_{\infty} (\rho, t)| \ge \epsilon_2>0, \qquad\forall\, t_0\le |t|\le 1-\epsilon_0,
c_1\le \rho\le c_2.
\end{align*}
By taking $\epsilon>0$ sufficiently small, we can then guarantee that
\begin{align*}
& |\partial_t g| >\epsilon_3>0, 
\qquad\forall\, |t|\le t_0, 
c_1\le \rho \le c_2; \\
&\partial_{tt} g  \le -\epsilon_4<0, \qquad\forall\, t_0\le |t|\le 1-\epsilon_0,
c_1\le \rho\le c_2.
\end{align*}
The desired result then follows from Proposition \ref{Au29e1}. 
\end{proof}

\subsection{Strong convexity near the global minimizers $u=\pm e_1$: analysis
of the limiting profile}
$\;$

\vspace{0.2cm}

In this section we shall show that in the small neighborhood of $u=\pm e_1$
where
\begin{align*}
| |\hat u \cdot e_1| -1| \ll 1, \quad  |\rho-1| \ll 1,
\end{align*}
 that the Hessian of the expectation of the loss function 
must be strictly positive definite. In yet other words $\mathbb E f$ must be strictly convex
in this neighborhood so that $u=\pm e_1$ are the unique minimizers. 
To this end consider
\begin{align} \label{Sep6e71}
h(u) =\frac 12 ( \mathbb E f -3)= \frac 12 (1+2\beta) \rho^2 - 
 \mathbb E  \sqrt{ \beta\rho^2 + \rho^2(a\cdot \hat u)^2
 +(a\cdot e_1)^2 } \sqrt{\beta \rho^2 + 2(a\cdot e_1)^2},
\end{align}
where $a\sim \mathcal N (0, \operatorname{I}_n)$. 
\begin{thm}[Strong convexity of $\mathbb E f$ when $\| u\pm e_1\| \ll 1$]
\label{thSep6e50}
Consider $h$ defined by \eqref{Sep6e71}.
There exist $0<\epsilon_0\ll 1$ and a positive constant $\gamma_1$ such that the following hold:
\begin{enumerate}
\item If $\| u-e_1\|_2 \le \epsilon_0$, then for any $\xi \in \mathbb S^{n-1}$, we have
\begin{align*}
\sum_{i,j=1}^n \xi_i \xi_j 
(\partial_i \partial_j h)(u) \ge \gamma_1 >0.
\end{align*}
\item If $\| u+ e_1\|_2 \le \epsilon_0$, then for any $\xi \in \mathbb S^{n-1}$, we have
\begin{align*}
\sum_{i,j=1}^n \xi_i \xi_j 
(\partial_i \partial_j h)(u) \ge \gamma_1 >0.
\end{align*}

\end{enumerate}

\end{thm}
\begin{proof}
We shall only consider the case $\| u -e_1\|_2\ll 1$. The other case $\| u +e_1\|_2\ll 1$
is similar and therefore omitted. Note that
\begin{align*}
\| u -e_1\|_2^2 = \| \rho \hat u -e_1\|_2^2
= (\rho-1)^2 +2\rho (1- t) \le \epsilon_0^2,
\end{align*}
where $t=\hat u \cdot e_1$.  Thus for $0<\epsilon_0\ll 1$, we have
\begin{align*}
|\rho -1| \le  \epsilon_0, \qquad    1- {\epsilon_0^2} \le t \le 1.
\end{align*}
We  make a  change of variable and write
(recall $1-\epsilon_0^2 \le \hat u \cdot e_1 \to 1$),
\begin{align*}
\hat u = \sqrt{1-s^2} e_1 + s e^{\perp}, \qquad e^{\perp} \cdot e_1 =0, \, e^{\perp}
\in \mathbb S^{n-1}, 
\end{align*}
where we assume $0\le s \ll 1$. Note that $s= \frac {|u^{\prime}|}{\rho} $, and
$u^{\prime}=u-(u\cdot e_1) e_1= (0, u_2,\cdots,u_n)$. 
To calculate $\partial^2 h$ we need to compute the Hessian expressed in the
$(\rho, s)$ coordinate. It is not difficult to check that by \eqref{Sep6e71},
the value of $h(u)$  depends only on $(\rho, s)$. 
Thus by a slight abuse of notation
we write  $h(u)= h(|u|,\, \frac{|u^{\prime}|} {|u|} ) =h(\rho,s)$ (we denote $|u|=\|u\|_2$,
$|u^{\prime}|=\|u^{\prime}\|_2$) and compute
(below we assume $s>0$ so that $|u^{\prime}|>0$)
\begin{eqnarray*}
\sum_{i,j=1}^n \xi_i \xi_j \partial_{ij } h
&= &  \partial_{\rho\rho} h \cdot a^2 
+ 2 a \partial_{\rho s} h \cdot( - \frac{a s}{\rho} +\frac b {\rho}) + \partial_{\rho} h \cdot ( \frac{|\xi|^2- |\xi \cdot \hat u|^2} {\rho} ) + \partial_{ss} h \cdot ( \frac {a^2 s^2-2abs}{\rho^2} ) \\
&& + (\partial_{ss} h-\frac 1s \partial_s h) \frac{b^2} {\rho^2} + \partial_s h \cdot( -\frac{|\xi|^2}{\rho^2} s + 3 \frac{a^2 s}{\rho^2} - 2\frac{ab}{\rho^2}+\frac{|\xi^{\prime}|^2}
{\rho^2 s}  ).
\end{eqnarray*}
In the above computation, one does not need to worry about the formal singularity
caused by $\frac 1 s$. Since (by Lemma \ref{leSep6_61})
$\partial_s h(\rho, s=0)=0$ for any $\rho>0$, we write
\begin{align*}
(\partial_s h)(\rho, s) \cdot
\frac {1}s 
= \frac { (\partial_s h)(\rho, s) -(\partial_s h)(\rho, 0) } s
= \int_0^1 (\partial_{ss} h) (\rho, \theta s) d\theta, \quad s>0.
\end{align*}
In particular we have
\begin{align*}
&\lim_{s\to 0^{+}}
(\partial_s h)(\rho, s)  \cdot \frac 1s
 =( \partial_{ss} h) (\rho, 0) ;\\
 &\Bigl| (\partial_{ss} h)(\rho,s)- \frac 1 s \partial_s h_1 (\rho,s) \Bigr|
 = O(s) \to 0, \quad \text{as $s\to 0$}.
\end{align*}
By using this observation and Lemma \ref{leSep6_61}, we obtain
\begin{align*}
\sum_{i,j=1}^n \xi_i \xi_j (\partial_{ij} h)(e_1)
&=  (\partial_{\rho\rho} h)(1,0) \cdot a^2\Bigr|_{a=\xi_1}
+ (\partial_{ss} h)(1,0) \cdot |\xi^{\prime}|^2  \ge \gamma_0 \cdot |\xi|^2, \qquad \forall\, \xi \in \mathbb S^{n-1},
\end{align*}
where $\gamma_0>0$ is a constant. 
Now for $\| u -e_1\|_2 \ll 1$, by using Lemma \ref{leSep6_60} and
Lemma \ref{leSep6_61}, we have
\begin{align*}
 & \Bigl| \sum_{i,j=1}^n \xi_i \xi_j \Bigl( (\partial_{ij}h)(e_1)-
 (\partial_{ij} h)(u) \Bigr) \Bigr| \notag \\
 \lesssim & \; 
 |(\partial_{\rho\rho} h)(\rho, s) - (\partial_{\rho\rho} h)(1,0) |
+ |(\partial_{\rho\rho}h)(\rho,s)|\cdot | (\xi \cdot \hat u)^2- (\xi \cdot e_1)^2| \notag \\
&\quad + | (\partial_{\rho s} h) (\rho, s)| \cdot (1+s) 
 + |(\partial_{\rho} h)(\rho, s)| +|\partial_{ss} h(\rho,s)| \cdot (s+s^2) \notag \\
 & \qquad + | \partial_s h(\rho, s)|
 + \Bigl| \frac  1 {\rho^2} \int_0^1 (\partial_{ss} h) (\rho, \theta s) d\theta
 - (\partial_{ss} h)(1,0) \Bigr| \notag \\
 & \qquad + \Bigl| (\partial_{ss} h)(\rho,s) - (\partial_s h) (\rho,s) \cdot \frac 1 s\Bigr|
 \cdot \frac {b^2} {\rho^2} \notag \\
 \lesssim & \; O(|\rho-1| + |s| +\| \hat u-e_1\|_2).
 \end{align*}
 It follows that if $\|u-e_1\|_2$ is sufficiently small, we then have
 \begin{align*}
 \sum_{i,j=1}^n \xi_i \xi_j (\partial_{ij} h)(u)
 \ge \frac {\gamma_0} 2 |\xi|^2, 
 \qquad\forall\, \xi \in \mathbb S^{n-1}.
 \end{align*}
\end{proof}

\subsection{Near the global minimizer:  strong convexity} 
$\,$

\vspace{0.2cm}

In this section we show strong convexity of the loss function $f(u)$ near the global
minimizer $u=\pm e_1$.

\begin{thm}[Strong convexity near the global minimizer] \label{Sep6e1}
There exists $0<\epsilon_0\ll 1$ and a constant $\beta_1>0$ such that if
$m\gtrsim n$, then the following hold with high probability:

\begin{enumerate}
\item If $\|u -e_1\|_2\le \epsilon_0$, then 
\begin{align*}
\sum_{i,j=1}^n \xi_i \xi_j (\partial_{ij} f)(u) \ge \gamma>0, \qquad \forall\, \xi \in \mathbb
S^{n-1},
\end{align*}
where $\gamma$ is a constant.

\item If $\|u +e_1\|_2\le \epsilon_0$, then 
\begin{align*}
\sum_{i,j=1}^n \xi_i \xi_j (\partial_{ij} f)(u) \ge \gamma>0, \qquad \forall\, \xi \in \mathbb
S^{n-1},
\end{align*}
where $\gamma$ is a constant.
\end{enumerate}

In yet other words, $f(u)$ is strongly convex in a sufficiently small neighborhood of $\pm e_1$.
\end{thm}
\begin{proof}
Recall 
\begin{align*}
f(u) = 2 f_0(u) + \frac 1m \sum_{k=1}^m 
\Bigl((a_k\cdot u)^2 +2\beta |u|^2 +3(a_k\cdot e_1)^2\Bigr),
\end{align*}
where 
\begin{align*}
f_0(u) =- \frac1m \sum_{k=1}^m \sqrt {\beta |u|^2+ (a_k\cdot u)^2 +(a_k\cdot e_1)^2}
\cdot \sqrt{\beta |u|^2+2(a_k\cdot e_1)^2}.
\end{align*}
Clearly 
\begin{align*}
\sum_{i,j=1}^n \xi_i \xi_j \partial_{ij} f(u)
=2\Bigl( \frac 1m \sum_{k=1}^m |a_k \cdot \xi |^2 \Bigr) 
+ 4\beta|\xi |^2 +2 \sum_{i,j=1}^n \xi_i \xi_j \partial_{ij} f_0(u). 
\end{align*}
Obviously we have for $m\gtrsim n$, it holds with high probability that
\begin{align*}
\Bigl| \frac 1m \sum_{k=1}^m |a_k\cdot \xi|^2 
- 1 \Bigr| \le \frac{\epsilon}{100}, \qquad\forall\, \xi \in \mathbb S^{n-1}.
\end{align*}
By Lemma \ref{Sep6Sp1_1}, we have
\begin{align*}
\Bigl| \sum_{i,j=1}^n \xi_i \xi_j (\partial_{ij} f_0(u) -
\mathbb E \partial_{ij} f_0(u) )
\Bigr| \le \frac{\epsilon}{100}, \qquad\forall\, \xi \in \mathbb S^{n-1},
\;\forall\,  \frac 13 \le \|u \|_2 \le 3.
\end{align*}
Thus we have
\begin{align*}
\Bigl| \sum_{i,j=1}^n \xi_i \xi_j (\partial_{ij} f(u) -
\mathbb E \partial_{ij} f(u) )
\Bigr| \le \frac{\epsilon}{100}, \qquad\forall\, \xi \in \mathbb S^{n-1},
\;\forall\,  \frac 13 \le \|u \|_2 \le 3.
\end{align*}
The desired result then follows from Theorem \ref{thSep6e50} by taking 
$\epsilon>0$ sufficiently small. 
\end{proof}
We now complete the proof of the main theorem.

\begin{proof}[Proof of Theorem \ref{thmD}]
We proceed in several steps. 
\begin{enumerate}
\item By Theorem \ref{thmSep3_1}, we see that
with high probability the function
$f(u)$ has non-vanishing gradient in the regimes
\begin{align*}
0<\|u\|_2 \le R_2=R_2(\beta)
\end{align*}
and 
\begin{align*}
\|u\|_2 \ge R_1=R_1(\beta),
\end{align*}
where $R_1>0$, $R_2>0$ depend only on $\beta$.
Moreover the point $u=0$ is a local maximum point with strictly negative-definite
Hessian.

\item By Theorem \ref{Sep6e1}, there exists $\epsilon_0>0$ sufficiently
small, such that with high probability, $f(u)$ is strongly convex in the neighborhood
$\|u\pm e_1\|_2\le \epsilon_0$.  

\item By Theorem \ref{Sep3_e3}, we have that  with high
probability
\begin{align*}
\| \nabla f\|_2 >0, 
\end{align*}
if $|\rho -1|\ge c(\eta_0)$  and $| |\hat u \cdot e_1| -1| \le \eta_0$. 
Here we recall $\rho =\|u\|_2$ and $u=\rho \hat u$.  Observe that
\begin{align*}
\| u\pm e_1\|_2^2 = (\rho-1)^2 + 2\rho (1\pm \hat u\cdot e_1).
\end{align*}
By taking $\eta_0=\epsilon_0^2/100$, we see that 
$\|u\pm e_1\|_2 >\epsilon_0$, $||\hat u\cdot e_1| -1|\le \eta_0$ must
imply
\begin{align*}
|\rho-1| >\frac {\epsilon_0}{10}.
\end{align*}
Thus it remains for us to treat the regime $\| |\hat u\cdot e_1|-1
|>\eta_0$, $R_1 \le \|u\|_2\le R_2$. 

\item In the regime $||\hat u\cdot e_1|-1|>\eta_0$, $\|u\|_2\sim 1$,
we have by Theorem \ref{thmSep4_1a}, with high probability it holds that
either the function has a non-vanishing gradient at the point $u$, or the
gradient vanishes at $u$, but $f$ has a negative directional curvature
at this point.
\end{enumerate}
\end{proof}

\section{Numerical Experiments} \label{S:numerics}

In this section, we demonstrate the numerical efficiency of our estimators by simple gradient descent and compare their performance with other competitive algorithms.

In a concurrent work \cite{2020a}, we considered the following piecewise 
Smoothed Amplitude loss (SAF):
\begin{equation*} 
\min_{u \in \Rn} \quad f(u)=\frac{1}{2m}\sum_{j=1}^m \xkh{\gamma\xkh{\frac{\abs{a_j\cdot u }}{\abs{a_j\cdot x}}}-1}^2\cdot \abs{a_j\cdot x}^2
\end{equation*}
with the function $\gamma(t)$
\begin{equation*}
\gamma(t):=\left \{ \begin{array}{cl}
     \abs{t}, &  \abs{t} > \beta;  \\
    \frac{1}{2\beta} t^2+\frac{\beta}{2}, &  \abs{t} \leq \beta. 
\end{array} \right.
\end{equation*}
In this work, our first Perturbed Amplitude Model (PAM1) is
\[
\min_{u \in \Rn} \quad f(u) = \frac 1 m \sum_{j=1}^m \Bigl( \;\sqrt{\beta |u|^2+(a_j\cdot u)^2} -\sqrt{ \beta |u|^2 + (a_j\cdot x)^2}\;\;  \Bigr)^2.
\]
The second Perturbed Amplitude Model (PAM2) is
\[
\min_{u \in \Rn} \quad f(u) = \frac 1 m \sum_{j=1}^m \Bigl( \;\sqrt{\beta |u|^2+(a_j\cdot u)^2 +(a_j\cdot x)^2}-\sqrt{ \beta |u|^2 +2 (a_j\cdot x)^2}\;\;  \Bigr)^2.
\]

We have show theoretically that any gradient descent algorithm will not get trapped in a local minimum for the loss functions above. Here we present numerical experiments to show that the estimators perform very well with randomized initial guess.
   
   We use the following vanilla gradient descent algorithm 
\[
     u_{k+1}=u_{k}-\mu \nabla f(u_{k})
\]
with a random initial guess to minimize the loss function $f(u)$ given above. The pseudocode for the algorithm  is as follows.

\begin{algorithm}[H]
\caption{Gradient Descend Algorithm Based on Our New Models}
\begin{algorithmic}[H]
\Require
Measurement vectors: $a_i\in \R^n, i=1,\ldots,m $; Observations: $y \in \R^m$;  Parameter $\beta$; Step size $\mu$; Tolerance $\epsilon>0$  \\
\begin{enumerate}
\item[1:] Random initial guess $u_0\in \Rn$.
\item[2:] For $k=0,1,2,\ldots,$ if $\norm{\nabla f(u_{k})} \ge \epsilon $ do
\[
u_{k+1}=u_{k}-\mu \nabla f(u_{k})
\]
\item[3:] End for
\end{enumerate}

\Ensure
The vector $ u_T $.
\end{algorithmic}
\end{algorithm}

The performance of our  PAM1 and PAM2  algorithms are conducted via a series of numerical experiments in comparison against SAF, Trust Region \cite{sun2016complete},  WF \cite{WF}, TWF \cite{TWF} and TAF \cite{TAF}. Here, it is worth emphasizing that random initialization is used for SAF, Trust Region \cite{sun2016complete} and our PAM1, PAM2 algorithms while all other algorithms have adopted a spectral initialization.  Our theoretical results are for real Gaussian case, but the algorithms can be easily adapted to the complex Gaussian and CDP cases. All experiments are carried out on a MacBook Pro with a 2.3GHz Intel Core i5 Processor and 8 GB 2133 MHz LPDDR3 memory.

\subsection{Recovery of 1D Signals}
In our numerical experiments, the target vector $x\in \Rn$ is chosen  randomly from the standard Gaussian distribution and the measurement vectors $ a_i, \,i=1,\ldots,m$ are generated randomly from standard Gaussian distribution or CDP model.  For the real Gaussian case, the signal $ x \sim  \mathcal{N}(0,I_n)$ and measurement vectors $ a_i \sim  \mathcal{N}(0,I_n)$ for $i=1,\ldots,m$. For the complex Gaussian case, the signal $ x \sim  \mathcal{N}(0,I_n)+i  \mathcal{N}(0,I_n)$ and measurement vectors $ a_i  \sim \mathcal{N}(0,I_n/2)+i \mathcal{N}(0,I_n/2)$. For the CDP model, we use masks of octanary patterns as in \cite{WF}.  For simplicity, our parameters and step size are fixed for all experiments. Specifically, we adopt parameter $\beta=1/2$ and step size $\mu=1$ for SAF.
We choose the parameter $\beta=1$, step size $\mu=0.6$ and  $\mu=2.5$ for PAM1 and PAM2, respectively.
  For Trust Region, WF, TWF and TAF, we use the code provided in the original papers with suggested parameters.
 
 \begin{figure}[H]
\centering
    \subfigure[]{
     \includegraphics[width=0.4\textwidth]{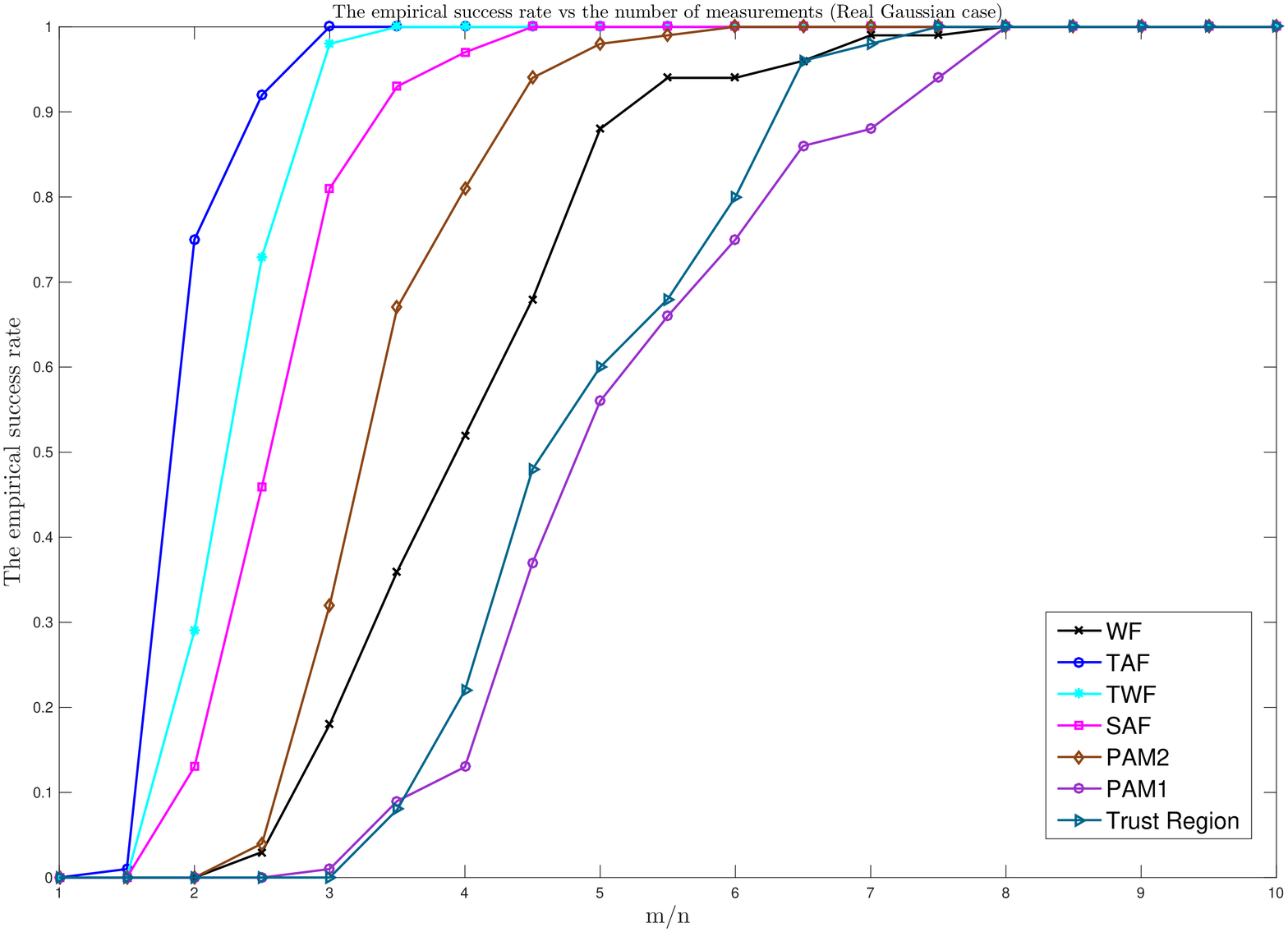}}
\subfigure[]{
     \includegraphics[width=0.4\textwidth]{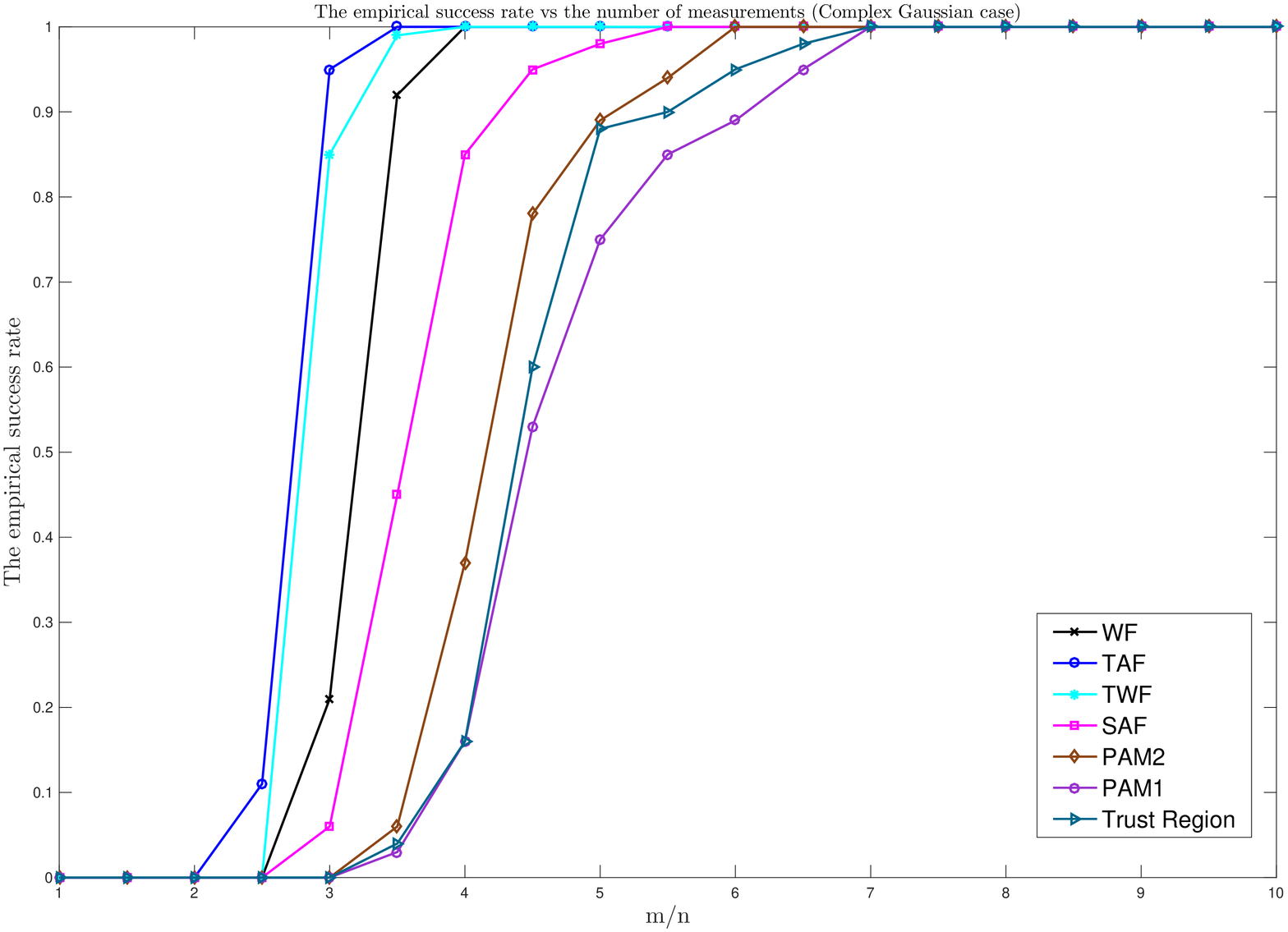}}
   \subfigure[]{
     \includegraphics[width=0.4\textwidth]{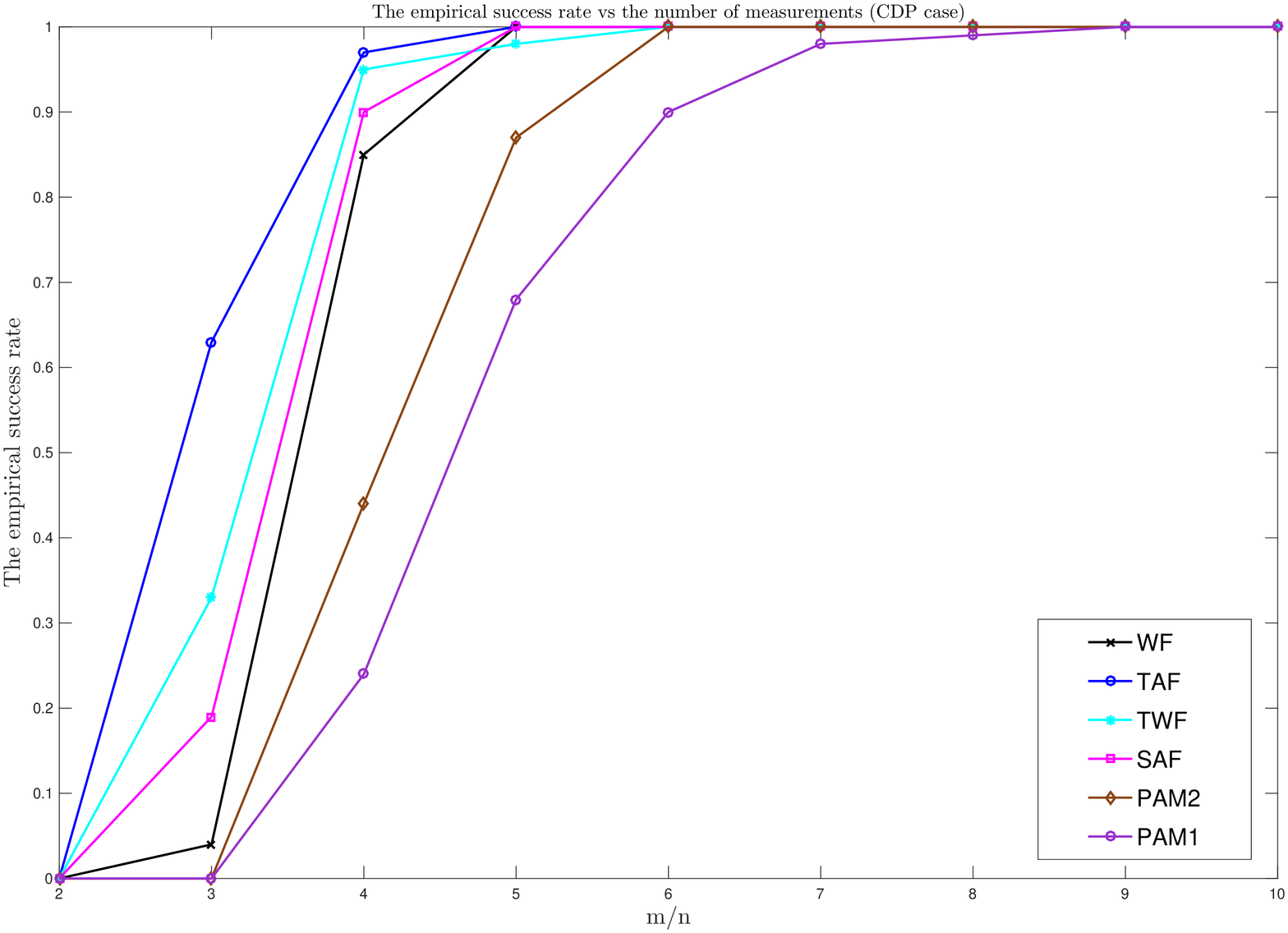}}  
\caption{ The empirical success rate for different $m/n$ based on $100$ random trails. (a) Success rate for real Gaussian case, (b) Success rate for complex Gaussian case, (c) Success rate for CDP case.}
\label{figure:succ}
\end{figure}

\begin{example}{\rm 
In this example, we test the empirical success rate of PAM1, PAM2 versus the number of measurements.  We conduct the experiments for the real Gaussian, complex Gaussian and CDP cases respectively.
We choose $n=128$ and the maximum number of iterations is $T=2500$.   For real and complex Gaussian cases, we vary $m$ within the range $[n,10n]$. For CDP case, we set the ratio $m/n=L$ from $2$ to $10$.
For each $m$, we run $100$ times trials to calculate the success rate. Here, we say a trial to have successfully reconstructed the target signal if the relative error satisfies $\mbox{dist}(u_{T}-x)/\norm{x} \le 10^{-5}$.
The results are plotted in Figure \ref{figure:succ}. 
It can be seen that
$6n$ Gaussian phaseless measurement  or $7$ octanary patterns are enough for exactly recovery for PAM2.

}\end{example}

\begin{example}
{\rm In this example, we compare the convergence rate of PAM1, PAM2 with those of SAF, WF, TWF, TAF for real Gaussian and complex Gaussian cases. We choose $n=128$ and $m=6n$. The results are presented in Figure \ref{figure:relative_error}. Since PAM1 as well as PAM2 algorithm chooses a random initial guess according to the standard Gaussian distribution instead of adopting a spectral initialization, it sometimes need to escape the saddle points with a small number of  iterations.  Due to its high efficiency to escape the saddle points, it still performs well comparing with state-of-the-art algorithms with spectral initialization. 
\begin{figure}[H]
\centering
\subfigure[]{
     \includegraphics[width=0.45\textwidth]{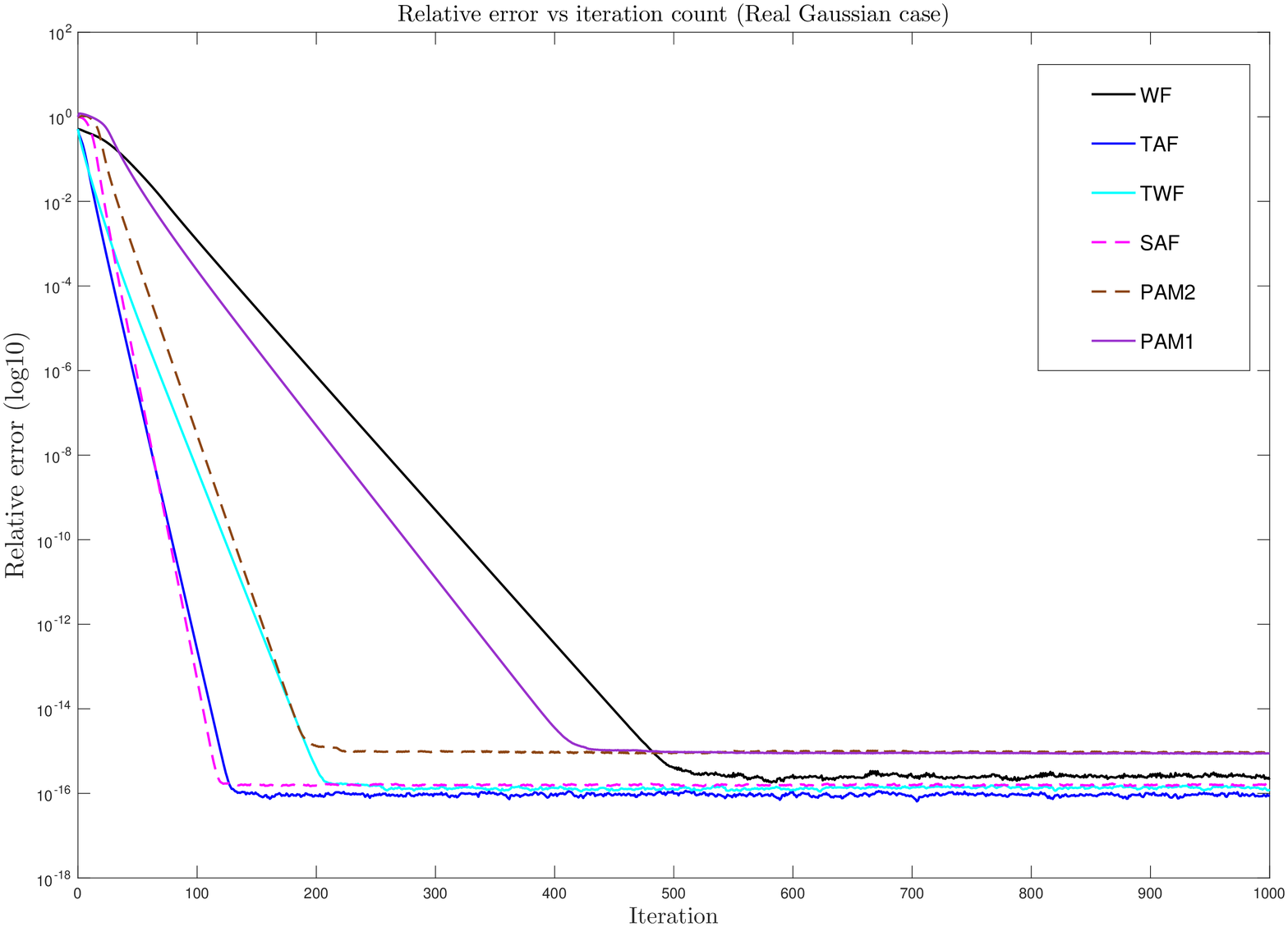}}
\subfigure[]{
     \includegraphics[width=0.45\textwidth]{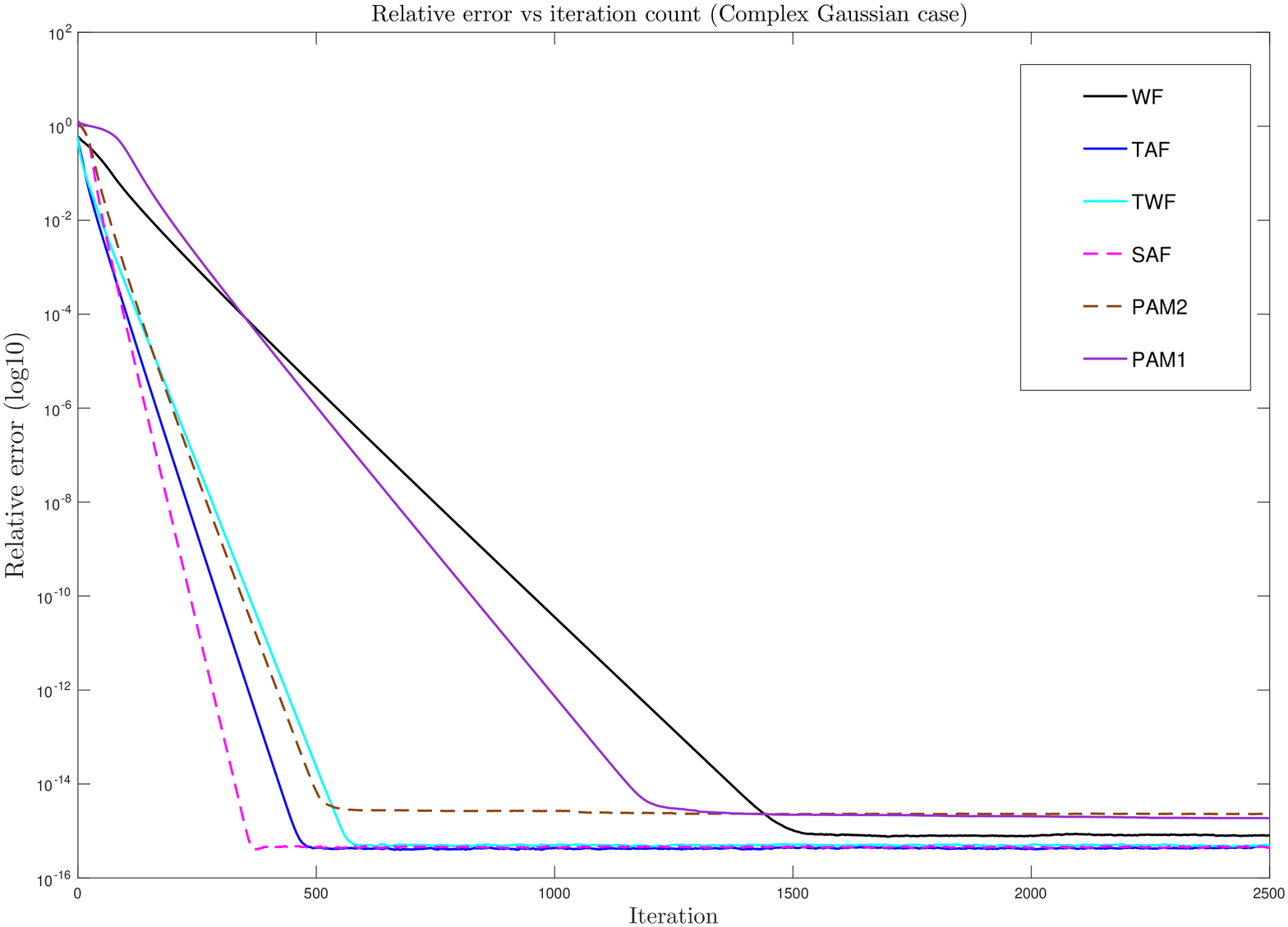}}
\caption{ Relative error versus number of iterations for PAF, SAF, WF, TWF, and TAF method: (a) Real Gaussian case; (b)  Complex Gaussian case.}
\label{figure:relative_error}
\end{figure}
}
\end{example}

\begin{example}{\rm
In this example, we compare the time elapsed and the iteration needed for WF, TWF, TAF, SAF and our PAM1, PAM2 to achieve the relative error  $10^{-5}$ and $10^{-10}$, respectively.  We choose $n=1000$ with $m=8n$. We adopt the same spectral initialization method for WF, TWF, TAF and the initial guess is obtained by power method with $50$ iterations.  We run $50$ times trials to calculate the average time elapsed and iteration number for those algorithms. The results are shown in Table \ref{tab:performance_comparison}. The numerical results show that PAM2 takes around $15$ and $42$ iterations to escape the saddle points for the real and complex Gaussian cases, respectively. 

}
\end{example}

\begin{table}[tp]
  \centering
  \fontsize{12}{16}\selectfont
  \caption{Time Elapsed and Iteration Number among Algorithms on Gaussian Signals with $n=1000$.}
  \label{tab:performance_comparison}
    \begin{tabular}{|c|c c|cc|cc|cc|}
    \hline
    \multirow{2}{*}{Algorithm}&
    \multicolumn{4}{c|}{Real Gaussian}&\multicolumn{4}{c|}{ Complex Gaussian }\cr\cline{2-9}
    &\multicolumn{2}{c|}{$10^{-5}$ }&\multicolumn{2}{c|}{$10^{-10}$ }& \multicolumn{2}{c|}{$10^{-5}$ } &\multicolumn{2}{c|}{$10^{-10}$ }\cr \hline
    & Iter & Time(s) & Iter & Time(s) & Iter & Time(s) & Iter & Time(s) \cr \hline
   SAF & 44&\bf{0.1556} &68 &\bf{0.2276} &113&\bf{1.3092} & 190 &\bf{2.3596} \cr\hline
    PAM1&108 &3.3445&204 &5.5768 &291&35.8624&591 & 75.3231\cr\hline
    PAM2&46&1.5816&84 &2.1980 &129&15.8295& 239&27.6362\cr\hline
    WF &125&4.4214& 229 &6.3176 &304&34.6266& 655&86.6993\cr \hline
    TAF &29&0.2744&60&0.3515 &100&1.7704& 211 &2.7852\cr \hline
    TWF&40&0.3181&87&0.4274&112&1.9808& 244&3.7432\cr \hline
    Trust Region &\bf{21}&2.9832&\bf{29}&4.4683&{\bf 33}&19.1252& \bf{42}&29.0338\cr \hline
    \end{tabular}
\end{table}


\subsection{Recovery of Natural Image}
We next compare the performance of the above algorithms on recovering a natural image from masked Fourier intensity  measurements. The image is the Milky Way Galaxy with resolution $1080 \times 1920$. The colored image has RGB channels. We use $L=20$ random octanary patterns to obtain the Fourier intensity measurements for each R/G/B channel as in \cite{WF}. Table \ref{tab:performance_comp_image} lists the averaged time elapsed and the iteration needed to achieve the relative error  $10^{-5}$ and $10^{-10}$  over the three RGB channels. We can see that our algorithms have good performance comparing with state-of-the-art algorithms with spectral initialization. 
Furthermore, our algorithms perform well even with $L=10$ under $300$ iterations, while WF fails.  Figure \ref{figure:galaxy} shows the image recovered by PAM2. 

\begin{figure}[H]
\centering
     \includegraphics[width=0.9\textwidth]{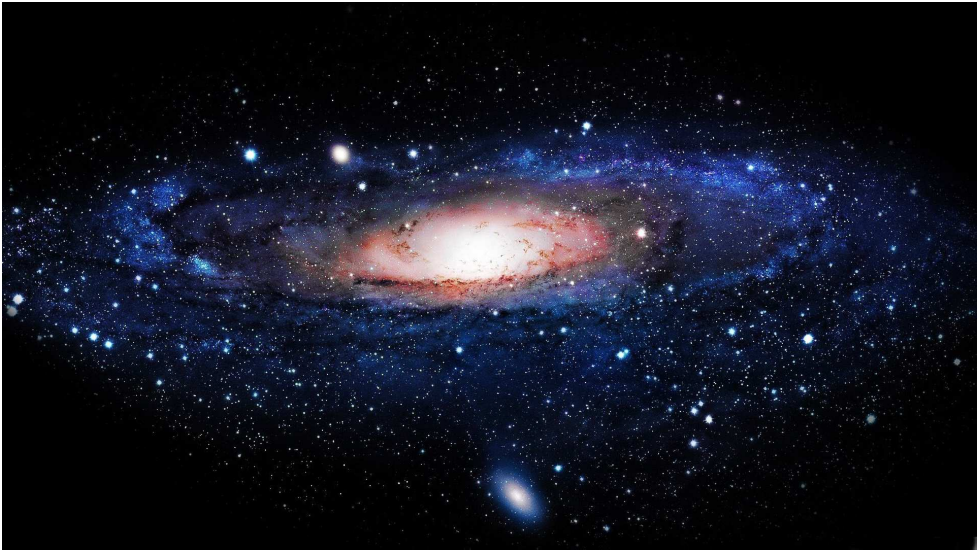}
\caption{ The Milky Way Galaxy image: PAM2 with $L=10$ takes $300$ iterations, computation time is $524.1$ s, relative error is  $7.26 \times 10^{-13}$.}
\label{figure:galaxy}
\end{figure}

\begin{table}[tp]
  \centering
  \fontsize{13}{16}\selectfont
  \caption{Time Elapsed and Iteration Number among Algorithms on Recovery of Galaxy Image.}
  \label{tab:performance_comp_image}
    \begin{tabular}{|c|c c|cc|}
    \hline
    \multirow{2}{*}{Algorithm}&
    \multicolumn{4}{c|}{The Milky Way Galaxy}\cr\cline{2-5}
    &\multicolumn{2}{c|}{$10^{-5}$ }&\multicolumn{2}{c|}{$10^{-10}$ }\cr \hline
    & Iter & Time(s) & Iter & Time(s) \cr \hline
   SAF & 92 &\bf{202.47} &148 &\bf{351.21} \cr\hline
    PAM1&198 &462.27 & 306 &710.27  \cr\hline
    PAM2&113 &260.48 &187 &441.55 \cr\hline
    WF &158 &381.7 & 277 &621.63 \cr \hline
    TAF &\bf{65} &223.89&\bf{122}&368.22 \cr \hline
    TWF&68 &315.14&145&566.84\cr \hline
    \end{tabular}
\end{table}

\subsection{ Recovery of signals with noise}
We now demonstrate the robustness of PAM1, PAM2 to noise and compare them with SAF,  WF, TWF, TAF. We  consider the noisy model $y_i=\abs{\nj{ a_i, x}}+\eta_i$ and add different level of Gaussian noises to explore the relationship between the signal-to-noise rate (SNR) of the measurements and the mean square error (MSE) of the recovered signal. Specifically, SNR and MSE are evaluated by
\[
\mbox{MSE}:= 10 \log_{10} \frac{\mbox{dist}^2(u,x)}{\norms{x}^2} \quad \mbox{and} \quad \mbox{SNR}=10 \log_{10} \frac{\sum_{i=1}^m \abs{a_i^\T x}^2}{\norms{\eta}^2},
\]
where $u $ is the output of the algorithms given above after $2500$ iterations. We choose $n=128$ and $m=8n$. The SNR varies from $20$db to $60$db. The result is shown in Figure \ref{figure:SNR}. We can see that our algorithms are stable for noisy phase retrieval.

\begin{figure}[H]
\centering
\subfigure[]{
     \includegraphics[width=0.45\textwidth]{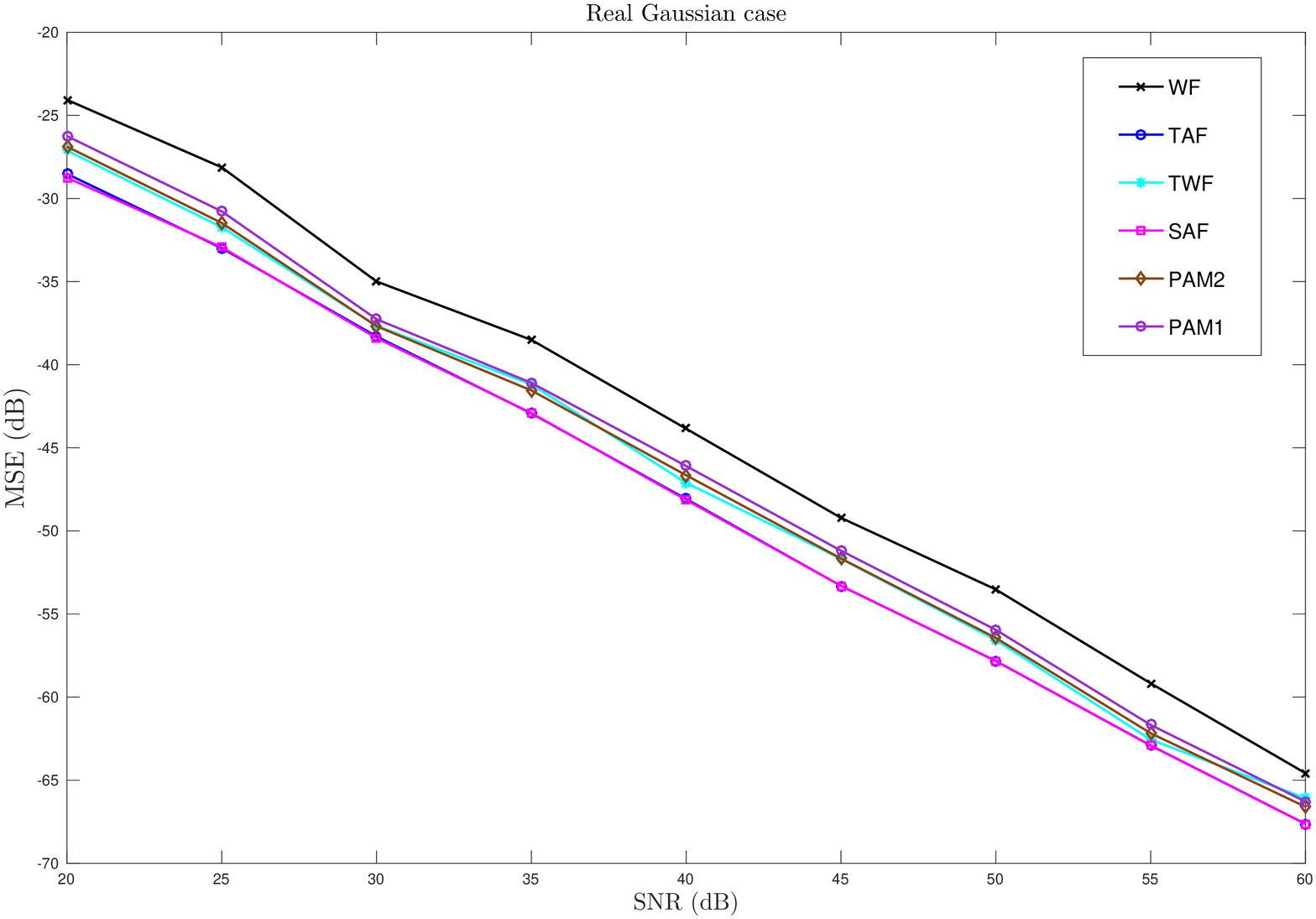}}
\subfigure[]{
     \includegraphics[width=0.45\textwidth]{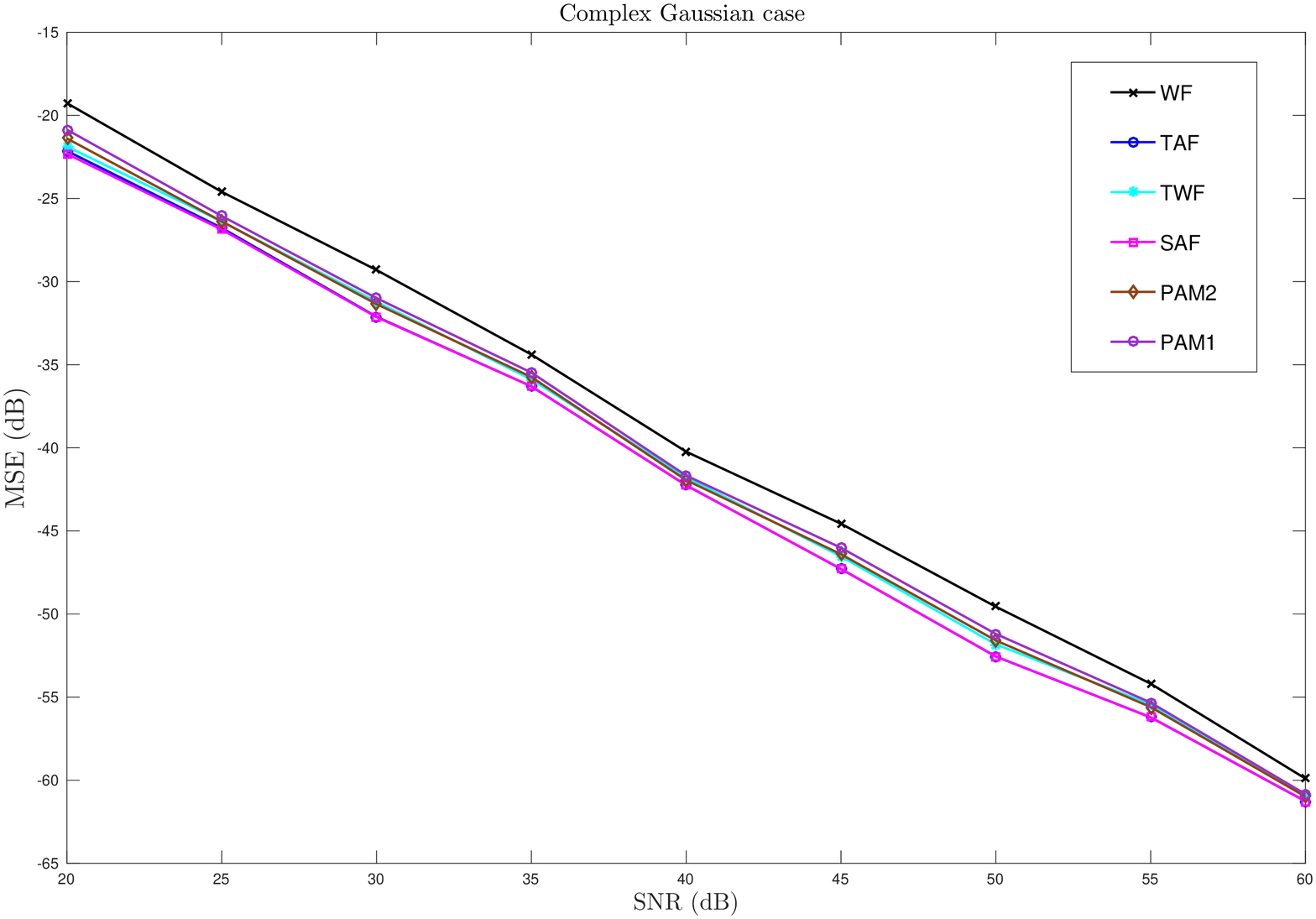}}
\caption{ SNR versus relative MSE on a dB-scale under the noisy Gaussian model: (a) Real Gaussian case; (b)  Complex Gaussian case.}
\label{figure:SNR}
\end{figure}

 \appendix
 \renewcommand{\appendixname}{Appendix~\Alph{section}}
 
\section{Auxiliary estimates for Section \ref{S:model2} } \label{Ap2mod}

\begin{proof}[Proof of Lemma \ref{lemAu29_2}]
Recall that 
\begin{align*}
\mathbb E \sqrt{\beta+X_t^2} \sqrt{\beta \rho^2+X_1^2} =:\, h_{\infty}(\rho, t). 
\end{align*}
Since $\rho \sim 1$ we shall slightly abuse notation and write $h_{\infty}(\rho,t)$ simply as $h(t)$
in this proof. 
Denote $g(x) =\sqrt{\beta+x^2}$. Clearly
\begin{align*}
g^{\prime}=\frac x {\sqrt{\beta+x^2}}, \quad 
g^{\prime\prime} = \beta (\beta+x^2)^{-\frac 32}.
\end{align*}
Since $X_t = t X_1+ \sqrt{1-t^2} Y_1$, where $X_1$ and $Y_1$ are independent
standard 1D Gaussian random variables,  we clearly have 
\begin{align*}
h(t) = \int g(tx+\sqrt{1-t^2} y)  k(x)  \rho_1(x, y) dx dy,
\end{align*}
where 
\begin{align*}
& k(x) =\sqrt{\beta \rho^2 +x^2}, \qquad \rho_1(x,y)=\frac 1{2\pi} e^{-\frac{x^2+y^2} 2}.
\end{align*}
Observe that
\begin{align*}
&\partial_x (  g(tx+\sqrt{1-t^2} y) ) = g^{\prime} \cdot t;\\
&\partial_y ( g (tx+\sqrt{1-t^2} y) ) = g^{\prime} \cdot \sqrt{1-t^2};\\
& \partial_t (  g(tx+\sqrt{1-t^2} y) ) =g^{\prime}\cdot( x -\frac t{\sqrt{1-t^2} } y)
=\frac 1 {\sqrt{1-t^2}} (  x \partial_y g - y \partial_x g).
\end{align*}
The third identity is the key to obtaining cancellation when calculating $h^{\prime}(t)$ and
$h^{\prime\prime}(t)$. 

Observe that 
\begin{align*}
(x \partial_y - y \partial_x ) \rho_1\equiv 0, \qquad\forall\, x, y \in \mathbb R. 
\end{align*}
Then clearly
\begin{align*}
h^{\prime}(t) &= \frac 1{\sqrt{1-t^2} } \int \Bigl( 
(x \partial_y  -y \partial_x ) (g(tx +\sqrt{1-t^2} y) ) 
\Bigr) k(x) \rho_1(x,y) dx dy \notag \\
&=\frac 1 {\sqrt{1-t^2}} \int g(tx+\sqrt{1-t^2}y) \cdot y k^{\prime}(x) \rho_1 dx dy \\
&=\frac 1 {\sqrt{1-t^2}}
2\int_{x>0,y>0} (g(tx+\sqrt{1-t^2} y)
-g(tx-\sqrt{1-t^2} y) ) \cdot y k^{\prime}(x) \rho_1 dx dy.
\end{align*}
Clearly then 
\begin{align*}
&h^{\prime}(t) >0, \quad 0<t<1; \\
&h^{\prime}(t) <0, \quad -1<t<0.
\end{align*} 
 Moreover,
\begin{align*}
|h^{\prime}(t)| \le 4 \int \|g^{\prime}\|_{\infty} y^2 \cdot \rho_1 dx dy \lesssim 1,
\quad\forall\, |t| < 1.
\end{align*}
Note that we can actually obtain $|h^{\prime}(t) | \lesssim 1$ for all $|t| \le 1$. 
On the other hand, for $0<t<1$,
\begin{align*}
h^{\prime}(t)
&\gtrsim \int_{x>0,y>0}
\frac{t y x}{\sqrt{\beta+(tx+\sqrt{1-t^2} y)^2}
+ \sqrt{\beta+(tx-\sqrt{1-t^2} y)^2} }
\cdot y \cdot \frac x {\sqrt{\beta \rho^2 +x}}
\cdot e^{-\frac{x^2+y^2} 2} dx dy \notag \\
&\gtrsim \int_{x\sim 10,y\sim 10}
\frac{t y x}{\sqrt{\beta+(tx+\sqrt{1-t^2} y)^2}
+ \sqrt{\beta+(tx-\sqrt{1-t^2} y)^2} }
\cdot y \cdot \frac x {\sqrt{\beta \rho^2 +x}}
\cdot e^{-\frac{x^2+y^2} 2} dx dy \notag \\
&\gtrsim t. 
\end{align*}
Note that the implied constants here are allowed to depend on $\beta$.
Similarly one can show $-h^{\prime}(t) \gtrsim |t|$ for $-1<t<0$. 
Next we treat $h^{\prime\prime}(t)$ in  the regime $|t| \ll 1$. Observe that
\begin{align*}
\frac d {dt} \Bigl( \frac 1 {\sqrt{1-t^2} } \Bigr) = O(t),\quad
\frac 1 {\sqrt{1-t^2}} =1+O(t^2).
\end{align*}
Thus
\begin{align*}
h^{\prime\prime}(t)
&=O(t) + \frac 1{\sqrt{1-t^2} }\int g^{\prime}(t x + \sqrt{1-t^2} y)
(x-\frac t {\sqrt{1-t^2}} y) y k^{\prime}(x) \rho_1 dx dy \notag \\
&=O(t) + \int g^{\prime}(t x + \sqrt{1-t^2} y)
(x-\frac t {\sqrt{1-t^2}} y) y k^{\prime}(x) \rho_1 dx dy.
\end{align*}
Observe that the contribution of $\frac t {\sqrt{1-t^2}} y$ is bounded by $O(t)$.
Then
\begin{align*}
h^{\prime\prime}(t)=O(t)+
\int g^{\prime}(tx+\sqrt{1-t^2} y) y  \frac {x^2} {\sqrt{\beta \rho^2+x^2}}  \rho_1 dx dy.
\end{align*}
Denote $x_t= t x +\sqrt{1-t^2} y$. Then $y= \frac {x_t -t x}{\sqrt{1-t^2}}$.  The
contribution due to $\frac{ tx}{\sqrt{1-t^2}}$ is also $O(t)$.  Thus
\begin{align*}
h^{\prime\prime}(t) = O(t)
+ \int g^{\prime}(x_t) \frac {x_t}{\sqrt{1-t^2}} \frac {x^2} {\sqrt{\beta \rho^2+x^2}}  \rho_1 dx dy.
\end{align*}
Note that  $g^{\prime}(z) z = \frac {z^2}{\sqrt{\beta+z^2}}$.  Thus for
$|t|\ll 1$,  if $9\le y\le 11$, $\frac 14 \le x\le \frac 12$,
then $x_t \sim 1$, and the main term is $O(1)$. Thus
\begin{align*}
h^{\prime\prime}(t) \gtrsim 1
\end{align*}
for all $|t|\ll 1$.
\end{proof}

If we take the limit $\beta \to 0$ in $h(\rho, t)= \mathbb E
\sqrt{\beta +X_t^2} \sqrt{\beta \rho^2 +X_1^2}$. Then we obtain the expression
\begin{align*}
\mathbb E |X_t| |X_1|.
\end{align*}
Understanding this limiting case is of some importance for the case $\beta>0$. The following
proposition gives a very explicit characterization. 
\begin{prop}
We have for $|t|<1$:
\begin{align*}
&\frac {\pi} 2\mathbb E |X_t| |X_1| = (\frac{\pi}2 - \arccos t) t + \sqrt{1-t^2};  \\
&(\frac {\pi} 2\mathbb E |X_t| |X_1|)^{\prime}= \frac{\pi}2- \arccos t; \\
&(\frac {\pi} 2\mathbb E |X_t| |X_1|)^{\prime\prime}= \frac 1 {\sqrt{1-t^2}}. 
\end{align*}
\end{prop}
\begin{proof}
We recall $X_t = t X+\sqrt{1-t^2}Y$ where $X=X_1\sim \mathcal N(0,1)$, $Y\sim
\mathcal N(0,1)$ are independent. Without loss of generality we can assume
$0\le t<1$.  Denote $t= \sin \theta_0$ where $0\le \theta_0< \frac {\pi} 2$.
Then by using polar coordinates, we have
\begin{align*}
\frac {\pi}2
\mathbb E|X_t X_1|
&=\frac 14\int_{\mathbb R^2}
|tx +\sqrt{1-t^2} y| \cdot |x| \cdot e^{-\frac {x^2+y^2} 2} dx dy \notag \\
&=\frac 14 \int_0^{2\pi} |\sin(\theta+\theta_0)| \cdot |\cos\theta | d\theta
\cdot \int_0^{\infty} r^3 e^{-\frac {r^2} 2} dr \notag \\
&= \frac 12 \int_0^{2\pi} |\sin(\theta+\theta_0)| \cdot |\cos\theta | d\theta \notag \\
& =  \int_0^{\pi} |\sin(\theta+\theta_0)| \cdot |\cos\theta | d\theta.
\end{align*}
Now observe that 
\begin{align*}
& \int_0^{\pi}| \sin(\theta+\theta_0)| \cdot |\cos\theta|  d\theta \notag \\
=& \int_0^{\frac {\pi}2 } \sin(\theta+\theta_0) \cdot \cos\theta d\theta-
\int_{\frac {\pi}2}^{\pi-\theta_0} \sin(\theta+\theta_0) \cdot \cos\theta  d\theta \notag \\
& \qquad + \int_{\pi-\theta_0}^{\pi} \sin(\theta+\theta_0) \cdot \cos\theta  d\theta.
\end{align*}
The desired result then easily follows by an explicit computation.
\end{proof}


\begin{rem*}
One may wonder why at $t=1$, the derivative is formally given by $\frac {\pi}2$ instead
of being zero since $u=x$ should be a critical point. The reason is due to the artificial
singularity introduced by our representation. To see this, one can consider
the regular variable $t=\cos \theta$ with $\theta \in [0,\pi]$, then
\begin{align*}
f(\theta) = (\frac {\pi} 2 - \theta) \cos \theta +\sin \theta;
\end{align*}
Then clearly $f^{\prime}(0)=0$ and $f^{\prime\prime}(0)=-\frac{\pi}2<0$. On the other hand,
\begin{align*}
f(\theta) = \tilde f( t) =\tilde f( \cos \theta).
\end{align*}
Then
\begin{align*}
f^{\prime}(\theta) = \tilde f^{\prime}(\cos \theta) (-\sin \theta).
\end{align*}
Thus
\begin{align*}
\tilde f^{\prime}(1) = \lim_{\theta \to 0}
\frac {f^{\prime}{(\theta})} { -\sin \theta} =  -f^{\prime\prime}(0) >0.
\end{align*}

\end{rem*}

\begin{lem} \label{bp7}
Let $X_i$: $1\le i\le m$ be independent random variables with 
\begin{align*}
\max_{1\le i \le m} \mathbb E |X_i|^4 \lesssim 1.
\end{align*}
Then for any $t>0$, 
\begin{align*}
\mathbb P ( \Bigl| \frac 1m \sum_{j=1}^m X_i  - 
\operatorname{mean} \Bigr| > t)  \lesssim \frac 1{m^2 t^4}.
\end{align*}
\end{lem}
\begin{proof}
Without loss of generality we can assume $X_i$ has zero mean. The result then follows from
the observation that
\begin{align*}
\mathbb E (\sum_{j=1}^m X_j )^4 \lesssim \sum_{i<j} \mathbb E X_i^2 X_j^2 +
\sum_{i} \mathbb E X_i^4 \lesssim m^2.
\end{align*}
\end{proof}

\begin{lem} \label{A30_1}
Let $\psi: \mathbb R\to \mathbb R$ be continuously differentiable such that
\begin{align*}
\sup_{z \in \mathbb R} |\psi(z)| + \sup_{z\in \mathbb R}
\sqrt{1+z^2} |\psi^{\prime}(z)| \lesssim 1.
\end{align*}
Let 
\begin{align*}
\psi_0(R, z) = z \sqrt{R+z^2}, \quad z \in \mathbb R,  \, c_1\le R\le c_2,
\end{align*}
where $0<c_1<c_2<\infty$ are two fixed constants. For any
$0<\epsilon \le 1$, if $m\gtrsim n$, then the following hold
with high probability:

\begin{align*}
\Bigl| \frac 1 m \sum_{j=1}^m \psi( a_j\cdot u)
\psi_0 (R, a_j\cdot e_1) - \operatorname{mean} \Bigr|\le \epsilon,
\qquad \forall\, u \in \mathbb S^{n-1}, \quad\forall\, c_1\le R\le c_2.
\end{align*}

\end{lem}

\begin{proof}
Step 1. Let $0<\eta<\frac 12$ be a constant whose value will be chosen sufficiently small. Let
$\phi \in C_c^{\infty}(\mathbb R)$ be such that $0\le \phi(x)\le 1$ for all $x$,
$\phi(x)=1$ for $|x| \le 1$, and $\phi(x)=0$ for $|x| \ge 2$. Denote
\begin{align*}
\langle x \rangle = \sqrt{1+x^2}, \qquad x \in \mathbb R.
\end{align*}
Consider first the piece
\begin{align*}
I_1 &= \frac 1m \sum_{j=1}^m \psi(a_j\cdot u)
\phi( \frac {a_j \cdot u}  {\eta \langle a_j\cdot e_1 \rangle } )
\psi_0(R, a_j \cdot e_1) \notag \\
& = \frac 1m \sum_{j=1}^m \psi(a_j\cdot u)
\phi( \frac {a_j \cdot u}  {\eta \langle a_j\cdot e_1 \rangle } )
\psi_0(R, a_j \cdot e_1)  \phi(\frac {a_j \cdot e_1} K) \notag \\
& \quad + 
\frac 1m \sum_{j=1}^m \psi(a_j\cdot u)
\phi( \frac {a_j \cdot u}  {\eta \langle a_j\cdot e_1 \rangle } )
\psi_0(R, a_j \cdot e_1) (1- \phi(\frac {a_j\cdot e_1} K) ) \notag \\
&=:\; I_{1,a}+I_{1,b}.
\end{align*}
where $K=\eta^{-\frac 16}$. 
Thanks to the cut-off $\phi(\frac {a_j\cdot e_1} {K})$, we have
$|a_j\cdot e_1| \le 2K$ on its support. Thus 
\begin{align*}
|\psi_0(R,a_j\cdot e_1)| \phi(\frac {a_j\cdot e_1} K) \lesssim K^2,
\end{align*}
and 
\begin{align*}
|I_{1,a}| \le \alpha_0  K^2 \frac 1m \sum_{j=1}^m
\phi( \frac {a_j\cdot u} { \eta \langle 2K \rangle}),
\end{align*}
where $\alpha_0>0$ is an absolute constant. 
Clearly 
\begin{align*} 
K^2\mathbb E \phi( \frac{a_j\cdot u} {\eta \langle 2K \rangle})
&\le K^2 \frac 1 {\sqrt{2\pi}}  
\int e^{-\frac {x^2} { 2} }
\phi( \frac x {\eta \langle 2K \rangle } ) dx \notag \\
& \lesssim \eta^{-\frac 13} \cdot \eta^{\frac 56}=\eta^{\frac 12}. 
\end{align*}
By Bernstein's inequality, we have with high probability, 
\begin{align*}
\Bigl|  \frac 1m \sum_{j=1}^m
\phi( \frac {a_j\cdot u} { \eta \langle 2K \rangle})
- \operatorname{mean} \Bigr| \le \eta.
\end{align*}
Thus for $\eta>0$ sufficiently small, 
\begin{align*}
|I_{1,a}| \lesssim \eta^{\frac 12} \le \frac {\epsilon}{10}. 
\end{align*}
For $I_{1,b}$, we have
\begin{align*}
|I_{1,b}|
\le \alpha_1 \cdot \frac 1m \sum_{j=1}^m (a_j\cdot e_1)^2 (
1-\phi(\frac {a_j \cdot e_1}  K) ),
\end{align*}
where $\alpha_1>0$ is a constant.  Similar to the estimate in 
$I_{1,a}$, we have with high probability,
\begin{align*}
\alpha_1 \cdot \frac 1m \sum_{j=1}^m (a_j\cdot e_1)^2 (
1-\phi(\frac {a_j \cdot e_1}  K) ) \lesssim  \eta \le \frac {\epsilon}{10}.
\end{align*}
Thus with high probability, it holds that for  sufficiently small $\eta$, 
\begin{align*}
|I_1| \le \frac {\epsilon} 5,
\end{align*}
By a simple estimate we have $|\mathbb E I_1 | \le \frac {\epsilon} 5$ for sufficiently
small $\eta$.  Thus 
\begin{align*}
|I_1 -\mathbb E I_1| \le \frac {2\epsilon} 5.
\end{align*}

Step 2. We now consider the main piece
\begin{align*}
I_2(u,R) = \frac 1 m \sum_{j=1}^m \psi( a_j\cdot u)
\cdot \Bigl( 1- \phi( \frac {a_j\cdot u} { \eta \langle a_j\cdot e_1 \rangle} ) \Bigr)
\cdot \psi_0(R, a_j\cdot e_1).
\end{align*}
Note that $\eta>0$ is fixed in step 1. For simplicity we denote
\begin{align*}
X_j= a_j\cdot e_1,  \quad h_1(z) = \psi(z) \cdot \Bigl(1- \phi(\frac z {\eta \langle X_j \rangle} ) 
\Bigr).
\end{align*}
Thanks to the cut-off $1-\phi(\frac z {\eta \langle X_j \rangle} )$ and the 
fact that $|\psi^{\prime}(z) | \lesssim \langle z \rangle^{-1}$, we have
\begin{align*}
|h_1^{\prime}(z) | \lesssim \eta^{-1}  \langle X_j \rangle^{-1} \lesssim \langle X_j
\rangle^{-1},
\end{align*}
where in the last inequality we have included $\eta^{-1}$ into the implied constant.
Since in this step $\eta>0$ is a fixed constant this will not cause any problem.
Clearly then 
\begin{align} \label{A30e30}
|h_1(z) -h_1(\tilde z)| \lesssim \langle X_j \rangle^{-1} |z-\tilde z|, 
\qquad \forall\, z, \tilde z \in \mathbb R.
\end{align}
Also 
\begin{align}\label{A30e30b}
|\psi_0(R, X_j) -\psi_0(\tilde R, X_j) | \le |X_j| \cdot |R-\tilde R|^{\frac 12},
\qquad \forall\, c_1\le R,\, \tilde R \le c_2.
\end{align}
We shall need these important estimates below.

Let $\delta>0$ be a small constant whose smallness will be specified later. We choose 
a $\delta$-net $F_{\delta}$ 
covering the set $\mathbb S^{n-1} \times \{ R:\, c_1 \le R \le c_2 \}$. We endow
the set $S^{n-1} \times \{R: \, c_1\le R\le c_2\}$ with the simple metric:
\begin{align*}
d( (u,R), (\tilde u, \tilde R) ) = \| u-\tilde u\|_2 + |R-\tilde R|.
\end{align*}
Note that 
$$\operatorname{Card}(F_{\delta}) \le \exp (C_{\delta} n),$$
where $C_{\delta}>0$ depends only on $\delta$. 
By Bernstein's inequality, we have for any $0<\eta_1\le \frac 12$, 
\begin{align*}
\mathbb P 
\Bigl( \sup_{(u,R)\in F_{\delta} } \Bigl| I_2(u,R) - \operatorname{mean} \Bigr| \ge \eta_1\Bigr)
\le 2 e^{C_{\delta} n} \cdot  e^{- c \eta_1^2 m}.
\end{align*}
Thus with high probability and taking $\eta_1= \frac {\epsilon}{10}$, we have
\begin{align*}
|I_2(u, R) - \mathbb E I_2(u,R) | \le \frac{\epsilon}{10}, \qquad\forall\, (u,R) \in F_{\delta}.
\end{align*}

Now let $(u, R) \in F_{\delta}$, and consider any $(\tilde u, \tilde R)$ such that
\begin{align*}
\| u-\tilde u\|_2 +|R-\tilde R| \le \delta.
\end{align*}
By using the estimates \eqref{A30e30}, \eqref{A30e30b}, we have
\begin{align*}
  & | I_2(u, R) -I_2(\tilde u, \tilde R) | \notag \\
  = & \;\Bigl| \frac 1m \sum_{j=1}^m ( h_1(a_j\cdot u) \psi_0(R,X_j)
  -h_1(a_j\cdot \tilde u) \psi_0(\tilde R, X_j)) \Bigr| \notag \\
  \le & \; \frac 1m \sum_{j=1}^m | h_1(a_j\cdot u)
  -h_1(a_j\cdot \tilde u) | |\psi_0(R,X_j) | 
  +\frac 1m \sum_{j=1}^m |h_1(a_j\cdot \tilde u) | |\psi_0(R,X_j)
  -\psi_0(\tilde R,X_j) | \notag \\
\le & \; \alpha_2 \frac 1m \sum_{j=1}^m  \langle X_j \rangle^{-1} | a_j\cdot (u-\tilde u)|
|X_j|\cdot \langle X_j \rangle 
+\alpha_2 \frac 1m \sum_{j=1}^m |R-\tilde R|^{\frac 12} |X_j| \langle X_j \rangle \notag \\
\le & \; \alpha_2 \frac 1m \sum_{j=1}^m  \Bigl(\frac 12| a_j\cdot (u-\tilde u)|^2 \cdot \delta^{-1}
+ \frac 12\delta |X_j|^2 \Bigr) 
+\alpha_2 |R-\tilde R|^{\frac 12} \frac 1m \sum_{j=1}^m (|X_j|^2 +1), \notag
\end{align*}  
where $\alpha_2>0$ is an absolute constant. 
By Bernstein's inequality, it holds with high probability that 
\begin{align*}
\frac 1m \sum_{j=1}^m |a_j\cdot v|^2 \le 2, \qquad \forall\, v\in \mathbb S^{n-1}.
\end{align*}
Thus
\begin{align*}
|I_2(u,R)-I_2(\tilde u,\tilde R)|
\le 10 \alpha_2 (\delta + \delta^{\frac 12}).
\end{align*}
Also it is easy to check that
\begin{align*}
|\mathbb E ( I_2(u,R)-I_2(\tilde u,\tilde R))|
\le 10 \alpha_2 (\delta + \delta^{\frac 12}).
\end{align*}
Therefore
\begin{align*}
| I_2(u,R) -\mathbb E I_2(u,R) - ( I_2(\tilde u, \tilde R), -\mathbb E
I_2(\tilde u, \tilde R) | \le 20 \alpha_2 (\delta+\delta^{\frac 12}).
\end{align*}
Now take $\delta$ such that 
\begin{align*}
20 \alpha_2 (\delta+\delta^{\frac 12}) \le \frac {\epsilon}{10}. 
\end{align*}
We then obtain (with high probability) 
\begin{align*}
|I_2(u,R)- \mathbb E I_2(u,R) | \le \frac {\epsilon}5, \qquad \forall\, u\in S^{n-1},
\, \forall\, c_1\le R\le c_2.
\end{align*}
Together with the estimate of $I_1$ in step 1, we obtain the desired conclusion. 
\end{proof}

\begin{lem} \label{A30_1a}
Let $\psi: \mathbb R\to \mathbb R$ be Lipschitz continuous  such that
\begin{align*}
\sup_{z \in \mathbb R} \frac {|\psi(z)|}{1+|z|}  + \sup_{z \ne \tilde z\in \mathbb R}
\frac {|\psi(z)-\psi(\tilde z)|}{|z-\tilde z|}  \lesssim 1.
\end{align*}
Let  $0<c_1<c_2<\infty$ be two fixed constants. For any
$0<\epsilon \le 1$, if $m\gtrsim n$, then the following hold
with high probability:

\begin{align*}
\Bigl| \frac 1 m \sum_{j=1}^m \psi( a_j\cdot u)
\sqrt{R+(a_j\cdot e_1)^2}  - \operatorname{mean} \Bigr|\le \epsilon,
\qquad \forall\, u \in \mathbb S^{n-1}, \quad\forall\, c_1\le R\le c_2.
\end{align*}

\end{lem}
\begin{proof}
The main point is use a $\delta$-covering of the set $\mathbb S^{n-1}\times
\{R:\, c_1\le R\le c_2\}$. Note that
\begin{align*}
& \Bigl| \psi(a_j\cdot u) \sqrt{R+(a_j\cdot e_1)^2} 
- \psi(a_j\cdot \tilde u) \sqrt{\tilde R+(a_j\cdot e_1)^2} \Bigr| \notag \\
\le&\; |\psi(a_j\cdot u) -\psi(a_j\cdot \tilde u)|
\sqrt{R+(a_j\cdot e_1)^2}
+ |\psi(a_j\cdot \tilde u)| \cdot |\sqrt{R+(a_j\cdot e_1)^2}
-\sqrt{\tilde R+(a_j\cdot e_1)^2} | \notag \\
\lesssim & \;  |a_j\cdot (u-\tilde u)| (1+|a_j\cdot e_1|)
+ (1+|a_j\cdot \tilde u|) \cdot |R-\tilde R|^{\frac 12}.
\end{align*}
The argument is then similar to that in Lemma \ref{A30_1}. We omit details.
\end{proof}

Consider
\begin{align*}
h = \frac 1m \sum_{j=1}^m \sqrt{\beta + (a_j\cdot \hat u)^2}
\cdot \sqrt{R+X_j^2},
\end{align*}
where
\begin{align*}
& \hat u = te_1 + \sqrt{1-t^2} e^{\perp}, \quad |t| <1, \, e^{\perp}\cdot e_1=0,
e^{\perp} \in \mathbb S^{n-1} ;\\
& X_j= a_j\cdot e_1, \quad 0<c_1\le R\le c_2 <\infty.
\end{align*}
In the above we take $c_1>0$, $c_2>0$ as two fixed constants. In our original model,
$R=\beta \rho^2$ and $\rho \sim 1$,  and therefore this assumption is quite natural.
In the lemma below we shall study $h$ in the regime 
\begin{align*}
|t| \le 1- \epsilon_0,
\end{align*}
where $0<\epsilon_0\ll 1$.  The smallness of $\epsilon_0$ will be needed later when we
study the regime $||\hat u\cdot e_1|-1| \ll 1$.  Here we shall show that away
from $|t|=1$ we have good control of $h$.

\begin{lem} \label{A30_2}
Let $0<\epsilon_0\ll 1$ be fixed. 
For any $0<\epsilon\le 1$, if $m\gtrsim n$, then with high probability it holds that
\begin{align*}
&| \partial_t h -\mathbb E \partial_t h | \le \epsilon, \qquad\forall\, |t| \le 1-\epsilon_0, \,
 e^{\perp}\cdot e_1=0, e^{\perp} \in \mathbb S^{n-1}, c_1\le R\le c_2.
\end{align*}
\end{lem}
\begin{proof}
Denote $g(x)=\sqrt{\beta+x^2}$ and
\begin{align*}
Z_j =a_j\cdot \hat u = t X_j +\sqrt{1-t^2} Y_j, \qquad Y_j=a_j\cdot e^{\perp}.
\end{align*}
Clearly
\begin{align*}
\frac d{dt} Z_j&= X_j - \frac {t} {\sqrt{1-t^2} } Y_j  \notag \\
& = \frac 1 {1-t^2} X_j - \frac t {1-t^2} Z_j.
\end{align*}
Therefore
\begin{align*}
\partial_t h &= \frac 1 {1-t^2}
\cdot \frac 1 m\sum_{j=1}^m g^{\prime}(Z_j)  X_j \sqrt{R+X_j^2} 
-\frac t {1-t^2} \cdot \frac 1m
\sum_{j=1}^m g^{\prime}(Z_j) Z_j \sqrt{R+X_j^2} \notag \\
&=:\frac 1{1-t^2} H_1  -\frac t{1-t^2}  H_2. 
\end{align*}
By Lemma \ref{A30_1}, it holds with high probability that
\begin{align*}
|H_1 -\mathbb E H_1 | \le (1-\epsilon_0^2) \cdot \frac {\epsilon}3,
\qquad\forall\,  \hat u\in \mathbb S^{n-1}, c_1\le R\le c_2.
\end{align*}
For $H_2$, we observe that
\begin{align*}
g^{\prime}(x)x = \frac {x^2}{\sqrt{\beta +x^2}}.
\end{align*}
By Lemma \ref{A30_1a}, it then holds with high probability that
\begin{align*}
|H_2 -\mathbb E H_2 | \le (1-\epsilon_0^2) \cdot \frac {\epsilon}3,
\qquad\forall\,  \hat u\in \mathbb S^{n-1}, c_1\le R\le c_2.
\end{align*}
The desired result then easily follows.
\end{proof}
\begin{lem} \label{A30_3}
Let $0<\epsilon_0\ll 1$ be fixed. 
For any $0<\epsilon\le \frac 12$, if $m\gtrsim n$, then with high probability it holds that
\begin{align*}
&\partial_{tt} h \ge \mathbb E \partial_{tt} h - \epsilon, 
\qquad\forall\, |t| \le 1-\epsilon_0, \, e^{\perp}\cdot e_1=0, 
 e^{\perp} \in \mathbb S^{n-1},
c_1\le R\le c_2.
\end{align*}
Furthermore, it holds with probability at least $1-O(m^{-2})$ that
\begin{align*}
&|\partial_{tt} h - \mathbb E \partial_{tt} h | \le  \epsilon, 
\qquad\forall\, |t| \le 1-\epsilon_0, \, e^{\perp}\cdot e_1=0, 
 e^{\perp} \in \mathbb S^{n-1},
c_1\le R\le c_2.
\end{align*}
\end{lem}
\begin{proof}
We adopt the same notation as in Lemma \ref{A30_2}. 
Observe that
\begin{align*}
\partial_{tt}h &=\frac 1m \sum_{j=1}^m g^{\prime}(Z_j)
\frac {d^2}{dt^2} Z_j \sqrt{R+X_j^2}
+
\frac 1m \sum_{j=1}^m g^{\prime\prime}(Z_j)
\Bigl(\frac d{dt} Z_j \Bigr)^2 \sqrt{R+X_j^2} \notag \\
&=:\, H_1+ H_2.
\end{align*}
We first deal with $H_1$. Note that
\begin{align*}
\frac {d^2}{dt^2} Z_j= - (1-t^2)^{-\frac 32} Y_j.
\end{align*}
Since 
\begin{align*}
Y_j = \frac 1 {\sqrt{1-t^2}} ( Z_j - tX_j),
\end{align*}
we obtain
\begin{align*}
H_1=  -(1-t^2)^{-2} \frac 1m \sum_{j=1}^m g^{\prime}(Z_j)
Z_j \sqrt{R+X_j^2}
+\frac t {(1-t^2)^2} \frac 1m \sum_{j=1}^m
g^{\prime}(Z_j) X_j \sqrt{R+X_j^2}.
\end{align*}
By similar estimates as in Lemma \ref{A30_2}, we have with high probability,
\begin{align*}
&\Bigl |\frac 1m \sum_{j=1}^m g^{\prime}(Z_j) Z_j
\sqrt{R+X_j^2} 
-\operatorname{mean}
\Bigr| \le \frac {\epsilon}{20}\cdot (1-\epsilon_0^2)^2, \qquad
\forall\, u \in \mathbb S^{n-1}, \, c_1\le R\le c_2; \\
&\Bigl |\frac 1m \sum_{j=1}^m g^{\prime}(Z_j) X_j
\sqrt{R+X_j^2} 
-\operatorname{mean}
\Bigr| \le \frac {\epsilon}{20}\cdot (1-\epsilon_0^2)^2, \qquad
\forall\, u \in \mathbb S^{n-1}, \, c_1\le R\le c_2.
\end{align*}
Thus
\begin{align*}
|H_1 - \mathbb E H_1 | \le \frac {\epsilon}{10},
\qquad\forall\, |t| \le 1-\epsilon_0, \, e^{\perp}\cdot e_1=0,
e^{\perp} \in \mathbb S^{n-1},
c_1\le R\le c_2.
\end{align*}

Next we deal with $H_2$. 
Observe that
\begin{align*}
g^{\prime\prime}(x) = \beta (\beta +x^2)^{-\frac 32} >0.
\end{align*}
Let $\phi \in C_c^{\infty}(\mathbb R)$ be such that
$0\le \phi(x) \le 1$ for all $x$, $\phi(x)=1$ for $|x|\le 1$ and
$\phi(x)=0$ for $|x|\ge 2$.  Then 
\begin{align*}
H_2 &=H_3+ \frac 1m \sum_{j=1}^m g^{\prime\prime}(Z_j)
\cdot \Bigl( \frac d {dt} Z_j \Bigr)^2
\cdot \Bigl( 1- \phi( \frac {Z_j}{\eta \langle X_j \rangle} ) 
\Bigr) \sqrt{R+ X_j^2},
\end{align*}
where  $H_3\ge 0$ is given by
\begin{align*}
H_3= \frac 1m \sum_{j=1}^m g^{\prime\prime}(Z_j)
\cdot \Bigl( \frac d {dt} Z_j \Bigr)^2
\cdot \phi( \frac {Z_j}{\eta \langle X_j \rangle} ) 
\sqrt{R+ X_j^2}.
\end{align*}
We first show that if $0<\eta\le \frac 12$ is taken sufficiently small, then
\begin{align} \label{A31_e1}
\mathbb E H_3 \le \frac {\epsilon} {20};
\end{align}
and with probability at least $1-O(m^{-2})$ ,
\begin{align} \label{A31_e2}
H_3 \le \frac {\epsilon} {20}, \qquad\forall\, |t| \le 1-\epsilon_0,
\quad e^{\perp}\cdot e_1=0, e^{\perp} \in \mathbb S^{n-1},
c_1\le R\le c_2.
\end{align}
Here we stress that since $H_3\ge 0$, if we only care about the lower bound, we can
just discard it in order to obtain a high-in-probability statement. On the other hand,
to get a two-way bound of $H_3$, we need to work with weaker statements due
to the high-moment terms (i.e. more than quadratic) of $X_j$ in $H_3$.

Recall that
\begin{align*}
\Bigl| \frac d {dt} Z_j \Bigr|& =
\Bigl| X_j - \frac t {\sqrt{1-t^2} } Y_j \Bigr|
=\Bigl| \frac 1 {1-t^2} X_j - \frac t{1-t^2} Z_j \Bigr| \notag \\
&\lesssim  |X_j| + |Z_j|.
\end{align*}
We have
\begin{align*}
H_3 &\lesssim \frac 1m \sum_{j=1}^m \phi(\frac {Z_j} {\eta \langle X_j \rangle})
\langle Z_j \rangle^{-3} \cdot (X_j^2 +Z_j^2) \cdot 
(1+|X_j|) \notag \\
& \lesssim \frac 1m \sum_{j=1}^m \phi(\frac {Z_j} {\eta \langle X_j \rangle})
(1+|X_j|^3). \notag 
\end{align*}
Let $K=\eta^{-\frac 1 8}$. Then
\begin{align*}
H_3 \lesssim K^3 \frac 1m   \sum_{j=1}^m \phi(\frac {Z_j} {\eta \langle 2K \rangle} )
+\frac 1m \sum_{j=1}^m  (1+|X_j|^3)  (1- \phi(\frac {X_j} {K} ) ). 
\end{align*}
Clearly then for $\eta>0$ sufficiently small, 
\begin{align*}
\mathbb E H_3
&\lesssim K^3 \mathbb E \phi(\frac {Z_1} {\eta \langle 2K \rangle } )
+\mathbb E (1+|X_1|^3) (  1- \phi(\frac {X_1} K) ) \notag \\
&\le \frac {\epsilon}{20}.
\end{align*}
For $m\gtrsim n$, it holds with high probability that
\begin{align*}
&\Bigl| 
\frac 1m \sum_{j=1}^m \phi(\frac {Z_j} {\eta \langle2 K\rangle })
- \operatorname{mean} \Bigr| \le   \eta.
\end{align*}
On the other hand, by Lemma \ref{bp7}, it holds with probability at least $1-O(m^{-2})$ that
\begin{align*}
&\Bigl| 
\frac 1m \sum_{j=1}^m  (1+|X_j|^3)  (1- \phi(\frac {X_j} {K} ) )
- \operatorname{mean} \Bigr| \le \eta, \qquad \forall\, u
\in \mathbb S^{n-1}. 
\end{align*}
Thus for $\eta>0$ sufficiently small,  \eqref{A31_e1} and \eqref{A31_e2} hold.
Now we consider the main piece
\begin{align*}
H_4 =\frac 1m \sum_{j=1}^m g^{\prime\prime}(Z_j)
\cdot \Bigl( \frac d {dt} Z_j \Bigr)^2
\cdot \Bigl( 1- \phi( \frac {Z_j}{\eta \langle X_j \rangle} ) 
\Bigr) \sqrt{R+ X_j^2}.
\end{align*}
By using 
\begin{align*}
\Bigl( \frac d{dt } Z_j \Bigr)^2
&= \Bigl( \frac 1 {1-t^2} X_j - \frac t {1-t^2} Z_j \Bigr)^2 \notag \\
&= \frac 1{(1-t^2)^2} ( X_j^2 +t^2 Z_j^2 - 2t X_j Z_j).
\end{align*}
Then
\begin{align*}
H_4 
&= (1-t^2)^{-2} \frac 1m \sum_{j=1}^m g^{\prime\prime}(Z_j)
\Bigl( 1- \phi( \frac {Z_j}{\eta \langle X_j \rangle} ) 
\Bigr)  X_j^2 \sqrt{R+ X_j^2} \notag \\
& \quad + (1-t^2)^{-2} t^2 
 \frac 1m \sum_{j=1}^m g^{\prime\prime}(Z_j) Z_j^2
\Bigl( 1- \phi( \frac {Z_j}{\eta \langle X_j \rangle} ) 
\Bigr)  \sqrt{R+ X_j^2} \notag \\
& \quad - 2t (1-t^2)^{-2}
 \frac 1m \sum_{j=1}^m g^{\prime\prime}(Z_j) Z_j 
\Bigl( 1- \phi( \frac {Z_j}{\eta \langle X_j \rangle} ) 
\Bigr) X_j \sqrt{R+ X_j^2}.
\end{align*}
Define 
\begin{align*}
h_j(x) = g^{\prime\prime}(x) \cdot (1- \phi( \frac x {\eta \langle X_j \rangle } ) ).
\end{align*}
Clearly, thanks to the cut-off $1-\phi$, we have
\begin{align*}
&\|h_j  \|_{\infty} \lesssim \langle X_j \rangle^{-3};  \\
&\|h_j^{\prime} \|_{\infty} \lesssim  \langle X_j \rangle^{-4}.
\end{align*}
It is then easy to check that the summands in $H_4$ are bounded.  Moreover
\begin{align*}
|h_j(a_j\cdot u) -h_j(a_j\cdot \tilde u)|
\lesssim |a_j\cdot (u-\tilde u)| \cdot \langle X_j \rangle^{-4}.
\end{align*}
Similar bounds also hold for the other summands in $H_4$. Thus by a similar
union bound argument as in Lemma \ref{A30_1} (and taking care of the
covering in the $R$-variable), we have with high probability that
\begin{align*}
|H_4 -\mathbb EH_4| \le \frac {\epsilon}{20},
\qquad\forall\, |t|\le 1-\epsilon_0,
e^{\perp}\cdot e_1=0, e^{\perp} \in \mathbb S^{n-1},
c_1\le R\le c_2.
\end{align*}
Collecting all the estimates, 
we then obtain the desired estimate for $\partial_{tt} h$. 
\end{proof}

\begin{lem} \label{leAu31_60}
Let $X\sim \mathcal N(0,1)$, $Y\sim \mathcal N (0,1)$ be independent. Define
\begin{align*}
&H(\rho,s) =  \mathbb E
\sqrt{\beta +(\sqrt{1-s^2} X+ s Y)^2} \sqrt{\beta \rho^2 +X^2};\\
&h (\rho, s)= \frac 12 (1+2\beta) \rho^2 - \rho H(\rho,s).
\end{align*}
Then it holds that
\begin{align*}
\sup_{|\rho-1|\ll 1, |s|\ll 1} \sum_{j=1}^3( | \partial^j H |
+|\partial^j h|) \lesssim 1.
\end{align*}
where $\partial= \partial_{\rho}$ or $\partial_s$.

\end{lem}

\begin{proof}
Clearly it suffices for us to prove the estimate for $H$ since the estimate for $h$ will
follow from it. 
We first deal with $\partial_{sss} H$ which appears to be the most difficult case and
simultaneously $\partial_s H$, $\partial_{ss} H$. In some terms we shall even
exhibit ($\beta$, $\rho$)-independent bounds which will be of interest for
future investigations. 
Denote $A=\sqrt{1-s^2} x+ sy$. Then 
\begin{align*}
& \partial_s A = - \frac s {\sqrt{1-s^2}} x + y; \\
& \partial_y A = s, \qquad \partial_x A = \sqrt{1-s^2};\\
&  \partial_s A = \frac 1 {\sqrt{1-s^2}} (-x \partial_y +y \partial_x )A. 
\end{align*}
Now we have
\begin{align*}
2\pi H = \int \sqrt{\beta +A^2} \sqrt{ \beta \rho^2 +x^2} e^{-\frac{x^2+y^2}2} dx dy.
\end{align*}
Since 
\begin{align*}
&\partial_s ( \sqrt{\beta +A^2} )
= \frac 1 {\sqrt{1-s^2}}
(-x \partial_y +y \partial_x ) ( \sqrt{\beta+A^2} ), \\
&(-x \partial_y +y\partial_x) ( e^{-\frac {x^2+y^2}2} )=0,
\end{align*}
we obtain (by using integration by parts) that
\begin{align*}
2\pi \partial_s H &=\frac 1 {\sqrt{1-s^2}} \int
\sqrt{\beta +A^2} (-y \partial_x)(\sqrt{\beta \rho^2+x^2}) e^{-\frac{x^2+y^2}2} dx dy
\\
&= - \frac 1 {\sqrt{1-s^2}}
\int \sqrt{\beta +A^2}
\frac {xy}{\sqrt{\beta \rho^2+x^2}} e^{-\frac{x^2+y^2}2} dx dy.
\end{align*}
Note that the pre-factor $\frac 1 {\sqrt{1-s^2} }$ is smooth in the regime $|s|\ll 1$, therefore
to compute the higher order $\partial_s$-derivatives of $H$, 
it suffices for us to treat 
\begin{align*}
H_1 = \int \sqrt{\beta +A^2}
\frac {xy}{\sqrt{\beta \rho^2+x^2}} e^{-\frac{x^2+y^2}2} dx dy.
\end{align*}
Then 
\begin{align*}
\partial_s H_1 = \frac 1 {\sqrt{1-s^2}} 
\int \sqrt{\beta +A^2}
(x \partial_y - y \partial_x ) ( \frac{xy}{\sqrt{\beta\rho^2+x^2}})
e^{-\frac{x^2+y^2}2} dx dy.
\end{align*}
The most difficult term is the piece corresponding to $y \partial_x$. Thus we consider
\begin{align*}
H_2 &= \int \sqrt{\beta +A^2}
( y \partial_x ) ( \frac{xy}{\sqrt{\beta\rho^2+x^2}})
e^{-\frac{x^2+y^2}2} dx dy \notag \\
&= \int \sqrt{\beta +A^2} \cdot
y^2  \cdot  \frac {\beta \rho^2 }{  ({\beta\rho^2+x^2})^{\frac 32} } 
e^{-\frac{x^2+y^2}2} dx dy. \notag 
\end{align*}
Thus
\begin{align*}
\partial_s H_2
= \int \frac A {\sqrt{\beta +A^2}} (-\frac s {\sqrt{1-s^2}} x +y)
y^2  \cdot  \frac {\beta \rho^2 }{  ({\beta\rho^2+x^2})^{\frac 32} } 
e^{-\frac{x^2+y^2}2} dx dy.  \notag 
\end{align*}
The piece corresponding to $-\frac s {\sqrt{1-s^2}} x $ is clearly fine. 
So we only need to treat 
\begin{align*}
H_3 = 
 \int \frac A {\sqrt{\beta +A^2}} 
y^3  \cdot  \frac {\beta \rho^2 }{  ({\beta\rho^2+x^2})^{\frac 32} } 
e^{-\frac{x^2+y^2}2} dx dy.
\end{align*}
Observe that for $\eta>0$,
\begin{align*}
\int_{\mathbb R} \frac {\eta^2} {(\eta^2 +x^2)^{\frac 32} } dx = 
\int_{\mathbb R} \frac  1 {(1+x^2)^{\frac 32}} dx.
\end{align*}
Thus $H_3$ is bounded by an absolute constant. Collecting the estimates, we have
\begin{align*}
|\partial_{sss} H | \lesssim 1.
\end{align*}

Now we deal with $\partial_{\rho}H$, 
$\partial_{\rho\rho}H$, and $\partial_{\rho\rho\rho} H$. This case is easy. Denote 
$B=\beta \rho^2 +x^2$. Then
\begin{align*}
& \partial_{\rho} (\sqrt{B}) =  B^{-\frac 12} \beta \rho; \\
& \partial_{\rho\rho}(\sqrt {B})
=  B^{-\frac 12} \beta -B^{-\frac 32} \beta^2 \rho^2; \\
& \partial_{\rho\rho\rho}(\sqrt{B})
= -B^{-\frac 32} \beta^2 \rho +3 B^{-\frac 52} (\beta \rho)^3
- B^{-\frac 32} \beta^2 2\rho.
\end{align*}
Clearly all terms are bounded and we have
\begin{align*}
|\partial_{\rho} H| +|\partial_{\rho\rho} H|+ |\partial_{\rho\rho\rho}H| \lesssim 1.
\end{align*}

Next clearly $\partial_{\rho s} H$ and $\partial_{\rho\rho s}H $ are OK.
We only need to treat $\partial_{\rho ss} H$.  The main
term of $\partial_{ss} H$ is 

\begin{align*}
H_4 = 
\int \frac A {\sqrt{\beta +A^2}} (-\frac s {\sqrt{1-s^2}} x +y)
\frac{xy}{\sqrt{\beta\rho^2+x^2}} 
e^{-\frac{x^2+y^2}2} dx dy.
\end{align*}
Now 
\begin{align*}
\partial_{\rho} H_4
=  -\int \frac A {\sqrt{\beta +A^2}} (-\frac s {\sqrt{1-s^2}} x +y)
\frac{xy \beta \rho}{(\beta\rho^2+x^2 )^{\frac 32}} 
e^{-\frac{x^2+y^2}2} dx dy.
\end{align*}
Clearly for any $0<\eta\lesssim 1$,
\begin{align*}
\int \frac {\eta^2 |x|} { (\eta^2+x^2)^{\frac 32} }dx 
={\eta} \int \frac {|x|} {(1+x^2)^{\frac 32} } dx <\infty.
\end{align*}
Thus $\partial_{\rho ss} H$ is also OK for us. 
\end{proof}
\begin{lem}[Calculation of $\partial^2 h$ at ($\rho=1$, $s=0$)] \label{leAu31_61}
Let
\begin{align*}
&H(\rho,s) =  \mathbb E
\sqrt{\beta +(\sqrt{1-s^2} X+ s Y)^2} \sqrt{\beta \rho^2 +X^2}; \\
&h (\rho, s)= \frac 12 (1+2\beta) \rho^2 - \rho H(\rho,s).
\end{align*}
Then at $\rho=1$, $s=0$, we have
\begin{align*}
&(\partial_{\rho\rho} H)(1, 0) =- \gamma_1<0, \quad 
(\partial_{\rho s} H)(1,0)=0; \\
&(\partial_{ss} H)(1, 0)= -\gamma_2<0; \\
& (\partial_s h)(\rho,0)=0, \forall\, \rho>0, \quad (\partial_{\rho} h)(1,0)=0;\\
&(\partial_{\rho\rho} h)(1, 0) =\gamma_3>0, \quad
(\partial_{\rho s} h)(1,0)=0;\\
&(\partial_{ss} h)(1, 0)= \gamma_4>0, 
\end{align*}
where $\gamma_i>0$, $i=1,\cdots, 4$ are constants depending on $\beta$. 

\end{lem}

\begin{proof}
\texttt{\underline{Calculation of $\partial_{ss} H$}}. 

Denote $A=\sqrt{1-s^2} x+ sy$. Then 
\begin{align*}
& \partial_s A = - \frac s {\sqrt{1-s^2}} x + y, \quad
 \partial_y A = s, \qquad \partial_x A = \sqrt{1-s^2};\\
&  \partial_s A = \frac 1 {\sqrt{1-s^2}} (-x \partial_y +y \partial_x )A. 
\end{align*}
Now we have
\begin{align*}
2\pi H = \int \sqrt{\beta +A^2} \sqrt{ \beta \rho^2 +x^2} e^{-\frac{x^2+y^2}2} dx dy.
\end{align*}
Then 
\begin{align*}
2\pi \partial_s H &=\frac 1 {\sqrt{1-s^2}} \int
\sqrt{\beta +A^2} (-y \partial_x)(\sqrt{\beta \rho^2+x^2}) e^{-\frac{x^2+y^2}2} dx dy
\\
&= - \frac 1 {\sqrt{1-s^2}}
\int \sqrt{\beta +A^2}
\frac {xy}{\sqrt{\beta \rho^2+x^2}} e^{-\frac{x^2+y^2}2} dx dy.
\end{align*}

One should observe that $\partial_s H\Bigr |_{\rho>0,s=0}=0$. 

Then
\begin{align*}
2\pi \partial_{ss} H \Bigr|_{\rho=1,s=0}
&= - \int \frac {A}{\sqrt {\beta +A^2} }\Bigr|_{s=0}  \cdot \frac {x y^2}{\sqrt{\beta +x^2}} e^{-\frac {x^2+y^2} 2} 
dx dy \notag \\
&= - \int \frac {x^2 y^2} {\beta +x^2} e^{-\frac {x^2+y^2}2} dx dy.
\end{align*}

\texttt{\underline{Calculation of $\partial_{\rho\rho} H$}}. 
Clearly
\begin{align*}
2\pi \partial_{\rho} H
= \int \sqrt{\beta+A^2} \frac{\beta \rho}
{\sqrt{\beta \rho^2+x^2}} e^{-\frac{x^2+y^2}2} dx dy.
\end{align*}

Observe that
\begin{align*}
2\pi \partial_{\rho} H\Bigr|_{\rho=1,s=0}
= \int \beta e^{-\frac {x^2+y^2} 2} dx dy.
\end{align*}

Then
\begin{align*}
2\pi \partial_{\rho\rho} H \Bigr|_{\rho=1,s=0}
&=  \int \sqrt{\beta+A^2} \Bigr|_{\rho=1,s=0} \frac{\beta }
{\sqrt{\beta \rho^2+x^2}} e^{-\frac{x^2+y^2}2} dx dy \notag \\
& \quad -  \int \sqrt{\beta+A^2} \Bigr|_{\rho=1,s=0} \frac{ \beta^2 \rho^2}
{(\beta \rho^2+x^2)^{\frac 32} } e^{-\frac{x^2+y^2}2} dx dy \notag \\
&= \int 
\frac {\beta x^2} {\beta +x^2} e^{-\frac{x^2+y^2}2 }dxdy.
\end{align*}

\texttt{\underline{Calculation of $\partial_{\rho s} H$}}. 
We have
\begin{align*}
2 \pi
\partial_{\rho s} H \Bigr|_{\rho=1,s=0}
&= \int \frac {A}{\sqrt{\beta +A^2}} \Bigr|_{s=0} y \frac {\beta \rho}
{\sqrt{\beta \rho^2 +x^2}} e^{-\frac {x^2+y^2} 2} dx dy \notag \\
& = \int  \frac{\beta xy}{\beta +x^2} e^{-\frac {x^2+y^2} 2}
dx dy =0.
\end{align*}

Now we calculate the corresponding Hessian for $
h = \frac 12 (1+2\beta) \rho^2 - \rho H$.  Clearly
\begin{align*}
\partial_{\rho\rho} h \Bigr|_{\rho=1,s=0}
&= 1+2\beta - ( \partial_{\rho\rho} H + 2 \partial_{\rho} H)
=1+2\beta- \frac 1 {2\pi}\int ( \frac {\beta x^2}{\beta +x^2}
+2\beta) e^{-\frac{x^2+y^2} 2} dx dy \notag \\
& = \frac 1 {2\pi}\int  (1-\frac {\beta x^2}{\beta +x^2}) 
e^{-\frac {x^2+y^2} 2} dx dy.
\end{align*}
By Lemma \ref{lemE.7} this is clearly positive and has a lower bound depending only 
in terms of $\beta$. 

On the other hand, 
\begin{align*}
\partial_{\rho s} h \Bigr|_{\rho=1,s=0}
= - \partial_s H - \partial_{\rho s} H =0.
\end{align*}
Finally
\begin{align*}
\partial_{ss } h = - \partial_{ss } H
= \frac 1 {2\pi} \int \frac {x^2 y^2} {\beta +x^2} e^{-\frac {x^2+y^2}2} dx dy
>0.
\end{align*}
\end{proof}

\begin{lem} \label{lemE.7}
For any $0<\beta <\infty$, we have
\begin{align*}
\int ( 1- \frac{\beta x^2} {\beta +x^2} ) e^{-\frac {x^2} 2} dx >0.
\end{align*}
\end{lem}
\begin{proof}
For $0<\beta \le 1$, this is obvious. For $\beta>1$, denote $\epsilon=\frac 1 {\beta}$.
Then
\begin{align*}
&\tilde h(\epsilon) = \int ( 1- \frac {x^2} {1+\epsilon x^2}) e^{-\frac {x^2} 2} dx;\\
&\tilde h^{\prime}(\epsilon)
= \int \frac { \epsilon x^4} {(1+\epsilon x^2)^2 } e^{-\frac {x^2} 2} dx >0.
\end{align*}
Clearly $\tilde h(0)=0$. Thus $\tilde h(\epsilon)>0$ for all $0<\epsilon<\infty$.
\end{proof}
\begin{lem} \label{Sp1_1_0.1}
Let $0<c_1<c_2<\infty$ be fixed. Consider for $\xi \in \mathbb S^{n-1}$,
$u\in \mathbb R^n$ with $c_1\le \|u\|_2 \le c_2$, the following:
\begin{align*}
&I_1= I_1(\xi, u)=\frac 1m \sum_{k=1}^m
(\beta |u|^2+ (a_k\cdot u)^2 )^{-\frac 32}
(\beta |u|^2 + (a_k\cdot e_1) )^{\frac 12}
\cdot (a_k \cdot \xi)^2, \\
&I_2= I_2(\xi,u)=\frac 1m \sum_{k=1}^m
(\beta |u|^2+ (a_k\cdot u)^2 )^{-\frac 32}
(\beta |u|^2 + (a_k\cdot e_1) )^{\frac 12}
\cdot (a_k\cdot u) \cdot (a_k \cdot \xi).
\end{align*}
For any $0<\epsilon\le 1$, if $m\gtrsim n$, then it holds with probability
at least $1-O(m^{-2})$ that
\begin{align*}
&|I_1 -\mathbb E I_1 | \le \epsilon, 
\qquad \forall\, \xi \in \mathbb S^{n-1}, \quad \forall\, c_1\le \| u\|_2 \le c_2; \\
&|I_2 -\mathbb E I_2 | \le \epsilon, 
\qquad \forall\, \xi \in \mathbb S^{n-1}, \quad \forall\, c_1\le \| u\|_2 \le c_2.
\end{align*}
\end{lem}
\begin{proof}[Proof of Lemma \ref{Sp1_1_0.1}]
We first note that, in order to prove the statement for $I_2$, 
 it suffices for us to prove the statement for $I_1$ under 
a more general condition (instead of $\xi \in \mathbb S^{n-1}$):
\begin{align*}
 \|\xi\|_2 \le c_3:=2+c_2.
 \end{align*}
The reason is as follows. By using the simple identity
\begin{align*}
(a_k \cdot (\xi+u) )^2 = (a_k\cdot \xi)^2 +(a_k\cdot u)^2 +2
(a_k\cdot \xi ) (a_k \cdot u),
\end{align*}
we have
\begin{align*}
I_2(\xi, u)= \frac 12 I_1(\xi+u, u) -\frac 12 I_1(\xi, u) - I_3,
\end{align*}
where
\begin{align*}
I_3 = \frac 1m \sum_{k=1}^m (
\beta |u|^2+(a_k\cdot u)^2)^{-\frac 32} (a_k\cdot u)^2
(\beta |u|^2+ (a_k\cdot e_1)^2)^{\frac 12}.
\end{align*}
Clearly $I_3$ is OK for union bounds and we have with high probability 
\begin{align*}
|I_3 -\mathbb I_3 | \le \epsilon,
\qquad\forall\,  c_1 \le \| u\|_2 \le c_2.
\end{align*}
Thus to prove the statement for $I_2$ it suffices for us to prove it for $I_1$
uniformly in $\xi$ with $\|\xi\|_2\le c_3$. 

Next we observe that for $\xi \ne 0$ with $\| \xi \|_2 \le c_3$, we have
\begin{align*}
|I_1(\xi, u) - \mathbb E I_1(\xi, u) |
&\le\;  \|\xi\|_2 |I_1( \frac {\xi}{\|\xi\|_2}, u) - \mathbb E I_1(\frac {\xi}{\|\xi\|_2}, u) | 
\notag \\
& \le\; c_3 |I_1( \frac {\xi}{\|\xi\|_2}, u) - \mathbb E I_1(\frac {\xi}{\|\xi\|_2}, u) |.
\end{align*}
Thus it suffices for us to prove the statement for $I_1$ under the original assumption
$\xi \in \mathbb S^{n-1}$. 

Now let $\phi \in C_c^{\infty}(\mathbb R)$ be such that 
$0\le \phi(x) \le 1$ for all $x$, $\phi(x)=1$ for $|x|\le 1$ and $\phi(x)=0$ for 
$|x|\ge 2$. Let $\delta>0$ be a sufficiently small constant. The needed smallness
will be specified later. We write (below $\langle x \rangle =(1+x^2)^{\frac 12}$)
\begin{align*}
I_1 &= \frac 1m \sum_{k=1}^m
(\beta |u|^2+(a_k\cdot u)^2)^{-\frac 32}
(\beta |u|^2 +(a_k\cdot e_1)^2)^{\frac 12}\cdot (a_k\cdot \xi)^2
\cdot \phi(\frac {a_k \cdot u} {\delta \langle a_k \cdot \xi \rangle} ) \notag \\
&\quad+ \frac 1m \sum_{k=1}^m
(\beta |u|^2+(a_k\cdot u)^2)^{-\frac 32}
(\beta |u|^2 +(a_k\cdot e_1)^2)^{\frac 12}\cdot (a_k\cdot \xi)^2
\cdot \Bigl( 1-\phi(\frac {a_k \cdot u} {\delta \langle a_k \cdot \xi \rangle} ) \Bigr) \notag \\
&=:\, I_{1,a}+I_{1,b}.
\end{align*}

\underline{Estimate of $I_{1,a}$}. 
Let $K=\delta^{-\frac 19}$. Then
\begin{align*}
|I_{1,a}|
& \le \frac 1m \sum_{k=1}^m
(\beta |u|^2+(a_k\cdot u)^2)^{-\frac 32}
(\beta |u|^2 +(a_k\cdot e_1)^2)^{\frac 12}\cdot (a_k\cdot \xi)^2
\cdot \phi(\frac {a_k \cdot u} {\delta \langle a_k \cdot \xi \rangle} ) 
\phi(\frac {a_k \cdot \xi} K) \notag \\
& \quad +\frac 1m \sum_{k=1}^m
(\beta |u|^2+(a_k\cdot u)^2)^{-\frac 32}
(\beta |u|^2 +(a_k\cdot e_1)^2)^{\frac 12}\cdot (a_k\cdot \xi)^2
\cdot \phi(\frac {a_k \cdot u} {\delta \langle a_k \cdot \xi \rangle} ) 
\cdot \Bigl(1-\phi(\frac {a_k\cdot \xi} K) \Bigr) \notag \\
& \lesssim
\; K^2 \frac 1m \sum_{k=1}^m 
(1+(a_k\cdot e_1)^2)^{\frac 12} 
\phi(\frac {a_k\cdot u} {\delta \langle 2K \rangle} )
+ \frac 1m \sum_{k=1}^m (1+(a_k\cdot e_1)^2)^{\frac 12}
\cdot (a_k\cdot \xi)^2 \cdot 
\Bigl(1-\phi(\frac {a_k\cdot \xi} K) \Bigr) \notag \\
& \lesssim \; K^5 \frac 1m \sum_{k=1}^m 
\phi(\frac {a_k\cdot u} {\delta \langle 2K \rangle} )
+ \frac 1m \sum_{k=1}^m (1+(a_k\cdot e_1)^2) 
\cdot  K^{-1}  + K\cdot \frac 1m \sum_{k=1}^m 
(a_k\cdot \xi)^4 \cdot 
\Bigl(1-\phi(\frac {a_k\cdot \xi} K) \Bigr). 
\end{align*}
Clearly for sufficiently small $\delta$, we have
\begin{align*}
\mathbb E |I_{1,a}| \le \frac {\epsilon}{10}.
\end{align*}
Furthermore, with probability at least $1-O(m^{-2})$, we have
\begin{align*}
|I_{1,a}| \le \frac {\epsilon}{10}, \qquad
\forall\, c_1\le \|u\|_2\le c_2, \;\forall\, \xi \in \mathbb S^{n-1}.
\end{align*}

\underline{Estimate of $I_{1,b}$}.
Thanks to the cut-off $1-\phi(\frac {a_k \cdot u} {\delta \langle a_k \cdot \xi \rangle} )$,
we have $|a_k\cdot u|\gtrsim \langle a_k\cdot \xi \rangle$ on its support. 
It is then easy to check that the summands in $I_{1,b}$ are sub-exponential random
variables. It remains for us to check the union bound. 

To this end, take $u$, $\tilde u$ with $c_1\le \|u\|_2, \|\tilde u\|_2 \le c_2$,
and $\xi$, $\tilde \xi \in \mathbb S^{n-1}$. Then clearly 
\begin{align*}
&\biggl| (\beta |u|^2+(a_k\cdot u)^2)^{-\frac 32}
(\beta |u|^2 +(a_k\cdot e_1)^2)^{\frac 12}\cdot (a_k\cdot \xi)^2
\cdot \Bigl( 1- \phi(\frac {a_k \cdot u} {\delta \langle a_k \cdot \xi \rangle} ) 
\Bigr)
  \notag \\
& \quad - (\beta |\tilde u|^2+(a_k\cdot \tilde u)^2)^{-\frac 32}
(\beta |\tilde u|^2 +(a_k\cdot e_1)^2)^{\frac 12}\cdot (a_k\cdot \tilde {\xi})^2
\cdot \Bigl( 1-\phi(\frac {a_k \cdot \tilde u} {\delta \langle a_k \cdot \tilde {\xi} \rangle} ) 
\Bigr)  \biggr| \notag \\
\lesssim &\; (1+|a_k\cdot e_1|)\cdot \Bigl(\|u-\tilde u\|_2+
|a_k\cdot (u-\tilde u)| + |a_k\cdot (\xi -\tilde {\xi} )|\Bigr).
\end{align*}
Here in the above derivation we have used the fact that the function (it differs from the actual one
by some minor change of parameters)
\begin{align*}
G(t,s)= \langle t\rangle^{-3} s^2  \Bigl(1- \phi(\frac t {\langle s \rangle} ) \Bigr)
\end{align*}
satisfies
\begin{align*}
|G(t,s)-G(\tilde t, \tilde s)| \lesssim |t-\tilde t| +|s-\tilde s|.
\end{align*}
It is then clear that $I_{1,b}$ is OK for union bounds and we have with high probability
\begin{align*}
| I_{1,b}-\mathbb E I_{1,b}| \le \frac {\epsilon}{10},
\qquad\forall\, c_1\le \|u\|_2\le c_2, \, \forall\, \xi \in S^{n-1}.
\end{align*}
The desired estimate for $I_1$ then easily follows. 
\end{proof}

\begin{lem} \label{Sp1_1}
Let $0<c_1<c_2<\infty $ be fixed. Consider
\begin{align*}
f_0 (u) = -\frac 1 m \sum_{k=1}^m \sqrt{\beta |u|^2 + (a_k \cdot u)^2}
\sqrt{\beta |u|^2 + (a_k\cdot e_1)^2}.
\end{align*}
For any $0<\epsilon\le 1$, if $m\gtrsim n$, 
then it holds with probability at least $1-O(m^{-2})$ that 
\begin{align*}
&\Bigl| \sum_{i,j=1}^n \xi_i \xi_j 
(\partial_{ij} f_0)(u) - 
\sum_{i,j=1}^n \xi_i \xi_j 
\mathbb E(\partial_{ij} f_0)(u) \Bigr| \le \epsilon, 
\qquad \forall\, \xi \in \mathbb S^{n-1}, \quad \forall\, c_1\le \| u\|_2 \le c_2.
\end{align*}
\end{lem}
\begin{proof}[Proof of Lemma \ref{Sp1_1}]
To simplify the notation, write $a_k$ as $a$, and denote
\begin{align*}
&A= \beta |u|^2 + (a\cdot u)^2, \qquad B= \beta |u|^2 + (a\cdot e_1)^2; \\
& \partial_i A= 2 \beta u_i + 2 (a\cdot u)a_i, \qquad \partial_{ij}A
= 2\beta \delta_{ij} + 2 a_i a_j; \\
& \partial_i B= 2 \beta u_i, \qquad \partial_{ij} B= 2\beta \delta_{ij}.
\end{align*}

We need to compute $\partial_{ij} \tilde F$ for 
\begin{align*}
\tilde F =  A^{\frac 12} B^{\frac 12}.
\end{align*}
Clearly
\begin{align*}
&\partial_i \tilde F = \frac 12 A^{-\frac 12} \partial_i A B^{\frac 12}
+ \frac 12 B^{-\frac 12} \partial_i B A^{\frac 12};\\
&\partial_{ij} \tilde F
= -\frac 14 A^{-\frac 32} \partial_i A \partial_j A B^{\frac 12} + \frac 12 A^{-\frac 12} \partial_{ij} A B^{\frac 12} + \frac 12 A^{-\frac 12} \partial_i A \frac 12 B^{-\frac 12}
\partial_j B \notag \\
&\quad  \qquad -\frac 14 B^{-\frac 32} \partial_j B \partial_i B A^{\frac 12} + \frac 12 B^{-\frac 12} \partial_{ij} B A^{\frac 12} + \frac 14 B^{-\frac 12} A^{-\frac 12} \partial_i B \partial_j A. 
\end{align*}
We then have
\begin{align}
  & \sum_{i,j=1}^n \xi_i \xi_j (\partial_{ij} f_0)(u) \notag \\
=&\; \frac 14 \cdot  \frac 1 m \sum_{k=1}^m  A_k^{-\frac 32}
|\xi \cdot \nabla A_k|^2  B_k^{\frac 12} \label{Sp1_1.1} \\
& \; + \frac 14 \cdot \frac 1m \sum_{k=1}^m
B_k^{-\frac 32} |\xi \cdot \nabla B_k|^2 A_k^{\frac 12} \label{Sp1_1.2} \\
& \; - \frac 12\cdot \frac 1m \sum_{k=1}^m A_k^{-\frac 12}
\langle \xi, (\nabla^2 A_k) \xi \rangle B_k^{\frac 12} \label{Sp1_1.3} \\
& \; - \frac 12 \cdot \frac 1m \sum_{k=1}^m 
A_k^{-\frac 12} B_k^{-\frac 12} (\nabla A_k \cdot \xi) (\nabla B_k \cdot \xi) 
\label{Sp1_1.4} \\
& \; - \frac 12 \cdot \frac 1m \sum_{k=1}^m B_k^{-\frac 12} A_k^{\frac 12}
\langle \xi, (\nabla^2 B_k) \xi \rangle, \label{Sp1_1.5} 
\end{align}  
where $A_k=\beta |u|^2 + (a_k\cdot u)^2$, $B_k=
\beta |u|^2 +(a_k\cdot e_1)^2$, and we have denoted 
\begin{align*}
\langle \xi, (\nabla^2 A_k) \xi \rangle = \sum_{i,j=1}^n \xi_i \xi_j 
\partial_{ij} A_k.
\end{align*}

\underline{Estimate of \eqref{Sp1_1.5}}. We have
\begin{align*}
 & \frac 1m \sum_{k=1}^m B_k^{-\frac 12} A_k^{\frac 12} 
 \langle \xi, (\nabla^2 B_k ) \xi \rangle \notag \\
 =&\; 2\beta |\xi|^2 \Bigl( \frac 1m \sum_{k=1}^m B_k^{-\frac 12} A_k^{\frac 12}
 \Bigr).
 \end{align*}
 The summand consists of sub-exponential random variables and are clearly OK
 for union bounds. Thus
  with high probability, it holds that
 \begin{align*}
 \Bigl| \frac 1m \sum_{k=1}^m
 B_k^{-\frac 12} A_k^{\frac 12}
 -\operatorname{mean} \Bigr| 
 \le \frac {\epsilon}{100(1+2\beta)},
 \qquad \forall\, c_1\le \|u\|_2\le c_2.
 \end{align*}
 Thus the contribution of \eqref{Sp1_1.5} is OK for us. 

\underline{Estimate of \eqref{Sp1_1.2}}. We have
\begin{align*}
 & \frac 1m \sum_{k=1}^m B_k^{-\frac 32} A_k^{\frac 12} 
 |\xi \cdot \nabla B_k|^2  \notag \\
 =&\; 4\beta^2 (\xi \cdot u)^2 \Bigl( \frac 1m \sum_{k=1}^m B_k^{-\frac 32} A_k^{\frac 12}
 \Bigr).
 \end{align*}
 Again the summand consists of sub-exponential random variables and are clearly OK
 for union bounds. Thus
  with high probability, it holds that
 \begin{align*}
 \Bigl| \frac 1m \sum_{k=1}^m
 B_k^{-\frac 32} A_k^{\frac 12}
 -\operatorname{mean} \Bigr| 
 \le \frac {\epsilon}{100(1+4\beta^2 c_2^2)},
 \qquad \forall\, c_1\le \|u\|_2\le c_2.
 \end{align*}
 Thus the contribution of \eqref{Sp1_1.2} is OK for us.

\underline{Estimate of \eqref{Sp1_1.4}}. We have
\begin{align}
& \frac 1m \sum_{k=1}^m A_k^{-\frac 12}
B_k^{-\frac 12} (\nabla A_k \cdot \xi) (\nabla B_k \cdot \xi) \notag \\
=&\;
\frac 1m \sum_{k=1}^m A_k^{-\frac 12} B_k^{-\frac 12}
\cdot \Bigl( 2\beta (u\cdot \xi) + 2(a_k\cdot u) (a_k\cdot \xi) \Bigr)
2\beta (\xi \cdot u) \notag \\
=&\; 4\beta^2 (\xi\cdot u)^2 \frac 1m \sum_{k=1}^m
A_k^{-\frac 12} B_k^{-\frac 12} \notag \\
& \qquad + 4\beta(\xi \cdot u)
\frac 1m \sum_{k=1}^m A_k^{-\frac 12} B_k^{-\frac 12}
(a_k \cdot \xi) (a_k\cdot u). \label{Sp1_1.4b}
\end{align}
The first term is clearly under control and therefore we focus only on
\eqref{Sp1_1.4b}. For this observe that for any
$u$, $\tilde u$ with $c_1\le \|u\|_2, \|\tilde u\|_2\le c_2$,
$\xi$, $\tilde \xi \in \mathbb S^{n-1}$, it holds that
\begin{align*}
& \Bigl| \frac {a_k\cdot u} {\sqrt{\beta |u|^2+ |a_k\cdot u |^2} }
-\frac {a_k\cdot \tilde u} {\sqrt{ \beta | u|^2 + |a_k\cdot \tilde u|^2 } }
\Bigr| \lesssim  | a_k\cdot (u-\tilde u) |, \notag \\
& \Bigl| \frac {a_k\cdot \tilde u} {\sqrt{\beta |u|^2+ |a_k\cdot \tilde u|^2} }
-\frac {a_k\cdot \tilde u}  {\sqrt{ \beta |\tilde u|^2 + |a_k\cdot \tilde u|^2 } }
\Bigr| \lesssim  \| u-\tilde u\|_2, \notag \\
&\Bigl|
\frac {a_k \cdot u}
{\sqrt{\beta |u|^2 + |a_k\cdot u|^2}}
\cdot \frac {a_k \cdot \xi} { \sqrt{\beta |u|^2+ |a_k\cdot e_1|^2} }
-\frac {a_k \cdot \tilde u}
{\sqrt{\beta |\tilde u|^2 + |a_k\cdot \tilde u|^2}}
\cdot \frac {a_k \cdot \tilde {\xi} } { \sqrt{\beta |\tilde u|^2+ |a_k\cdot e_1|^2} }
\Bigr| \notag \\
&\qquad 
\lesssim (|a_k\cdot (u-\tilde u)| + \|u-\tilde u\|_2) |a_k\cdot \xi|
+ 
 |a_k\cdot (\xi -\tilde {\xi})|.
\end{align*}
Thus \eqref{Sp1_1.4b} is OK for union bounds and we have
with high probability, 
\begin{align*}
\Bigl|
\frac 1m \sum_{k=1}^m A_k^{-\frac 12} B_k^{-\frac 12}
(a_k\cdot \xi) (a_k \cdot u) - \operatorname{mean} 
\Bigr| \le \frac {\epsilon} {200(1+4\beta c_2)} , 
\qquad \forall\, c_1\le \|u\|_2 \le c_2,
\, \forall\, \xi \in \mathbb S^{n-1}.
\end{align*}
Thus \eqref{Sp1_1.4} is under control. 

\underline{Estimate of \eqref{Sp1_1.3}}.
We have 
\begin{align}
& -\frac 12 \cdot \frac 1m \sum_{k=1}^m A_k^{-\frac 12}
\langle \xi, (\nabla^2 A_k) \xi \rangle B_k^{\frac 12} \notag \\
=&\;  -\beta |\xi|^2 \frac 1m \sum_{k=1}^m
A_k^{-\frac 12} B_k^{\frac 12} \notag \\
& \quad -  \frac 1m \sum_{k=1}^m A_k^{-\frac 12}
B_k^{\frac 12} (a_k \cdot \xi)^2. \label{Sp1_1.3b}
\end{align}
The first term is clearly under control. Therefore we only
need to treat \eqref{Sp1_1.3b}. We shall treat it together with
\eqref{Sp1_1.1a} below.

\underline{Estimate of \eqref{Sp1_1.1}}.   We have
\begin{align} 
 & \frac 14\cdot \frac 1m \sum_{k=1}^m A_k^{-\frac 32} |\xi \cdot \nabla A_k|^2 B_k^{\frac 12} 
 \notag \\
 =& \;   \frac 1m \sum_{k=1}^m
 A_k^{-\frac 32} B_k^{\frac 12} (a_k\cdot u)^2 (a_k\cdot \xi)^2 
 \label{Sp1_1.1a}\\
 & \; +2\beta (\xi \cdot u) \frac 1m \sum_{k=1}^m
 A_k^{-\frac 32} B_k^{\frac 12} (a_k \cdot u) (a_k \cdot \xi) \label{Sp1_1.1b} \\
 & \; + \beta^2 (\xi \cdot u)^2  \frac 1m \sum_{k=1}^m A_k^{-\frac 32}
 B_k^{\frac 12}. \label{Sp1_1.1c} 
 \end{align}
 Clearly \eqref{Sp1_1.1c} is perfectly under control. Now observe
 \begin{align*}
 \eqref{Sp1_1.3b}+ \eqref{Sp1_1.1a}
 =-\beta |u|^2 \frac 1m \sum_{k=1}^m
 A_k^{-\frac 32} B_k^{\frac 12} (a_k\cdot \xi)^2. 
 \end{align*}
 One can then apply Lemma \ref{Sp1_1_0.1} to get the desired estimate
 for this term as well as \eqref{Sp1_1.1b}. 
\end{proof}

\section{Technical estimates for Section \ref{S:model3}}

\begin{lem} \label{ApMod3E1}
Denote $X_j=a_j\cdot e_1$ and $Z_j=a_j\cdot \hat u$, where $\hat u \in \mathbb S^{n-1}$.
For any $\epsilon>0$, there exists $R=R(\epsilon,\beta)>0$, such that 
if $m \gtrsim n$, then the following hold with high probability:
\begin{align*}
 \frac 1m \sum_{j=1}^m 
(\beta+Z_j^2) \sqrt{ 
\frac {\rho^2(\beta+2Z_j^2)} {\rho^2(\beta+Z_j^2) +X_j^2} }
\le \epsilon, \qquad\forall\, 0<\rho \le R, \quad\forall\, \hat u \in \mathbb S^{n-1}.
\end{align*}
\end{lem}
\begin{proof}
We shall only sketch the proof. Choose $\phi \in C_c^{\infty}(\mathbb R)$ such 
that $0\le \phi(x) \le 1$ for all $x$, $\phi(x)=1$ for $|x|\le 1$ and $\phi(x)=0$
for $|x| \ge 2$. Then 
\begin{equation}  \label{ApMo3E1.1} 
  (\beta+Z_j^2) \sqrt{ 
\frac {\rho^2(\beta+2Z_j^2)} {\rho^2(\beta+Z_j^2) +X_j^2} }
 \le  (\beta+Z_j^2) \sqrt{ 
\frac {\rho^2(\beta+2Z_j^2)} {\rho^2(\beta+Z_j^2) +X_j^2} } \phi(\frac {Z_j} K)
+ (\beta +Z_j^2) \cdot \sqrt 2 \cdot \Bigl(1- \phi(\frac {Z_j} K) \Bigr),
\end{equation}
where $K>0$ is a constant to be specified momentarily.  Clearly by taking $K$ 
sufficiently large, we have with high probability that
\begin{align*}
\frac 1m \sum_{j=1}^m 
(\beta +Z_j^2) \cdot \sqrt 2 \cdot \Bigl(1- \phi(\frac {Z_j} K) \Bigr)
\le \frac {\epsilon}{10}, \qquad \forall\, \hat u \in \mathbb S^{n-1}. 
\end{align*}
It then remains for us to deal with \eqref{ApMo3E1.1}.  Thanks to the smooth cut-off,
we have
\begin{align*}
\eqref{ApMo3E1.1} &\le \rho C_{K,\beta}
\cdot \frac 1 {\sqrt{\rho^2\beta + X_j^2}}  \notag \\
& \le  \rho C_{K,\beta} \cdot \frac 1 {\eta} + E_{K,\beta} \cdot \phi(\frac {X_j}{\eta} ),
\end{align*}
where $C_{K,\beta}>0$, $E_{K,\beta}>0$ are constants depending only on
$K$ and $\beta$.  We first choose $\eta>0$ sufficiently small such that 
with high probability,
\begin{align*}
E_{K,\beta} \Bigl| \frac 1m \sum_{j=1}^m \phi(\frac {X_j} {\eta} ) \Bigr| 
\le \frac {\epsilon}{10}.
\end{align*}
Then the desired result follows by taking $\rho$ sufficiently small. 
\end{proof}

\begin{lem} \label{lemSep5_0}
Let $\gamma_1>0$, $\gamma_2>0$ and $\gamma_3\ge 0$. Consider
\begin{align*}
g(\theta_0)
=\int_0^{\pi}
\sqrt{\gamma_1+\gamma_2 \cos^2(\theta-\theta_0)
+\gamma_3 \sin^2 \theta}
\sqrt{\gamma_1+2\gamma_3 \sin^2 \theta} d\theta.
\end{align*}
Then 
\begin{align*}
g^{\prime}(\theta_0)\ge 0, \qquad\forall\, \theta_0 \in[0,\frac {\pi}2).
\end{align*}
Furthermore, if $\gamma_1 \sim 1$, $\gamma_2\sim1$, $\gamma_3 \ge 0$,
then 
\begin{align*}
g^{\prime}(\theta_0) \gtrsim  \frac 1 {1+\gamma_3} \sin 2\theta_0.
\end{align*}
In particular we have
\begin{align*}
g^{\prime\prime}(0) \gtrsim \frac 1{1+\gamma_3}.
\end{align*}
\end{lem}
\begin{rem}
There exists a subtle balance of coefficients in the expression of $g(\theta_0)$ without which
we cannot have the positivity of $g^{\prime}$. As a counter-example, consider
\begin{align*}
f(s,b)=\int_0^{\pi}  (1+b\cos^2(\theta-s)+2 \sin^2 \theta)^{\frac 12}
(1+\sin^2 \theta)^{\frac 12} d\theta.
\end{align*}
One can check that $\partial_s f(s,b)<0$ for $b<1.99$ and $\partial_s f(s,b)>0$ for some
 $b\ge 2$ and $s$.
\end{rem}
\begin{proof}[Proof of Lemma \ref{lemSep5_0}]
Clearly
\begin{align*}
g^{\prime}(\theta_0)
&= \gamma_2 \int_0^{\pi}
\sqrt{ \frac {\gamma_1+2\gamma_3 \sin^2 \theta}
{\gamma_1+\gamma_2 \cos^2(\theta-\theta_0)
+\gamma_3 \sin^2 \theta} } \sin 2(\theta-\theta_0) d\theta \notag \\
&=\gamma_2 
\int_0^{\pi}
\sqrt{ \frac {\gamma_1+2\gamma_3 \sin^2( \theta+\theta_0) }
{\gamma_1+\gamma_2 \cos^2\theta
+\gamma_3 \sin^2( \theta+\theta_0)} } \sin 2\theta d\theta \notag \\
&= \gamma_2 \int_0^{\frac {\pi} 2} \Bigl( 
\sqrt{ \frac {\gamma_1+2\gamma_3 \sin^2( \theta+\theta_0) }
{\gamma_1+\gamma_2 \cos^2\theta
+\gamma_3 \sin^2( \theta+\theta_0)} }
-
\sqrt{ \frac {\gamma_1+2\gamma_3 \sin^2( \theta- \theta_0) }
{\gamma_1+\gamma_2 \cos^2\theta
+\gamma_3 \sin^2( \theta-\theta_0)} }
\Bigr) \sin2\theta d\theta.
\end{align*}
Clearly for $\theta, \theta_0 \in [0,\frac {\pi} 2)$, we have 
\begin{align*}
\sin(\theta+\theta_0) \ge |\sin (\theta-\theta_0)|.
\end{align*}
The non-negativity of $g^{\prime}$ then follows from the monotonicity of
the function (below $a\ge 1$ is a constant)
\begin{align*}
\tilde g(z) = \frac {\gamma_1 +2\gamma_3 z} {a\gamma_1+ \gamma_3 z}
=2-\frac {(2a-1) \gamma_1} {\gamma_3 z+a\gamma_1}, \quad z\ge 0.
\end{align*}
Next if $\gamma_1, \gamma_2\sim 1, \gamma_3 \ge 0$, then clearly 
(note that $a=1+\frac{\gamma_2}{\gamma_1} \cos^2 \theta \ge 1$, $a\sim 1$)
\begin{align*}
{\tilde g}^{\prime}(z) \gtrsim \frac 1 {1+\gamma_3}, \qquad \forall\,
z\in [0,1].
\end{align*}
Thus
\begin{align*}
g^{\prime}(\theta_0) &\gtrsim \frac 1{ 1+\gamma_3} 
\int_0^{\frac {\pi} 2}
(\sin^2(\theta+\theta_0) -\sin^2(\theta-\theta_0) ) \sin 2\theta d\theta \notag \\
&\gtrsim \frac 1 {1+\gamma_3}
\int_0^{\frac{\pi}2}
\sin^2(2\theta) d\theta \sin 2\theta_0 \notag \\
&\gtrsim \frac 1{1+\gamma_3} \sin 2\theta_0.
\end{align*}
Since $g^{\prime\prime}(0)=\lim_{\theta_0\to 0+}
\frac {g^{\prime}(\theta_0)} {\theta_0}$, the estimate for $g^{\prime\prime}(0)$
easily follows. 
\end{proof}

\begin{proof}[Proof of Lemma \ref{lemSep4_2a}]
Clearly
\begin{align*}
h_{\infty}(\rho, t) &= \mathbb E \sqrt{ \beta \rho^2 +\rho^2 X_t^2
+X_1^2} \sqrt{\beta\rho^2 +2X_1^2} \notag \\
& =\frac 1 {2\pi} \int_{\mathbb R^2} 
\sqrt{\beta \rho^2 +\rho^2(t x
+\sqrt{1-t^2} y)^2+ x^2}
\sqrt{\beta \rho^2 +2x^2}  e^{-\frac {x^2+y^2}2} dx dy.
\end{align*}
Since $\rho\sim 1$, it is easy to check that 
\begin{align*}
\sup_{|t|\le 1-\eta_0} (|\partial_t h_{\infty}(\rho, t)|
+|\partial_{tt} h_{\infty}(\rho,t)| + |\partial_{ttt} h_{\infty}(\rho, t) |\lesssim 1.
\end{align*}
To show the lower bound on $|\partial_t h_{\infty}(\rho, t)|$, observe that
$h_{\infty}(\rho, t)$ is an even function of $t$. Thus 
without loss of generality we assume $0\le t<1$. 
Now
let $t= \sin \theta_0$ with $\theta_0 \in[0,\frac {\pi}2)$.  
By using polar coordinates, we obtain
\begin{align*}
h_{\infty}(\rho, t)
&= \frac 1 {2\pi}
\int_0^{\infty} 
\int_{0}^{2\pi}
\sqrt{ \beta \rho^2 + \rho^2 \cos^2 (\theta-\theta_0)
+r^2 \sin^2 \theta}
\sqrt{\beta \rho^2 +2 r^2 \sin^2 \theta}
e^{-\frac {r^2} 2} r d\theta dr \notag \\
&= \frac 1 {\pi}
\int_0^{\infty} 
\int_{0}^{\pi}
\sqrt{ \beta \rho^2 + \rho^2 \cos^2 (\theta-\theta_0)
+r^2 \sin^2 \theta}
\sqrt{\beta \rho^2 +2 r^2 \sin^2 \theta}
e^{-\frac {r^2} 2} r d\theta dr.
\end{align*}
Observe that
\begin{align} \label{lemSep5_0a.0}
\partial_{\theta_0} \Bigl(  h_{\infty} (\rho, \sin \theta_0 ) \Bigr)
= (\partial_t h_{\infty} )(\rho, t)\Bigr|_{t=\sin \theta_0}  \cos \theta_0.
\end{align}
By Lemma \ref{lemSep5_0} (note that $\gamma_3=r^2$) and integrating in $r$, we then obtain
\begin{align*}
\partial_t h_{\infty}(\rho, t) \gtrsim t, \qquad\forall\, 0\le t<1.
\end{align*}
Finally to show that $\partial_{tt} h_{\infty}(\rho, t) \gtrsim 1$ for $|t|\ll 1$,
it suffices for us to show (since $ |\partial_{ttt} h_{\infty}(\rho ,t)| \lesssim 1$
for $|t|\ll 1$)
\begin{align*}
\partial_{tt} h_{\infty}(\rho, 0) \gtrsim 1.
\end{align*}
By using \eqref{lemSep5_0a.0}, we only need to check 
\begin{align*}
\partial_{\theta_0 \theta_0} \Bigl( h_{\infty} (\rho, 
\sin \theta_0 ) \Bigr) \Bigr|_{\theta_0=0} \gtrsim 1.
\end{align*}
This again follows from Lemma \ref{lemSep5_0}. 
\end{proof}

\begin{lem} \label{lemSep5b_1}
Suppose $\phi_1:\; \mathbb R\to \mathbb R$, $\phi_2:\; \mathbb R\to \mathbb R$
are $C^1$ functions such that
\begin{align*}
\max_{|z|\le L} ( |\phi_1(z)|+|\phi_1^{\prime}(z)| +
|\phi_2(z)|+|\phi_2^{\prime}(z)| ) \le C_{L,\phi_1,\phi_2},
\end{align*}
where $C_{L,\phi_1,\phi_2}>0$ is finite for each finite $L$. 

Suppose $0<c_1<c_2<\infty$ and $\phi_3:\, (\frac {c_1} 2, 2c_2) \to \mathbb R$ is a smooth 
function such that
\begin{align*}
\sup_{\frac {c_1}2 <|z|<2c_2} (|\phi_3(z)|+|\phi_3^{\prime}(z)|) \le C_{c_1,c_2,\phi_3},
\end{align*}
where $C_{c_1,c_2,\phi_3}>0$ depends only on $c_1$, $c_2$ and $\phi_3$.

Let $(d_{ij})_{1\le i\le 2, 1\le j\le 3}$ be given constants and consider
\begin{align*}
I (u,w,v)&= \frac 1m \sum_{j=1}^m
\phi_1\Bigl(
\frac {d_{11} |u|^2 + d_{12} (a_j \cdot e_1)^2 +d_{13} (a_j\cdot u)^2}
{\beta |u|^2 + (a_j\cdot u)^2 +(a_j\cdot e_1)^2} \Bigr)
\phi_2\Bigl(
\frac {d_{21} |u|^2 + d_{22} (a_j \cdot e_1)^2 +d_{23} (a_j\cdot u)^2}
{\beta |u|^2 + (a_j\cdot u)^2 +(a_j\cdot e_1)^2} \Bigr) \notag \\
&\qquad \qquad\cdot \phi_3(\|u\|_2)  (a_j\cdot w) (a_j\cdot v),
\quad u\in \mathbb R^n, \, w,v\in \mathbb S^{n-1}.
\end{align*}
Then for any $0<\epsilon \le 1$, if $m\gtrsim n$, then the following hold with
high probability:
\begin{align*}
|I(u,w,v)-\mathbb E I(u,w,v)|
\le \epsilon,
\qquad\forall\, w,v\in \mathbb S^{n-1}, \forall\, c_1\le \|u\|_2 \le c_2.
\end{align*}
\end{lem}
\begin{proof}
We first note that, by using a polarization argument and scaling (cf. the beginning part
of the proof of Lemma \ref{Sp1_1_0.1}), it suffices for us to prove the statement
for $I(u,w,w)$ uniformly in $w\in\mathbb S^{n-1}$ and $u\in \mathbb R^n$
with $c_1\le \|u\|_2 \le c_2$. 

Now let $\phi \in C_c^{\infty}(\mathbb R)$ such that $0\le \phi (x) \le 1$ for 
all $x$, $\phi(x)=1$ for $|x|\le 1$ and $\phi(x)=0$ for $|x|\ge 2$. Let
$\delta>0$ be a sufficiently small constant. The smallness of $\delta$ will
be specified momentarily. Then
\begin{align*}
& |I_1(u,w)|=\biggl| \frac 1m \sum_{j=1}^m
\phi_1\Bigl(
\frac {d_{11} |u|^2 + d_{12} (a_j \cdot e_1)^2 +d_{13} (a_j\cdot u)^2}
{\beta |u|^2 + (a_j\cdot u)^2 +(a_j\cdot e_1)^2} \Bigr)
\phi_2\Bigl(
\frac {d_{21} |u|^2 + d_{22} (a_j \cdot e_1)^2 +d_{23} (a_j\cdot u)^2}
{\beta |u|^2 + (a_j\cdot u)^2 +(a_j\cdot e_1)^2} \Bigr) \notag \\
&\qquad \qquad\cdot \phi_3(\|u\|_2)  (a_j\cdot w)^2
\phi\Bigl( \frac {a_j \cdot u} {\delta \langle a_j\cdot w \rangle } \Bigr)
\biggr| \notag \\
\lesssim &\;
\frac 1m \sum_{j=1}^m (a_j\cdot w)^2 \phi\Bigl( \frac {a_j \cdot u} {\delta \langle a_j\cdot w \rangle } \Bigr) \notag \\
\lesssim & \;
\frac 1m \sum_{j=1}^m (a_j\cdot w)^2 
\Bigl(1-\phi(2\delta^{\frac 18} (a_j\cdot w) ) \Bigr)
+ \frac 1m \sum_{j=1}^m  \delta^{-\frac 14} 
\phi(\frac {a_j\cdot u} {\delta \langle \delta^{-\frac 18} \rangle } ).
\end{align*}
The expectation of the above two terms are clearly small if we take $\delta>0$
sufficiently small. Moreover they are clearly OK for union bounds and can be made
small in high probability. Thus for 
sufficiently small $\delta$, if $m\gtrsim n$, then with high probability we have 
\begin{align*}
|I_1(u,w)-\mathbb E I_1(u,w)| \le &\; \frac {\epsilon}{3}, \qquad\forall\, w\in \mathbb S^{n-1},
\, \forall\, c_1\le \|u\|_2 \le c_2. 
\end{align*}
We now fix $\delta$ and deal with the main term
\begin{align*}
 &I_2 (u, w) \notag \\ 
=& \frac 1m \sum_{j=1}^m
\phi_1\Bigl(
\frac {d_{11} |u|^2 + d_{12} (a_j \cdot e_1)^2 +d_{13} (a_j\cdot u)^2}
{\beta |u|^2 + (a_j\cdot u)^2 +(a_j\cdot e_1)^2} \Bigr)
\phi_2\Bigl(
\frac {d_{21} |u|^2 + d_{22} (a_j \cdot e_1)^2 +d_{23} (a_j\cdot u)^2}
{\beta |u|^2 + (a_j\cdot u)^2 +(a_j\cdot e_1)^2} \Bigr) \notag \\
&\qquad \qquad\cdot \phi_3(\|u\|_2)  (a_j\cdot w)^2
\cdot\left( 1- \phi\Bigl( \frac {a_j \cdot u} {\delta \langle a_j\cdot w \rangle } \Bigr)
\right) \notag \\
=&\, \frac 1m \sum_{j=1}^m H(\|u\|_2, a_j\cdot u, a_j\cdot w, a_j\cdot e_1),
\end{align*}
where
\begin{align*}
&H(s,z,y,b) \notag \\
=&\;
\phi_1( \frac { d_{11} s^2 + d_{12} b^2 +d_{13} z^2}{\beta s^2
+z^2 +b^2} )
\phi_2( \frac { d_{21} s^2 + d_{22} b^2 +d_{23} z^2}{\beta s^2
+z^2 +b^2} ) \cdot \phi_3 (s) y^2 \left(
1- \phi( \frac {z} {\delta \langle y \rangle} ) \right).
\end{align*}
The main point is to check the union bounds. 
Note that $s=\|u\|_2\sim 1$. We have
\begin{align*}
& | \partial_s H(s,z,y,b)| \lesssim y^2; \\
& |\partial_z H(s,z,y,b)| \lesssim  |y| ;\\
&| \partial_y H(s,z,y,b)| \lesssim |y|.
\end{align*}
Thus for $c_1\le \|u\|_2, \|\tilde u\|_2 \le c_2$,
$w, \tilde w \in \mathbb S^{n-1}$, we have
\begin{align*}
 & \Bigl| H(\|u\|_2, a_j\cdot u, a_j\cdot w, a_j\cdot e_1)
 - H(\|\tilde u\|_2, a_j\cdot \tilde u, a_j\cdot \tilde w,
 a_j\cdot e_1) \Bigr| \notag \\
 \lesssim &\;
 \|u-\tilde u\|_2 ( |a_j\cdot w|^2)
 +|a_j\cdot (u-\tilde u)| |a_j\cdot w|
 + |a_j\cdot (w-\tilde w)| (|a_j\cdot w| +|a_j\cdot \tilde w|).
 \end{align*}
Clearly then the union bounds hold for $I_2$. Thus
for $m\gtrsim n$, with high probability it holds that
\begin{align*}
|I_2(u,w)-\mathbb E I_2(u,w)|
\le &\; \frac {\epsilon}{3}, \qquad\forall\, w\in \mathbb S^{n-1},
\, \forall\, c_1\le \|u\|_2 \le c_2. 
\end{align*}
The desired estimate for $I(u,w,w)$ then easily follows.
\end{proof}

Consider
\begin{align*}
h (\rho, t,e^{\perp})= \frac 1m \sum_{j=1}^m \sqrt{\beta \rho^2+ \rho^2 (a_j\cdot \hat u)^2
+X_j^2}
\cdot \sqrt{\beta \rho^2+2X_j^2},
\end{align*}
where
\begin{align*}
&X_j=a_j\cdot e_1, \quad
u=\rho \hat u, \quad 0<c_1\le \rho\le c_2<\infty;\\
& \hat u = te_1 + \sqrt{1-t^2} e^{\perp}, \quad |t| <1, \, e^{\perp}\cdot e_1=0,
e^{\perp} \in \mathbb S^{n-1}.
\end{align*}
Here we take $c_1>0$, $c_2>0$ as two fixed constants. The main point is that
$\rho \sim 1$. We consider $h$ in the regime 
\begin{align*}
|t| \le 1- \epsilon_0,
\end{align*}
where $0<\epsilon_0\ll 1$ is fixed.  
\begin{lem} \label{Sep5a_2}
Let $0<\epsilon_0\ll 1$ be fixed. 
For any $0<\epsilon\le 1$, if $m\gtrsim n$, then with high probability it holds that
\begin{align*}
&| \partial_t h -\mathbb E \partial_t h | 
+|\partial_{tt} h - \mathbb E \partial_{tt} h |\le \epsilon, \qquad\forall\, |t| \le 1-\epsilon_0, \,
 e^{\perp}\cdot e_1=0, e^{\perp} \in \mathbb S^{n-1}, c_1\le \rho \le c_2.
\end{align*}
\end{lem}
\begin{proof}[Proof of Lemma \ref{A30_2}]
Denote  $Y_j=a_j\cdot e^{\perp}$ and
\begin{align*}
Z_j =a_j\cdot \hat u = t X_j +\sqrt{1-t^2} Y_j. 
\end{align*}
Clearly
\begin{align*}
&\frac d{dt} Z_j= X_j - \frac {t} {\sqrt{1-t^2} } Y_j;  \notag \\
&\frac {d^2} {dt^2} Z_j = -(1-t^2)^{-\frac 32} Y_j.
\end{align*}
Using $Y_j =(1-t^2)^{-\frac 12} (Z_j-t X_j)$, we obtain
\begin{align*}
&\frac d{dt} Z_j = \frac 1 {1-t^2} X_j - \frac t {1-t^2} Z_j;\\
&\frac {d^2} {dt^2} Z_j=(1-t^2)^{-2} (tX_j-Z_j).
\end{align*}

Therefore
\begin{align*}
\partial_t h &= \frac 1 {1-t^2}
\cdot \frac 1 m\sum_{j=1}^m \sqrt{
\frac {\beta |u|^2+ 2X_j^2} {\beta |u|^2 +(a_j\cdot u)^2 +X_j^2}}
\|u\|_2 \cdot (a_j\cdot u) X_j  \notag \\
&\quad -\frac t {1-t^2} \cdot \frac 1m
\sum_{j=1}^m \sqrt{ \frac {\beta |u|^2+ 2X_j^2} {\beta |u|^2 +(a_j\cdot u)^2 +X_j^2}}
(a_j\cdot u)^2
\notag \\
&=:\frac 1{1-t^2} H_1  -\frac t{1-t^2}  H_2. 
\end{align*}

By Lemma \ref{lemSep5b_1}, it holds with high probability that
\begin{align*}
|H_1 -\mathbb E H_1 | +|H_2-\mathbb EH_2|\le  (1-\epsilon_0^2) \cdot \frac {\epsilon}3,
\qquad\forall\,  \hat u\in \mathbb S^{n-1}, c_1\le \|u\|_2\le c_2, |t|\le 1-\epsilon_0.
\end{align*}
The desired estimate for $\partial_t h$ then easily follows.

To compute $\partial_{tt}h$, we shall denote
\begin{align*}
&A_j= \beta \rho^2 +\rho^2 Z_j^2 +X_j^2=
\beta |u|^2 + (a_j\cdot u)^2 + X_j^2;\\
&B_j= \beta \rho^2 +2X_j^2 = \beta |u|^2 +2 X_j^2.
\end{align*}
Then
\begin{align*}
\partial_{tt} h 
& = - \frac 1m \sum_{j=1}^m A_j^{-\frac 32}
B_j^{\frac 12} (\rho^2 Z_j \frac d {dt} Z_j )^2 
+\frac 1m \sum_{j=1}^m A^{-\frac 12}_j
\sqrt{B_j}
\rho^2
\cdot \Bigl( (\frac d {dt}Z_j)^2
+ Z_j \frac {d^2}{dt^2} Z_j \Bigr) \notag \\
& = - (1-t^2)^{-2} \frac 1m \sum_{j=1}^m 
A_j^{-\frac 32} B_j^{\frac12} \|u\|_2^2 (a_j\cdot u)^2 X_j^2 \notag \\
& \quad +  {2t} (1-t^2)^{-2}
\frac 1m \sum_{j=1}^m A_j^{-\frac 32} B_j^{\frac12} \|u\|_2 (a_j\cdot u)^3 X_j
\notag \\
& \quad - t^2 (1-t^2)^{-2}
\frac 1m \sum_{j=1}^m A_j^{-\frac 32} B_j^{\frac 12}
(a_j\cdot u)^4\notag \\
& \quad 
+(1-t^2)^{-2} \frac 1m \sum_{j=1}^m A_j^{-\frac 12}
B_j^{\frac 12}  \|u\|_2^2 X_j^2 \notag \\
&\quad- t (1-t^2)^{-2} \frac 1m \sum_{j=1}^m
A_j^{-\frac 12} B_j^{\frac 12} \|u\|_2 X_j (a_j\cdot u) \notag \\
&\quad -(1-t^2)^{-1}
\frac 1m \sum_{j=1}^m A_j^{-\frac 12}
B_j^{\frac 12} (a_j\cdot u)^2. 
\end{align*}
It is then a bit tedious but not difficult to verify that the above terms  can be treated with the help of Lemma \ref{lemSep5b_1}. 
Thus with high probability it holds that
\begin{align*}
| \partial_{tt} h -\mathbb E \partial_{tt} h
| \le \frac {\epsilon}5, 
\qquad\forall\, \hat u \in \mathbb S^{n-1}, \, c_1\le \|u\|_2 \le c_2,
|t|\le 1-\epsilon_0.
\end{align*}
\end{proof}
\begin{lem} \label{leSep6_60}
Let $X\sim \mathcal N(0,1)$, $Y\sim \mathcal N (0,1)$ be independent. Define
\begin{align*}
&H(\rho,s) =  \mathbb E
\sqrt{\beta \rho^2 +\rho^2 (\sqrt{1-s^2} X+ s Y)^2+X^2 } \sqrt{\beta \rho^2 +2X^2};\\
&h (\rho, s)= \frac 12 (1+2\beta) \rho^2 -  H(\rho,s).
\end{align*}
Then it holds that
\begin{align*}
\sup_{|\rho-1|\ll 1, |s|\ll 1} \sum_{j=1}^3( | \partial^j H |
+|\partial^j h|) \lesssim 1.
\end{align*}
where $\partial= \partial_{\rho}$ or $\partial_s$.
\end{lem}

\begin{proof}
For $H(\rho,s)$, this is obvious since the integrand inside the expectation is smooth.
The estimate for $h(\rho,s)$ also follows easily.
\end{proof}
\begin{lem}[Calculation of $\partial^2 h$ at ($\rho=1$, $s=0$)] \label{leSep6_61}
Let
\begin{align*}
&H(\rho,s) =  \mathbb E
\sqrt{\beta\rho^2 +\rho^2(\sqrt{1-s^2} X+ s Y)^2+X^2} \sqrt{\beta \rho^2 +2X^2}; \\
&h (\rho, s)= \frac 12 (1+2\beta) \rho^2 - H(\rho,s).
\end{align*}
Then at $\rho=1$, $s=0$, we have
\begin{align*}
&(\partial_{\rho\rho} H)(1, 0) = \gamma_1>0, \quad 
(\partial_{\rho s} H)(\rho,0)=0, \, \forall\, \rho>0; \\
&(\partial_{ss} H)(1, 0)= -\gamma_2<0; \\
& (\partial_s h)(\rho,0)=0, \forall\, \rho>0, \quad (\partial_{\rho} h)(1,0)=0;\\
&(\partial_{\rho\rho} h)(1, 0) =\gamma_3>0, \quad
(\partial_{\rho s} h)(\rho,0)=0, \,\forall\, \rho>0 ;\\
&(\partial_{ss} h)(1, 0)= \gamma_4>0, 
\end{align*}
where $\gamma_i>0$, $i=1,\cdots, 4$ are constants depending on $\beta$. 
\end{lem}
\begin{proof}[Proof of Lemma \ref{leSep6_61}]
Firstly by using parity it is easy to check that $(\partial_{ s} H)(\rho,0)=0$ for any $\rho>0$.
It follows easily that $(\partial_{\rho s} h)(\rho,0)= (\partial_{\rho s} H)(\rho, 0)=0$
for any $\rho>0$. It is also easy to check that
\begin{align*}
(\partial_{\rho} H)(1,0)
&= \partial_{\rho} \mathbb E (\sqrt{ \beta \rho^2 +(\rho^2+1) X^2}
\sqrt{\beta \rho^2 +2X^2} ) \Bigr|_{\rho=1} \notag \\
& = \mathbb E (2\beta +X^2) = 2\beta +1.
\end{align*}
Clearly $(\partial_{\rho} h)(1,0)=0$. One should note that we can also deduce this
directly (and easily) from the fact that the original loss function attains a minimum at
$u=e_1$. 

\underline{\texttt{Calculation of $\partial_{ss} H$}}. By a tedious computation, we 
have
\begin{align*}
2\pi (\partial_{ss} H)(1,0)
&= 2\pi \partial_{ss} \Bigl( \mathbb E \sqrt{\beta +X^2+(\sqrt{1-s^2} X+s Y)^2}
\sqrt{\beta +2X^2} \Bigr) \biggr|_{s=0} \notag \\
& = \int_{\mathbb R^2}
\frac {-2x^4 -\beta x^2 + (\beta +x^2) y^2} {\beta +2x^2}
e^{-\frac {x^2+y^2} 2} dx dy \notag \\
& =\int_{\mathbb R} \frac{ -2 x^4 -\beta x^2 +\beta+x^2} {\beta +2x^2} e^{-\frac {x^2} 2} dx \notag \\
&=\int_{\mathbb R}
\Bigl( -x^2 \frac{\beta+1+2x^2}{\beta+2x^2} + 1 \Bigr)e^{-\frac {x^2} 2} dx \notag \\
&=- \int_{\mathbb R} \frac {x^2} {\beta+2x^2} e^{-\frac {x^2} 2} dx <0.
\end{align*}

\underline{\texttt{Calculation of $\partial_{\rho\rho} H$}}. By a tedious computation, we 
have
\begin{align*}
 (\partial_{\rho\rho} H)(1,0)
&= \partial_{\rho\rho}
\Bigl(
\mathbb E \sqrt{\beta \rho^2 + (\rho^2+1) X^2} \sqrt{\beta \rho^2 +2X^2} \Bigr)
\biggr|_{\rho=1} \notag \\
&=\frac 1 {\sqrt{2\pi}} \int_{\mathbb R}
\frac{2\beta^2 +5\beta x^2 +x^4} {\beta +2x^2} e^{-\frac {x^2} 2} dx \notag \\
&=\frac 1 {\sqrt{2\pi}} \int_{\mathbb R}
(2\beta+ \frac {\beta+x^2}{\beta+2x^2} x^2 ) e^{-\frac {x^2}2} dx.
\end{align*}
It follows that
\begin{align*}
(\partial_{\rho\rho} h)(1,0) & = 1+2\beta - (\partial_{\rho\rho} H)(1,0) \notag \\
&=-\frac 1{\sqrt{2\pi}} \int_{\mathbb R}
\frac{x^4}{\beta+2x^2} e^{-\frac {x^2}2} dx <0.
\end{align*}

\end{proof}
\begin{lem} \label{Sep6Sp1_1}
Let $0<c_1<c_2<\infty $ be fixed. Consider
\begin{align*}
f_0 (u) = -\frac 1 m \sum_{k=1}^m \sqrt{\beta |u|^2 + (a_k \cdot u)^2+(a_k\cdot e_1)^2}
\sqrt{\beta |u|^2 +2 (a_k\cdot e_1)^2}.
\end{align*}
For any $0<\epsilon\le 1$, if $m\gtrsim n$, 
then it holds with high probability  that 
\begin{align*}
&\Bigl| \sum_{i,j=1}^n \xi_i \xi_j 
(\partial_{ij} f_0)(u) - 
\sum_{i,j=1}^n \xi_i \xi_j 
\mathbb E(\partial_{ij} f_0)(u) \Bigr| \le \epsilon, 
\qquad \forall\, \xi \in \mathbb S^{n-1}, \quad \forall\, c_1\le \| u\|_2 \le c_2.
\end{align*}
\end{lem}
\begin{proof}[Proof of Lemma \ref{Sep6Sp1_1}]
To simplify the notation, write $a_k$ as $a$, and denote
\begin{align*}
&A= \beta |u|^2 + (a\cdot u)^2+(a\cdot e_1)^2, \qquad B= \beta |u|^2 +2 (a\cdot e_1)^2; \\
& \partial_i A= 2 \beta u_i + 2 (a\cdot u)a_i, \qquad \partial_{ij}A
= 2\beta \delta_{ij} + 2 a_i a_j; \\
& \partial_i B= 2 \beta u_i, \qquad \partial_{ij} B= 2\beta \delta_{ij}.
\end{align*}

We need to compute $\partial_{ij} \tilde F$ for 
\begin{align*}
\tilde F =  A^{\frac 12} B^{\frac 12}.
\end{align*}
Clearly
\begin{align*}
&\partial_i \tilde F = \frac 12 A^{-\frac 12} \partial_i A B^{\frac 12}
+ \frac 12 B^{-\frac 12} \partial_i B A^{\frac 12};\\
&\partial_{ij} \tilde F
= -\frac 14 A^{-\frac 32} \partial_i A \partial_j A B^{\frac 12}  + \frac 12 A^{-\frac 12} \partial_{ij} A B^{\frac 12} + \frac 12 A^{-\frac 12} \partial_i A \frac 12 B^{-\frac 12}
\partial_j B \notag \\
& \qquad -\frac 14 B^{-\frac 32} \partial_j B \partial_i B A^{\frac 12} + \frac 12 B^{-\frac 12} \partial_{ij} B A^{\frac 12} + \frac 14 B^{-\frac 12} A^{-\frac 12} \partial_i B \partial_j A. 
\end{align*}
We then have
\begin{align}
  & \sum_{i,j=1}^n \xi_i \xi_j (\partial_{ij} f_0)(u) \notag \\
=&\; \frac 14 \cdot  \frac 1 m \sum_{k=1}^m  A_k^{-\frac 32}
|\xi \cdot \nabla A_k|^2  B_k^{\frac 12} \label{Sep6Sp1_1.1} \\
& \; + \frac 14 \cdot \frac 1m \sum_{k=1}^m
B_k^{-\frac 32} |\xi \cdot \nabla B_k|^2 A_k^{\frac 12} \label{Sep6Sp1_1.2} \\
& \; - \frac 12\cdot \frac 1m \sum_{k=1}^m A_k^{-\frac 12}
\langle \xi, (\nabla^2 A_k) \xi \rangle B_k^{\frac 12} \label{Sep6Sp1_1.3} \\
& \; - \frac 12 \cdot \frac 1m \sum_{k=1}^m 
A_k^{-\frac 12} B_k^{-\frac 12} (\nabla A_k \cdot \xi) (\nabla B_k \cdot \xi) 
\label{Sep6Sp1_1.4} \\
& \; - \frac 12 \cdot \frac 1m \sum_{k=1}^m B_k^{-\frac 12} A_k^{\frac 12}
\langle \xi, (\nabla^2 B_k) \xi \rangle, \label{Sep6Sp1_1.5} 
\end{align}  
where $A_k=\beta |u|^2 + (a_k\cdot u)^2+(a_k\cdot e_1)^2$, $B_k=
\beta |u|^2 +2(a_k\cdot e_1)^2$, and we have denoted 
\begin{align*}
\langle \xi, (\nabla^2 A_k) \xi \rangle = \sum_{i,j=1}^n \xi_i \xi_j 
\partial_{ij} A_k.
\end{align*}
Thanks to the strong damping provided by $A_k$, it is tedious but not difficult to check that the
terms \eqref{Sep6Sp1_1.1}, 
\eqref{Sep6Sp1_1.3}, \eqref{Sep6Sp1_1.4} can be easily controlled with the 
help of Lemma \ref{lemSep5b_1}.  The term
\eqref{Sep6Sp1_1.5} can be estimated in a similar way as 
in the estimate of \eqref{Sp1_1.5} in the proof of Lemma \ref{Sp1_1} (note
that this is done in high probability therein!).  The term
\eqref{Sep6Sp1_1.2} is also easy to handle. We omit further details.
\end{proof}

\end{document}